\documentclass[openright,a4paper,american,11pt]{amsart}

\usepackage[latin1]{inputenc}

\usepackage{amsmath}
\usepackage{amssymb}
\usepackage{hyperref}
\usepackage{babel}

\usepackage{amsfonts}
\input xy
\xyoption{all}
\usepackage[dvips]{graphicx}
\usepackage{boxedminipage}
\usepackage{color}

\newcommand{\ZZ}{{\mathbb Z}}

\newcommand{\PP}{{\mathbb P}}

\newcommand{\CC}{{\mathbb C}}
\newcommand{\RR}{{\mathbb R}}

\newcommand{\TT}{{\mathbb T}}

\newcommand{\D}{{\mathcal D}}
\newcommand{\C}{{\mathcal C}}
\newcommand{\E}{{\mathcal E}}
\newcommand{\F}{{\mathcal F}}
\newcommand{\M}{{\mathcal M}}
\newcommand{\N}{{\mathcal N}}
\newcommand{\I}{{\mathfrak I}}
\renewcommand{\H}{{\mathcal H}}
\renewcommand{\S}{{\mathcal S}}
\renewcommand{\P}{{\mathcal P}}
\renewcommand{\L}{{\mathcal L}}

\renewcommand{\Im}{\text{Im\ }}
\newcommand{\Ve}{\text{Vert}}
\newcommand{\Ed}{\text{Edge}}
\renewcommand{\div}{\text{div}}
\newcommand{\val}{\text{val}}
\newcommand{\oC}{\overset{\circ}C}
\newcommand{\Q}{{\mathcal Q}}
\newcommand{\Sh}{\text{Sh}}

\newtheorem{thm}{Theorem}[section]
\newtheorem{defi}[thm]{Definition}
\newtheorem{defn}[thm]{Definition}
\newtheorem{prop}[thm]{Proposition}
\newtheorem{lemma}[thm]{Lemma}
\newtheorem{cor}[thm]{Corollary}

          {\theoremstyle{definition}
\newtheorem{rem}[thm]{Remark}}
          {\theoremstyle{definition}
\newtheorem{exa}[thm]{Example}}

\begin{document}
\title{Genus 0 characteristic numbers of the tropical projective plane}
\author{Benoît Bertrand}
\address{Benoît Bertrand, Université Paul Sabatier, Institut Mathématiques de Toulouse, 
118 route de Narbonne, F-31062 Toulouse Cedex 9, France}
\email{benoit.bertrand@math.univ-toulouse.fr}
\author{Erwan Brugallé}
\address{Erwan Brugallé, Université Pierre et Marie Curie,  Paris 6, 4 place Jussieu, 75 005 Paris, France} 
\email{brugalle@math.jussieu.fr}
\author{Grigory Mikhalkin}
\address{Grigory Mikhalkin, Section de mathématiques
Université de Genève,
Villa Battelle, 7 route de Drize,
1227 Carouge, Suisse} 
\email{grigory.mikhalkin@unige.ch}
\date{\today}



\begin{abstract}
Finding the so-called characteristic numbers of the complex projective plane ${\mathbb C}P^2$ is a classical
problem of enumerative geometry posed by Zeuthen more than a century ago.
For a given $d$ and $g$ one has to find the number of degree $d$ genus $g$ curves that pass through a certain generic
configuration of points and at the same time are tangent to a certain generic configuration of lines. The total number
of points and lines in these two configurations is $3d-1+g$ so that the answer is a finite integer number.

In this paper we translate this classical problem to the corresponding enumerative problem of tropical geometry
in the case when $g=0$. Namely, we show that the tropical problem is well-posed and establish a special case
of the correspondence theorem that ensures that the corresponding tropical and classical numbers coincide.
Then we use the floor diagram calculus to reduce the problem to pure
combinatorics. As a consequence,  we 
express 
genus 0 characteristic
numbers of $\CC P^2$ in terms of open Hurwitz numbers.\thanks{Research is supported in part by the project TROPGEO
  of the 
  European Research Council. 
Also B.B. is partially supported by the ANR-09-BLAN-0039-01, E.B. is
partially supported by the ANR-09-BLAN-0039-01 and ANR-09-JCJC-0097-01,  and
G.M. is partially supported by the Swiss National Science Foundation 
grants n° 125070 and 126817.}  
 \end{abstract}

\maketitle
\tableofcontents

\section{Introduction}\label{intro}

Characteristic numbers were considered by nineteenth century geometers
among which S. Maillard (\cite{Mai71}) who computed them in degree
$3$, H. Zeuthen (\cite{Zeu73}) who did third and fourth degree cases,
and H. Schubert (\cite{Schu79}). Modern mathematicians confirmed and
extended their predecessor's results thanks in particular to
intersection theory. P. Aluffi 
computed, for instance, all characteristic numbers for plane cubics and
some of them for  plane quartics (see \cite{Alu89},
\cite{Alu90}, \cite{Al1} and 
\cite{Alu92}), and R. Vakil completed to confirm  Zeuthen's computation
of all characteristic numbers of plane quartics in \cite{Vak99}.
 R. Pandharipande computed characteristic numbers  in
 the rational case in
\cite{Pan99},  Vakil achieved the genus $1$ case in  \cite{Vak01},
 and T. Graber, J. Kock and
Pandharipande computed  genus  $2$ characteristic numbers of plane
curves in
 \cite{GKP02}. 
 In this
article, we 
provide a new insight to 
the genus $0$ case and give new formulas via its tropical counterpart and
by reducing the problem to combinatorics of floor diagrams.

Before going
into the details of the novelties of this paper, let us recall the
classical problem of computing characteristic numbers of $\CC P^2$.
\vspace{1ex}

Let $d$, $g$ and $k$ be non negative integer numbers such that $g\le
\frac{(d-1)(d-2)}{2}$ and $k\le 3d+g-1$ and $d_1,\ldots, d_{3d+g-1-k}$ be
positive integer numbers. For any configurations $\P=\{p_1,\ldots, p_k
\}$ of $k$ points in $\CC P^2$, and $\L =\{L_1,\ldots, L_{3d+g-1-k}
\}$ of $3d+g-1-k$ complex non-singular algebraic curves in $\CC P^2$
such that $L_i$ has degree $d_i$, we consider the set $\S(d,g,\P,\L)$
of holomorphic maps $f:C\to\CC P^2$ from an irreducible non-singular  
complex algebraic curve of genus $g$, 
passing through all points $p_i\in\P$, tangent to all curves
$L_i\in\L$, and such that $f(C)$ has degree $d$ in $\CC P^2$.

If the constraints $\P$ and $\L$ are chosen generically, then the set
$\S(d,g,\P,\L)$ is finite, and the characteristic number
$N_{d,g}(k;d_1,\ldots ,d_{3d+g-1-k})$ is defined as
$$N_{d,g}(k;d_1,\ldots
,d_{3d+g-1-k})=\sum_{f\in\S(d,g,\P,\L)}\frac{1}{|Aut(f)|}$$
where $Aut(f)$ is the group of automorphisms of the map $f:C\to\CC
P^2$, i.e. isomorphisms $\Phi:C\to C$ such that $f\circ\Phi=f$. 
It depends only on $d$,
$g$, $k$ and $d_1,\ldots, d_{3d+g-1-k}$ (see for example \cite{Vak01}). In this text, we will use the
shorter notation $N_{d,g}(k;d_1^{i_1},\ldots ,d_{l}^{i_l})$ which
indicates that the integer $d_j$ is chosen $i_j$ times.
 Let us describe the characteristic numbers in some special instances.

 \begin{exa}
The number $N_{d,g}(3d-1+g)$ is the usual Gromov-Witten invariant of
degree $d$ and genus $g$ of $\CC P^2$.
\end{exa}

\begin{exa}
The numbers $N_{2,0}(5)$, $N_{2,0}(4;1)$, and $N_{2,0}(3;1^2)$ are
easy to compute by hand, and thanks to  projective duality we have

$$N_{2,0}(k;1^{5-k}) =N_{2,0}(5-k;1^k) = 2^{k} \ \ for \ 0\le k\le
2$$
 \end{exa}
\begin{exa}
All characteristic numbers $N_{3,0}(k;1^{8-k})$ for
 rational cubic curves have been computed by Zeuthen
(\cite{Zeu2}) and 
 confirmed by 
 Aluffi (\cite{Al1}). We sum up part of their results in the following table.
$$\begin{array}{c|c|c|c|c|c|c|c|c|c}
k & 8 & 7 & 6 & 5 & 4 & 3 & 2 & 1 & 0
\\ \hline N_{3,0}(k;1^{8-k}) & 12 & 36 & 100& 240 &480 & 712 & 756&
600 & 400
\end{array} $$
\end{exa}

\begin{exa}
The number $N_{2,0}(0;2^5)$ has been  computed 
independently by Chasles (\cite{Chas1}) and De Jonquiere. More than
one century later, Ronga, Tognoli, and Vust showed in \cite{RoToVu}
that it is possible to choose 5 real conics in such a way that all
conics tangent to these 5 conics are real. See also \cite{Sot2} and
\cite{Ghy1} for a historical account and digression on this subject.
See also Example \ref{ex real conic} for a tropical version of the
arguments from \cite{RoToVu}. 
We list below the numbers $N_{2,0}(k;2^{5-k})$.
$$  N_{2,0}(4;2) = 6 \ \ \ \  \ \ N_{2,0}(3;2^2) =36
\ \ \ \ \ \ N_{2,0}(2;2^3) = 184$$
$$ N_{2,0}(1;2^4) = 816 \ \  \ \ \ \ \ \ N_{2,0}(0;2^5) =3264
 $$
 \end{exa}

More generally,  the characteristic numbers
$N_{d,g}(k;1^{3d-1+g-k})$  of $\CC P^2$ determine
all the numbers $N_{d,g}(k;d_1,\ldots,d_{3d-1+g-k})$.
Indeed, by degenerating 
the non-singular curve 
$L_{d_{3d-1+g-k}}$
 to the union of two
non-singular curves of lower degrees
intersecting transversely, we obtain the following formula (see for
example {\cite[Theorem 8]{RoToVu}})
\begin{equation}\label{break constraint}
\begin{array}{ccl}
N_{d,g}(k;d_1,\ldots,d_{3d-1+g-k})&=&
2d'_{3d-1+g-k}d''_{3d-1+g-k}N_{d,g}(k+1;d_1,\ldots,d_{3d-2+g-k})\\ &&
  \\ &&
+ N_{d,g}(k;d_1,\ldots,d'_{3d-1+g-k}) 
+ N_{d,g}(k;d_1,\ldots,d''_{3d-1+g-k})
 \end{array}
\end{equation}
where $d_{3d-1+g-k}=d'_{3d-1+g-k}+ d''_{3d-1+g-k}$.

\vspace{3ex}
This paper contains three main contributions.
First we identify tropical
tangencies between two tropical morphisms. 
Then we 
deduce from the location of tropical tangencies a
 Correspondence Theorem which  allows one 
to compute characteristic numbers of $\CC P^2$
in genus 0. Finally, using the floor decomposition technique 
we provide a new insight on these characteristic
numbers and their relation to Hurwitz numbers.

Thanks to a series of Correspondence Theorems initiated in \cite{Mik1}
(see also \cite{NishinouSiebert}, \cite{Sh3}, \cite{Nishinou},
\cite{Tyomkin}, \cite{Mik08}), 
tropical geometry turned out to be a powerful tool to solve
enumerative problems in complex and real algebraic geometry.
 However, until now,  
 Correspondence Theorems  dealt with  problems only involving 
simple incidence conditions, i.e. 
with enumeration of curves intersecting transversally a given
set of constraints. 
Correspondence Theorems  \ref{Corres} and \ref{Corres2} in this paper
are the first 
ones concerning  
plane curves satisfying  tangency conditions to a given set of curves.
In the case of simple incidences, the (finitely many)
tropical curves arising as the limit
of amoebas of the  enumerated complex curves 
 could be 
 identified
  considering the tropical limit of embedded complex curves. This is no
 longer enough to identify the tropical limit of
 tangent curves, and we first refine previous studies
 by
considering the tropical limit of a
 family of holomorphic maps to a given projective space. This gives
 rise to the notion of \textit{phase-tropical curves} and morphisms. 
For a treatment of 
 phase-tropical geometry more general than the one in this paper, 
we refer to \cite{Mik08}. 
Using tropical morphisms and their approximation by holomorphic maps, we
 identify tropical tangencies between tropical morphisms, and prove
 Theorems  \ref{Corres} and  \ref{Corres2}.
Note that A.~Dickenstein and L.~Tabera also studied in \cite{Dic1}
 tropical tangencies but in a slightly different context, i.e. 
tangencies between tropical cycles instead of
tropical morphisms.

As in the works cited above, 
Theorems  \ref{Corres} and  \ref{Corres2}
 allow us to solve our enumerative problem 
by exhibiting
some special configurations of constraints for which we can actually
find all complex curves satisfying our conditions. In particular we 
obtain more 
 information 
 than using only the intersection theoretical
approach from complex geometry.
As a consequence, when
all constraints are real we are also able to identify all
 real curves  matching our constraints. 
This is the starting observation in applications of tropical geometry in
real algebraic geometry, which 
 already turned out to be fruitful
 (see for example \cite{Mik1}, \cite{IKS1}, \cite{Ber3},
\cite{Ber4}, \cite{Br14}). In this paper 
we just provide 
few examples 
of
such 
applications to real enumerative geometry
in Section \ref{sec:real enumerative}.
However there is no doubt that Theorems \ref{Corres} and
\ref{Corres2}
 should
lead to further results in real enumerative geometry.

The next step after proving our Correspondence Theorems is to use generalized  
\textit{floor diagrams} to reduce
the computations to pure 
 combinatorics. 
In the particular case where only incidence conditions are considered they are equivalent to
those defined in \cite{Br7} (see also \cite{Br6b}), and used later in several
contexts (see for example \cite{FM}, \cite{ArdBlo}, \cite{BGM}, 
\cite{Ber4}, \cite{Br8} , \cite{Br14}). 
Note that  floor
decomposition 
technique  
 has strong connections with the Caporaso and Harris method
(see \cite{CapHar1})  extended later by Vakil (see \cite{Vak1}), and
with the neck-stretching method in symplectic field theory (see
\cite{EGH}, \cite{IP00}).
This method allows one to
solve an enumerative problem by induction on the dimension of the
ambient space, i.e. to reduce enumerative problems in $\CC P^n$ to
enumerative problems in $\CC P^{n-1}$. In the present paper, the
enumerative problem we are concerned with is to count curves which interpolate
a given configuration of points and are tangent to a given set of
curves. 
On the
level 
 of maps, tangency conditions are naturally interpreted as 
ramification conditions. In particular, the 1-dimensional analogues of
characteristic numbers are \textit{Hurwitz numbers}, which are the
number of  maps from a (non-fixed) genus $g$ curve to a (fixed)
genus $g_0$ curves with a fixed ramification profile at some fixed points.
Hence using floor diagrams, we express
characteristic numbers of $\CC P^2$  in terms of
Hurwitz numbers.
Surprisingly, other 
$1$-dimensional 
 enumerative invariants also appear in this
expression. These are the so-called
\textit{open Hurwitz numbers},
 a slight generalization of Hurwitz
numbers defined and computed in \cite{Br13}. 
Computations of characteristic numbers of $\CC P^2$ performed in 
\cite{Pan99},    \cite{Vak01}, and
 \cite{GKP02},
were
done by induction 
 on 
 the degree of the enumerated curves. 
 To our
 knowledge, this is the first time that characteristic numbers are
expressed in  
 terms 
 of their analogue in dimension 1, i.e. in terms of
(open)
 Hurwitz numbers.

\vspace{2ex}
Here is the plan of the paper. In Section \ref{sec2} we review standard
definitions we need from tropical geometry. In Section \ref{sec3} we define
tropical tangencies and state our Correspondence Theorems.
Even though we can reduce all degrees $d_i$ of the constraints to 1 by Formula
(\ref{break constraint}),
we still leave $d_1,\ldots,d_{3d-1-k}$ 
in the statement of Correspondence Theorem  \ref{Corres},
in view of possible application to 
real geometry (see e.g. Example  \ref{ex real conic}).
Section \ref{generic conf} 
is devoted to the proof of technical lemmas on generic
configurations of constraints.
The multiplicity of a tropical curve in Theorems \ref{Corres} and 
\ref{Corres2} are defined
by a determinant, 
 and we give in Section \ref{practical} a practical
way of computing this determinant which will be used in Section
\ref{floor dec}. 
We introduce
in Section \ref{sec:phase}
 the notion of phase-tropical curves and morphisms, and the
tropical limit of a family of holomorphic maps, 
and we prove Theorems \ref{Corres} and 
\ref{Corres2} in Section \ref{sec:proof corres}.
 In Section \ref{floor dec}, floor diagrams are introduced, formulas for characteristic numbers in which they appear are proved, and examples given.

\vspace{2ex}
We end this 
 introduction
 elaborating on possible natural generalizations of the
techniques presented in this paper. All definitions, statements, and
proofs should generalize 
with no difficulty 
to the case of rational
curves in $\CC P^n$ intersecting cycles and tangent to non-singular
hypersurfaces. The resulting floor diagrams would then be a generalization of
those defined in \cite{Br7} and \cite{Br6}. The enumeration of plane
curves with higher order tangency conditions to other curves should
also be doable in principle using our methods.
 This would first necessitate to identify
tropicalizations of higher order tangencies between curves,
generalizing the simple tangency case treated in 
Section \ref{sec:trop pretang}. However this identification might be
intricate, and will certainly lead to much more different cases than for simple
tangencies (third order tangencies to a line are dealt with in \cite{BruLop2012}).
In  turn, the use of tropical techniques in the computation of higher
genus
characteristic
numbers requires some substantial additional work. The
main difficulty is that superabundancy appears for positive genus:
some combinatorial types appearing as solution of the enumerative
problem might be of actual dimension strictly bigger than the expected
one (see Remark \ref{rem:coarse def}). Hence before enumerating tropical curves, in addition to the
balancing condition one has first to
understand extra necessary conditions for a tropical morphism to be
the tropical limit of a family of algebraic maps 
(see Section \ref{sec:phase}). 
Using the techniques developed in
\cite{Br12}, we succeeded to compute genus 1 characteristic numbers 
 of $\CC P^2$. These results will appear in a separate paper. Also
for a small number of tangency constraints, it is possible to find a
configuration of constraints for which no superabundant curve shows
up. In this case Theorem \ref{Corres} applies,  the proof only requiring
minor adjustments.

\vspace{2ex}
\textbf{Acknowledgement:} We would like to thank the anonymous referee for
 numerous remarks which helped us to improve our original text.

\section{Tropical curves and morphisms}\label{sec2}
In this section, we define abstract tropical curves, their morphisms
to $\RR^n$, and tropical cycles in $\RR^2$. 

\subsection{Tropical curves}\label{defi trop curve}

Given a finite graph $C$ (i.e. $C$ has a finite number of edges and
vertices) we denote 
by $\Ve(C)$ the set of its vertices,
 by $\Ve^0(C)$ the set of its vertices which are not $1$-valent, 
 and by 
$\Ed(C)$ the set of its edges. By definition, the valency of a vertex
 $v\in\Ve(C)$, denoted by $\val(v)$,
 is the number of edges in $\Ed(C)$ adjacent to $v$. 
Throughout the text
we will  identify a graph and any of its topological realization.
Next definition is taken from \cite{Br13}. Tropical curves with
boundary will be needed in Section \ref{floor dec}. 

\begin{defi}\label{def:trop curve}
An \emph{irreducible  tropical curve} $C$ with boundary is a finite compact connected 
  graph
with $\Ed(C)\ne\emptyset$, together with  a set of 1-valent vertices  $\Ve^\infty(C)$ of $C$,
 such that 
\begin{itemize}
\item  $C\setminus  \Ve^\infty(C)$ is equipped with 
 a complete inner metric;

\item the vertices of $\Ve^0(C)$ have non-negative integer weights,
 i.e. 
$C$ is equipped  with a map
$$\begin{array}{ccc}
\Ve^0(C) &\longrightarrow & \ZZ_{\ge 0}
\\ v&\longmapsto & g_v
\end{array} $$

\item any $2$-valent vertex $v$ of $C$ satisfies $g_v\ge 1$.

\end{itemize} 

If  $v$ is an element of $\Ve^0(C)$,
 the integer $g_v$ is called the genus of $v$.
The genus of $C$  is defined as  
$$g(C)=b_1(C) + \sum_{v\in\Ve^0(C)}g_v$$
where $b_1(C)$ denotes the first Betti number of $C$. 

\vspace{1ex}
An element of $\Ve^\infty(C)$ is called a \emph{leaf} of $C$, and its
adjacent edge is 
called an \emph{end} of $C$. 
A 1-valent vertex of $C$  not in $\Ve^\infty(C)$ is
called
a \emph{boundary component} of $C$.
\end{defi}

Definition \ref{def:trop curve} might appear very general, whereas we
essentially deal  with \emph{rational}
tropical curves, i.e.  of genus 0, in the rest of the paper. 
The reason for us to give such a general definition is that
 the general
framework for correspondence theorems between complex algebraic and
tropical curves we set up in 
Section \ref{sec:trop limit} 
is valid for
tropical curves as in Definition \ref{def:trop curve}, and may be used
in  future  correspondences between complex algebraic
and tropical curves of any genus.

It follows immediately from Definition \ref{def:trop curve} that
the leaves of $C$ are 
at the infinite distance from all the other
points of $C$. 
We denote 
 by $\partial C$ the set of the boundary components of $C$, by
 $\Ed^\infty(C)$ the set of ends of $C$,
and by 
 $\Ed^{0}(C)$ the set of its edges which are not adjacent to a
1-valent vertex.

A
\textit{punctured tropical curve} $C'$ is given by $C\setminus\P$
where $C$ is a tropical curve, and $\P$ is a subset of
$\Ve^\infty(C)$. Note that $C'$ has a tropical structure inherited
from $C$.
Elements of $\P$ are called  \textit{punctures}.
An end of $C'$ is said to be \textit{open} if it is adjacent to a
puncture, and \textit{closed} otherwise. 

\begin{exa}
In Figure~\ref{Fig:tropical-curves} we depict some examples of the
simplest rational tropical curves. Boundary components and vertices of
$\Ve^0(C)$ are depicted as  black dots and vertices of
$\Ve^\infty(C)$ are omitted so that an edge not ending at a black
vertex is of infinite length and that no difference is made on the
picture between punctured and non-punctured tropical curves. 
\end{exa}

\begin{figure}[h]
\centering
\begin{tabular}{cccccc} 
\begin{minipage}{0.15\linewidth}
\centering
\includegraphics[height=4cm, angle=90]{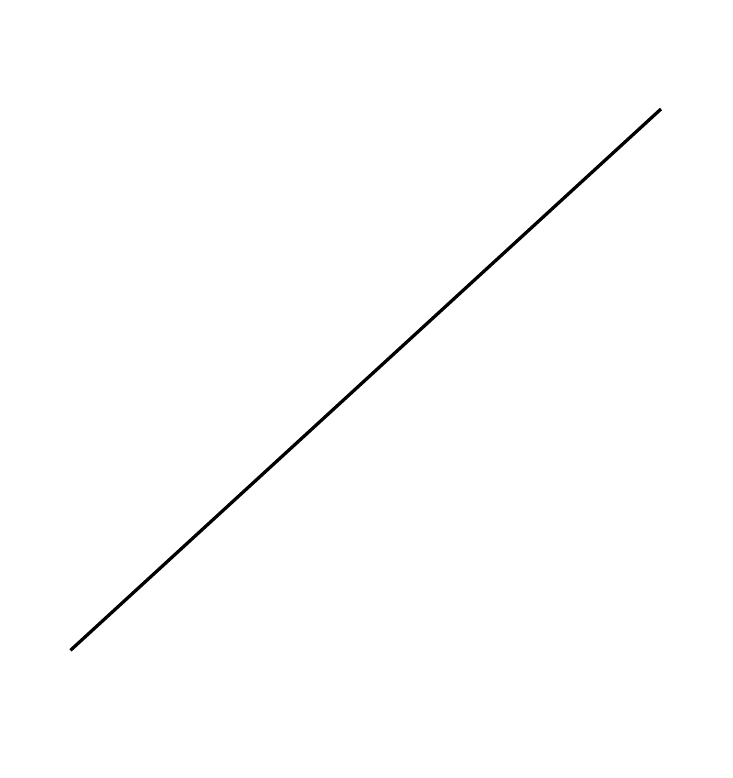}
\end{minipage} &\hspace{3ex}
& \begin{minipage}{0.15\linewidth}
\centering
\includegraphics[height=4cm, angle=0]{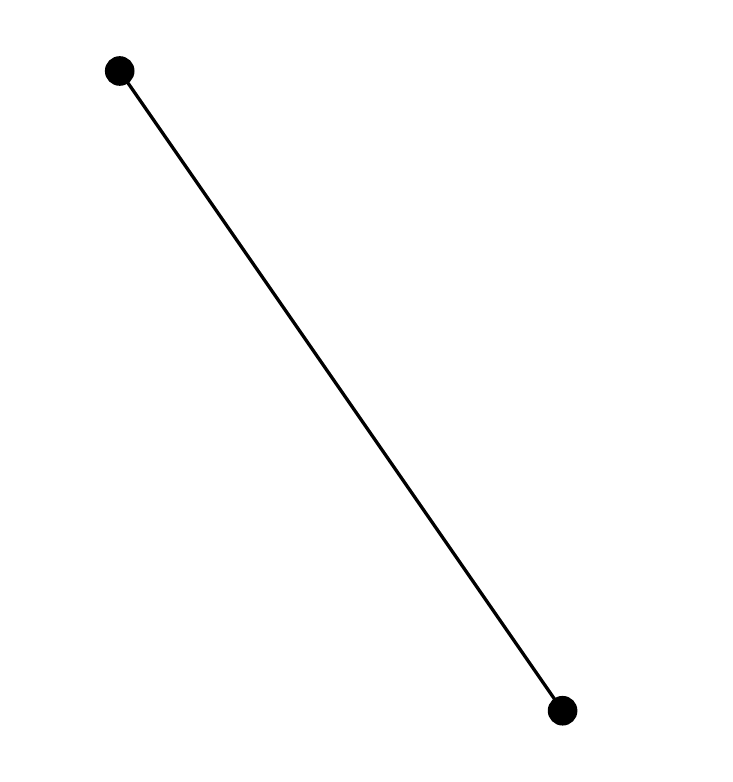}
 \end{minipage}
 &
\begin{minipage}{0.15\linewidth}
\centering
\includegraphics[height=4cm, angle=0]{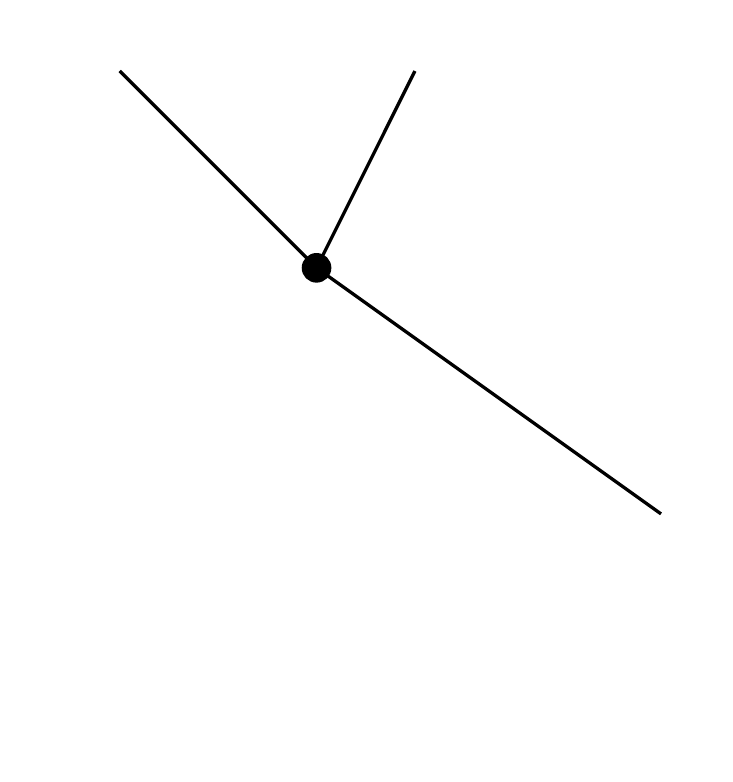}
\end{minipage}
&
\begin{minipage}{0.15\linewidth}
\centering
\includegraphics[height=4cm, angle=0]{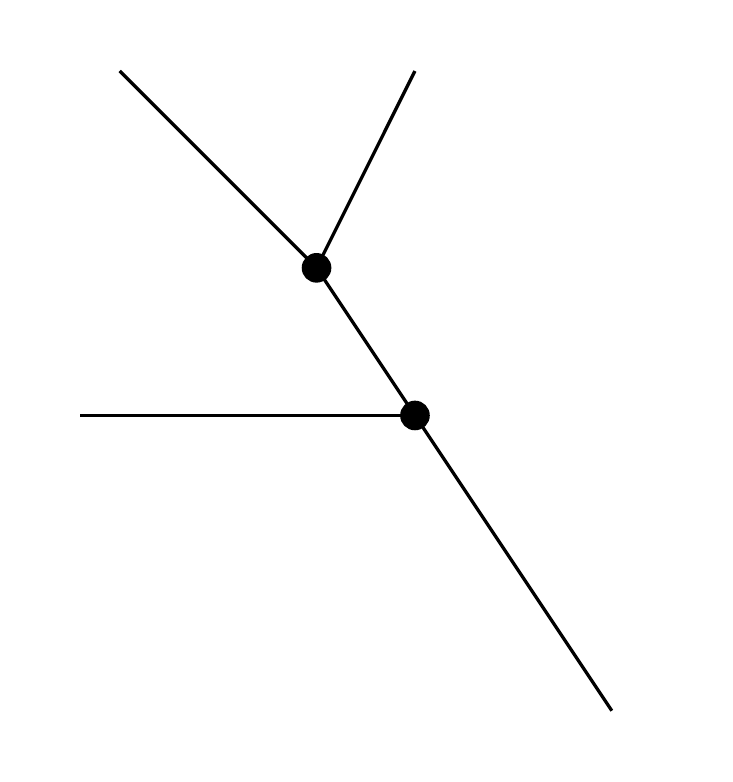}
\end{minipage}
&
\begin{minipage}{0.15\linewidth}
\centering
\includegraphics[height=4cm, angle=0]{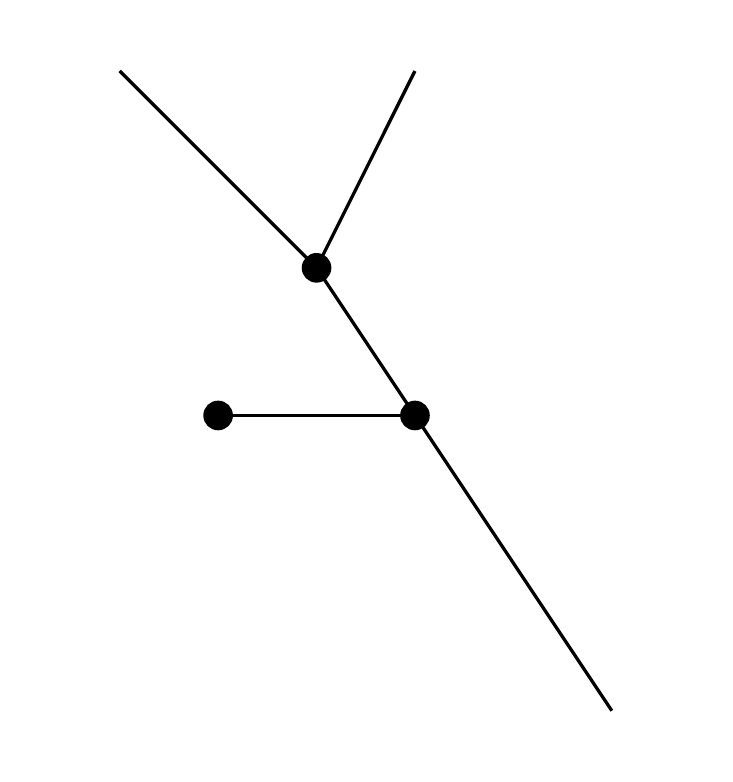}
\end{minipage}
\\ a)& &b) &c) & d) & e) 
\end{tabular}

\caption{A few rational tropical curves}
\label{Fig:tropical-curves}
\end{figure}

Given $C$  a tropical curve, and $p$  a point on $C$ which is not a
1-valent vertex, we can construct a new tropical curve $\widetilde C_p$ by 
attaching to $C$ a closed end $e_p$ (i.e. an end with a leaf point at
infinity) at $p$,
and by setting $g_p=0$ if $p\notin\Ve^0(C)$.
 The natural map
$\pi:\widetilde C_p\to C$ which contracts the edge $e_p$ to $p$
is called a \textit{tropical modification}. If we denote by $v$ the
1-valent vertex of $\widetilde C_p$ adjacent to $e_p$, the restriction
of the map $\pi$ to the punctured tropical curve $\widetilde
C_p\setminus\{v\}$ is called an \textit{open tropical modification}.

\begin{exa}
The closed curve depicted in Figure~\ref{Fig:tropical-curves}d can
be obtained by a tropical modification from the curve of
Figure~\ref{Fig:tropical-curves}c which in turn can be obtained
modifying the infinite closed segment of
Figure~\ref{Fig:tropical-curves}a.  More generally, every rational
tropical curve without boundary can be obtained from the infinite
closed segment (Figure~\ref{Fig:tropical-curves}a) by a finite sequence of
tropical modifications.
\end{exa}

\subsection{Tropical morphisms}
Given $e$  an edge of a tropical curve $C$, we choose a point $p$ in
the interior of $e$ and
a unit vector $u_e$
of the tangent line to $C$ at $p$ 
(recall that $C$ is equipped with a metric). 
Of course,  the vector $u_e$
depends on the choice of $p$ and is well-defined only up to
multiplication by -1, but this will not
matter in the following. We will sometimes need $u_e$ to have a
prescribed direction, and we will then specify this direction.
The
standard inclusion of $\ZZ^n$ in $\RR^n$ induces a standard
inclusion of $\ZZ^n$ in the tangent space of $\RR^n$  at any point of
$\RR^n$. 
A vector in $\ZZ^n$ is said to be \emph{primitive} if the
greatest common divisor of its
coordinates equals $1$.

\begin{defi}\label{def:trop morphism}
Let $C$ be a punctured
  tropical curve.
A
continuous map  $f : C\to \RR^n$ is a
  \emph{tropical morphism} if
\begin{itemize}

\item for any edge $e$ of $C$, the restriction $f_{|e}$ is a
  smooth map with $df(u_e)=w_{f,e}u_{f,e}$ where
 $u_{f,e}\in\ZZ^n$ is a
  primitive vector, and $w_{f,e}$ is a non-negative integer;

\item for any vertex $v$ in $\Ve^0(C)$ whose adjacent  edges are
  $e_1,\ldots,e_k$,
  one has the balancing condition
$$\sum_{i=1}^k  w_{f,e_i}u_{f,e_i}=0 $$
where $u_{e_i}$ is chosen so that it points away from $v$.
\end{itemize}

\end{defi}

The integer $w_{f,e}$
 is
called the \textit{weight of the edge $e$ with respect to $f$}.  When
no confusion is possible, we simply  speak 
about the weight of an edge, without referring to the morphism $f$.
If 
$w_{f,e}=0$, we say that the morphism $f$ \textit{contracts} the edge
$e$. 
The morphism $f$ is called
\textit{minimal} if $f$ does not contract
any 
edge.
If the morphism $f$ is proper then an open end of 
$C$ has to have a non-compact image, while a closed end has to be
contracted. Hence
if a proper tropical morphism $f:C\to\RR^n$ is minimal, then
$\Ve^\infty(C)=\emptyset$. 
The morphism $f$ is called an
\textit{immersion} if it is a topological immersion, i.e. if $f$ is a
local homeomorphism on its image. 

\begin{rem}\label{rem:coarse def}
Definition \ref{def:trop morphism} is a rather coarse
definition of a tropical morhism when $C$ has positive genus. Indeed,
in contrast to the case of rational curves, one can easily construct
tropical morphisms from a positive genus tropical curve which are
superabundant, i.e. whose space of deformation has a strictly bigger
dimension that the expected one (see {\cite[Section 2]{Mik1}}). In
particular, when the corresponding situation in classical geometry is
regular (i.e. with no superabundancy phenomenon) as in
the case of projective plane curves, such a superabundant tropical
morphism is unlikely to be presented as the tropical limit of a family
of holomorphic maps (see Section \ref{sec:trop limit} for the
definition of the tropical limit).

One may refine Definition \ref{def:trop morphism}, still using pure
combinatoric, to get rid of many of these superabundant tropical
morphisms.
 One of these possible refinements, explored in  \cite{Br12}, 
is to require in addition that the map
$f:C\to\RR^n$ should be \emph{modifiable}: the map $f$ has to be
liftable to any sequence of tropical modifications of $\RR^n$ with smooth center
(see \cite{Mik3}). This definition of modifiable tropical
morphisms relies on the more general definition 
 of a tropical morphism $f:C\to X$
where $X$ is any non-singular tropical variety. In addition to the
balancing condition,  a tropical
morphism $f:C\to X$ has to satisfy some combinatorial conditions coming from
complex algebraic geometry, such as the Riemann-Hurwitz condition at
points of $C$ mapped to the 1-dimensional squeleton of $X$ (see
\cite{Br13} for the case when $X$ is a tropical curve).

However since this paper is about enumeration of rational curves,
 Definition  \ref{def:trop morphism}, even if coarse, is sufficient for our
purposes here.
\end{rem}

\begin{figure}[h]
\centering
\begin{tabular}{ccc}
\includegraphics[height=5cm, angle=0]{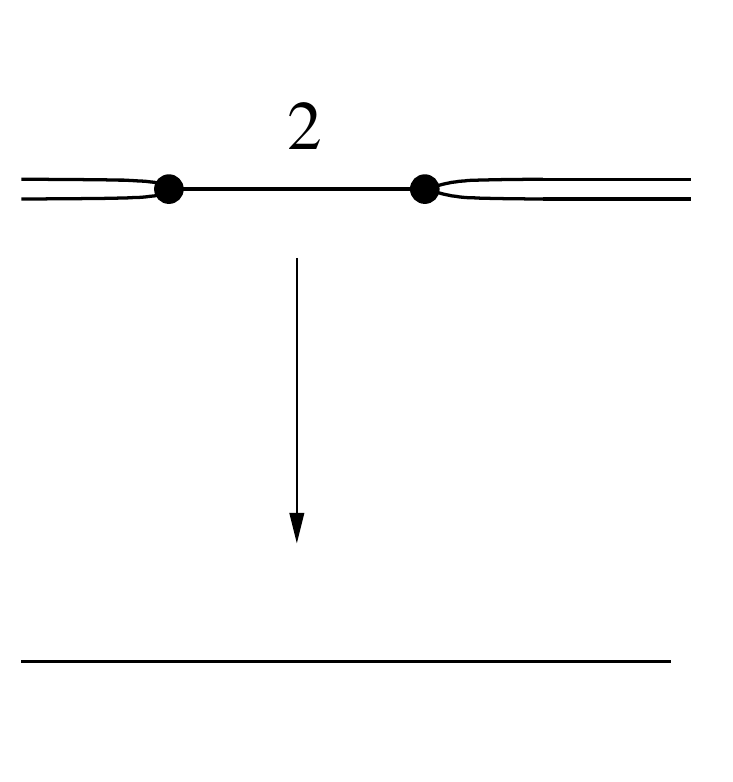}&  \hspace{10ex} & 
\includegraphics[height=5cm, angle=0]{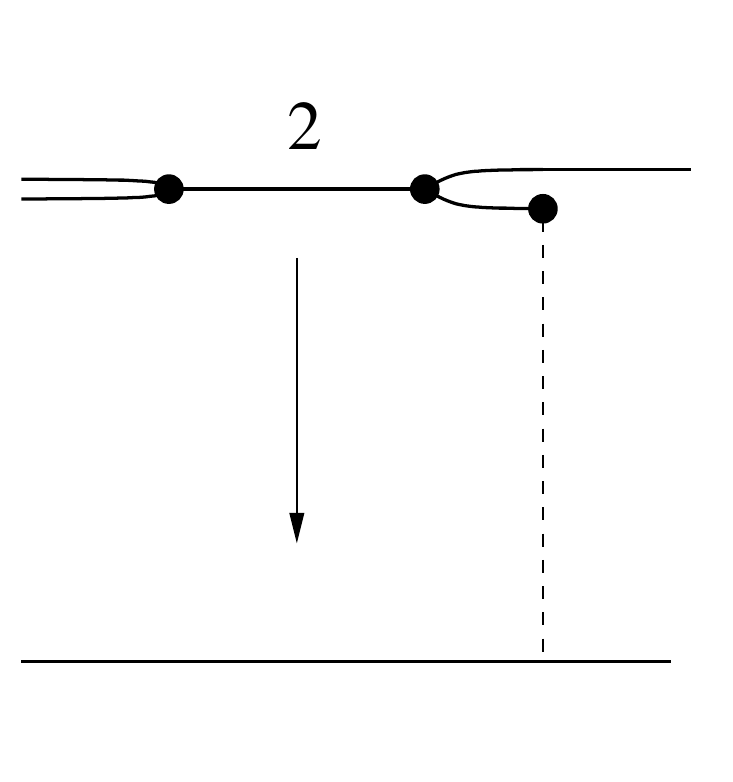}
\\ a)& & b) 
\end{tabular}

\caption{Example of tropical morphisms 
 $f:C\to \RR$}
\label{Fig:morphism-dim1}
\end{figure}

\begin{figure}[h]
\centering
\includegraphics[height=9cm, angle=0]{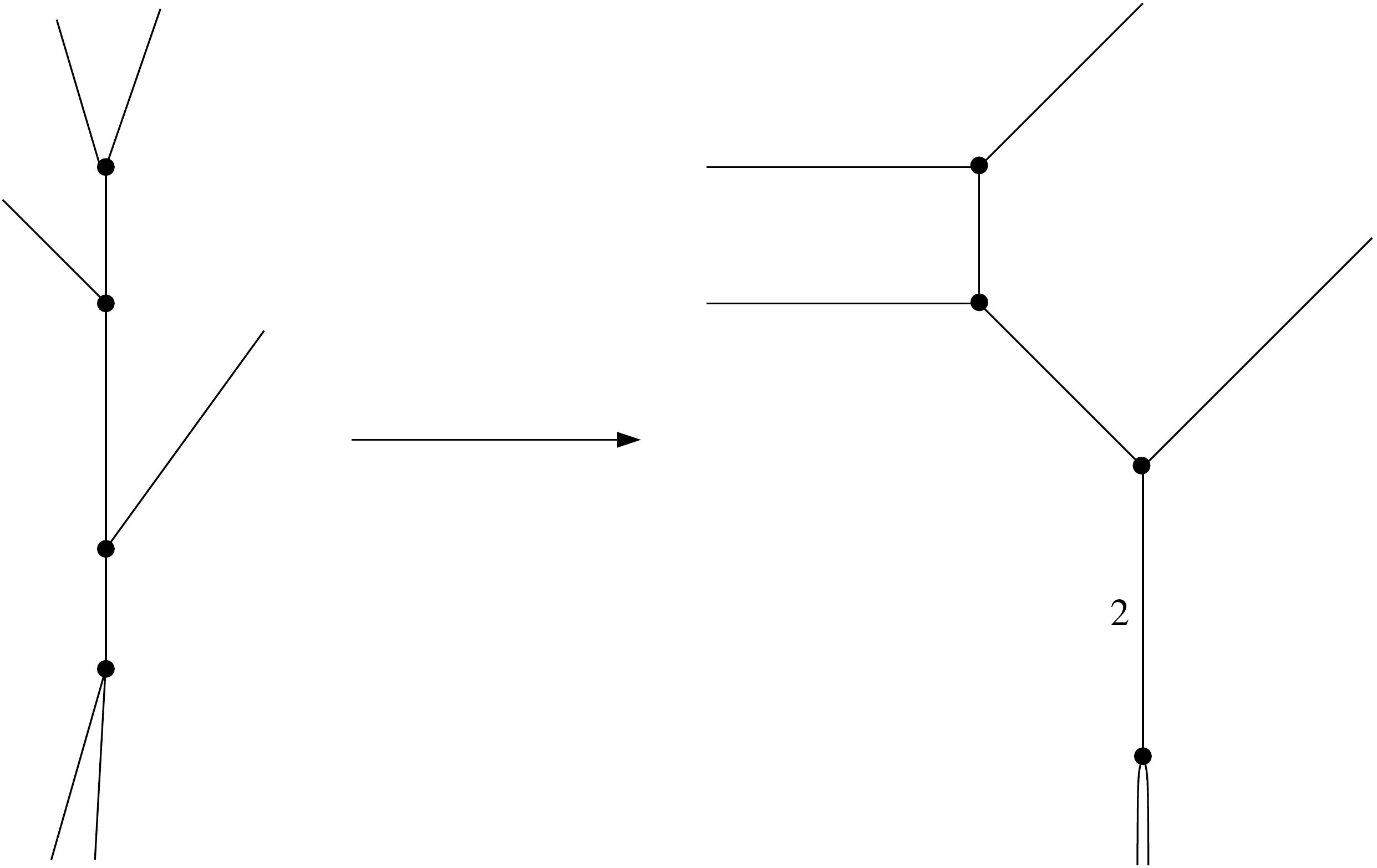}
\caption{Example of a morphism: a plane conic}
\label{Fig:morphism-dim2}
\end{figure}

\begin{exa}\label{Example2}
We represent morphisms from a curve to $\RR$ as in
Figure~\ref{Fig:morphism-dim1}. The weights of the edges with respect
to the morphism label the edge if they exceed
$1$. Figure~\ref{Fig:morphism-dim1}a represents a double cover of
the real line by a tropical curve without boundary, and
Figure~\ref{Fig:morphism-dim1}b represents
an open cover of $\RR$ (see Section \ref{open Hurwitz}) by a
tropical curve
with one boundary component.
\end{exa}

\begin{exa}
In contrast
to the one dimensional case, for morphisms from a curve to $\RR^2$, we
label the {\bf image} of an edge with the corresponding weight as it
often allows us to omit the source of the morphism which is then
implicit.  In Figure~\ref{Fig:morphism-dim2} we depicted a plane
conic $C$, which is the image in $\RR^2$ of a morphism from a trivalent
punctured curve with 
four vertices in $\Ve^0(C)$.
\end{exa}

Two tropical morphisms
$f_1:C_1\to\RR^n$ and $f_2:C_2\to\RR^n$ are said to be
\textit{isomorphic} if there exists
 an 
isometry
$\phi:C_1\to
C_2$ such that $f_1=f_2\circ\phi$
and $g_{\phi (v)}=g_v$ for any $v\in\Ve^0(C_1)$.
In this text, we consider  
 tropical curves and tropical morphisms up to
isomorphism. 

A less restrictive equivalence relation is the one associated to
 \textit{combinatorial types}.
Two  tropical morphisms $f_1:C_1\to  \RR^n$ and
$f_2:C_2\to  \RR^n$    
are said to have the same combinatorial type if 
there exists a
homeomorphism of graphs $\phi:C_1\to C_2$  inducing two bijections 
$\Ve^{\infty}(C_1)\to \Ve^{\infty}(C_2)$ and 
$\partial C_1\to \partial C_2$, and such that
for all edges $e$ of $C_1$ and all vertices 
 $v\in\Ve^{0}(C)$ 
 one has
 $$ w_{f_1,e}=w_{f_2,\phi(e)},\ \ 
u_{f_1,e}=u_{f_2,\phi(e)},\ \  \mbox{and} \ \  g_{\phi (v)}=g_v.$$

Given a combinatorial type $\alpha$ of  tropical morphisms,
we denote by $\M_\alpha$ the space of 
all such  tropical morphisms having this combinatorial
type. Given $f\in\M_\alpha$, we say that $\M_\alpha$ is the
 \textit{rigid deformation space}
of $f$.

\begin{lemma}[Mikhalkin, {\cite[Proposition 2.14]{Mik1}}]\label{exp dim 1}
Let $\alpha$ be a combinatorial type of  
 tropical morphisms $f:C\to
\RR^n$ where $C$ is a rational tropical curve with $\partial C=\emptyset$.
Then the space  $\M_\alpha$ is  an open convex polyhedron in
the vector space $\RR^{n+|\Ed^0(\alpha)|}$, and 
$$  \dim \M_\alpha =   |\Ed^\infty(\alpha)|  +n-3 
 -
\sum_{v\in \Ve^0(\alpha)}(val(v)-3).$$

\end{lemma}
\begin{proof}
We remind the proof in order to fix notations we will need later in
Section \ref{practical}. If $\Ve^0(\alpha)\ne\emptyset$, 
choose a root vertex $v_1$ of $\alpha$, and an ordering
$e_1,\ldots,e_{|\Ed^0(\alpha)|}$ of the edges in
$\Ed^0(\alpha)$. Given $f:C\to\RR^n$ in $\M_\alpha$, we write
$f(v_1)=(x_1,\ldots ,x_n)\in\RR^n$, and
we denote by
$l_i\in\RR^* $ the length of the edge $e_i\in\Ed^0(C)$. 
 Then 
$$\M_\alpha=\{(x_1,\ldots, x_n,l_1,\ldots,l_{|\Ed^0(\alpha)|}\quad |\quad 
l_1,\ldots, \ldots,l_{|\Ed^0(\alpha)|}>0\}=\RR^n\times
\RR_{>0}^{|\Ed^0(\alpha)|}.$$ 
If $\Ve^0(C)=\emptyset$, then $\M_\alpha=\RR^n/\RR u_{f,e}$, where $e$
is the only edge of $\alpha$.
\end{proof}
 Other choices of $v_1$ and of the ordering of elements of $\Ed^0(\alpha)$
provide other coordinates on $\M_\alpha$, and the change of
coordinates is given by an element
of $GL_{n+|\Ed^0(\alpha)|}(\ZZ)$.

\begin{exa}\label{coord syst ex}
In the simplest case when $\alpha$ is the combinatorial type of
tropical morphisms $f:C\to\RR^2$ with $\Ve^0(C)=\{v\}$,
the space $\M_\alpha$ is
$\RR^2$ and the coordinates are 
given by $f(v)$.

If $\alpha$ is the morphism depicted in Figure~\ref{Fig:morphism-dim2},
$\M_\alpha$ is the unbounded polyhedron $\RR^2\times\RR_{> 0}^3$
with coordinates $\left\{x_1, x_2, l_1, l_2, l_3\right\}$ where $x_1$
and $x_2$ are the coordinates of the image of the lowest vertex and
$l_1, l_2$ and $l_3$ are the lengths of the bounded edges at the
source ordered from bottom to top.
\end{exa}

\subsection{Tropical cycles}
Here we fix notations concerning standard facts in tropical
geometry. We refer for example to \cite{Mik3}, \cite{St2},  \cite{Mik9}, or
\cite{BIT}
 for more details.

An \textit{effective tropical 1-cycle} in $\RR^2$
 is the  tropical divisor defined by some
tropical polynomial $P(x,y)$, whose Newton
polygon is  denoted by $\Delta(P)$.
Given a positive integer $d$, 
we denote by $T_d$ the triangle in
$\RR^2$ with vertices $(0,0)$, $(0,d)$ and $(d,0)$. 

The number 1 in 1-cycle stands for dimension 1 as such cycle is necessarily a graph.
Since any smaller-dimensional cycles in $\RR^2$ can only be 0-cycles
which
are just linear combinations of points in $\RR^2$,
 in this paper we will be just saying
``cycles in $\RR^2$" for 1-cycles.

\begin{defi}\label{degree-dim2}
An effective tropical cycle in $\RR^2$ defined by a tropical
polynomial $P(x,y)$ is said to have
 \emph{degree $d\ge 1$} if 
 $\Delta(P)\subset T_d$ and  $\Delta(P)\nsubseteq T_{d-1}$.
\end{defi}

Recall that  a tropical polynomial
induces a subdivision of its Newton polygon.

\begin{defi}
An effective tropical cycle in $\RR^2$ of degree $d$ defined by a tropical
polynomial $P(x,y)$ is said to be \emph{non-singular} if
$\Delta(P)=T_d$ and
the subdivision of $\Delta(P)$ induced by $P(x,y)$ is primitive,
i.e. contains only triangles of Euclidean area $\frac{1}{2}$. 

The tropical cycle is said to be \emph{simple} if this subdivision of $\Delta(P)$
contains only triangles and parallelograms.
\end{defi}

Two effective tropical cycles are said to be of the same
\textit{combinatorial type} if they have the same Newton polygon, and
the same dual subdivision. As in the case of tropical morphisms, we
denote by $\M_\alpha$ the set of all effective tropical cycles of a given
combinatorial type $\alpha$. 
\begin{lemma}\label{dim comb type ns}
If $\alpha$ is a combinatorial type of 
non-singular 
effective tropical
cycles, then $\M_\alpha$ is an open convex polyhedron in
$\RR^{|\Delta(\alpha)\cap \ZZ^2|-1}$.
\end{lemma}

If $f : C \to \RR^2$ is   a non-constant
 tropical morphism, then the image $f(C)$ is a balanced polyhedral
 graph in  $\RR^2$, and hence is defined by some
tropical polynomial $P(x,y)$ (see {\cite[Proposition 2.4]{Mik12}}). The Newton polygon of
the morphism 
$f:C\to\RR^2$, denoted by $\Delta(f)$, is defined as the Newton
polygon of $P(x,y)$.
The polygon $\Delta(f)$ is
well-defined up to 
translation by a vector in $\ZZ^2$, in particular the \textit{degree
  of $f$} is well-defined.

The \textit{genus} of a tropical cycle $A$ is the smallest genus of a tropical
immersion $f : C \to \RR^2$ such that $f(C)=A$. Note that if $A$ is
simple and $f:C\to\RR^2$ is such a tropical immersion of minimal
genus then the tropical curve $C$ contains only 3-valent vertices,
which are called \textit{the vertices} of $A$.

We can refine the notion of Newton polygon of a tropical cycle $A$ (or
of an algebraic curve in $(\CC^*)^2$) by
encoding how $A$ intersect the toric divisors at infinity. Namely, we
define the \textit{Newton fan} of $A$ to be the multiset of vectors 
$w_eu_e$ where $e$ runs over all unbounded edges of $A$, $w_e$ is the
weight of $e$ (i.e. the integer length of its dual edge), and $u_e$ is
the primitive integer vector of $e$ pointing to the unbounded
direction of $e$. The definition in the case 
of algebraic curves in $(\CC^*)^2$ works similarly.

\begin{rem}
Note that our simple effective cycles are in 1-1 correspondence with 
3-valent immersed tropical curves such that their self-intersections
are isolated points that come as intersection of different edges 
at their interior points. We call them tropical cycles (instead of calling them
tropical curves) to emphasize their role as constraints.
\end{rem}

\section{Tangencies}\label{sec3}
\subsection{Tropical pretangencies in $\RR^2$}\label{sec:trop pretang}
Our definition of tropical pretangencies 
between tropical morphisms
is motivated by the study of tropical
tangencies in \cite{Br12}, to which we refer  for more details.
We also refer to \cite{Dic1} for the notion of tropical tangencies between two
tropical cycles in $\RR^2$.
Let $f:C\to\RR^2$ be a tropical morphism, and $L$ be a 
simple
tropical cycle in $\RR^2$.

\begin{defi}\label{defi tang}
The  tropical morphism $f$
 is said to be \emph{pretangent} to 
 $L$ if
there exists a connected component $E$ of the set theoretic intersection
of $f(C)$ and $L$ which contains  either  a
vertex of $L$ or the image of a vertex of
$C$.

The set $E\subset\RR^2$ is called a \emph{pretangency set} of $f$
and $L$. A connected component of $f^{-1}(E)\subset C$  is
called a \emph{pretangency component} of $f$ with $L$ 
 if $E$ contains either a vertex of $C$ or a point
$p$ such that $f(p)$ is a vertex of $L$.
\end{defi}

\begin{figure}[h]
\centering
\begin{tabular}{ccccc}
\includegraphics[height=4.5cm, angle=0]{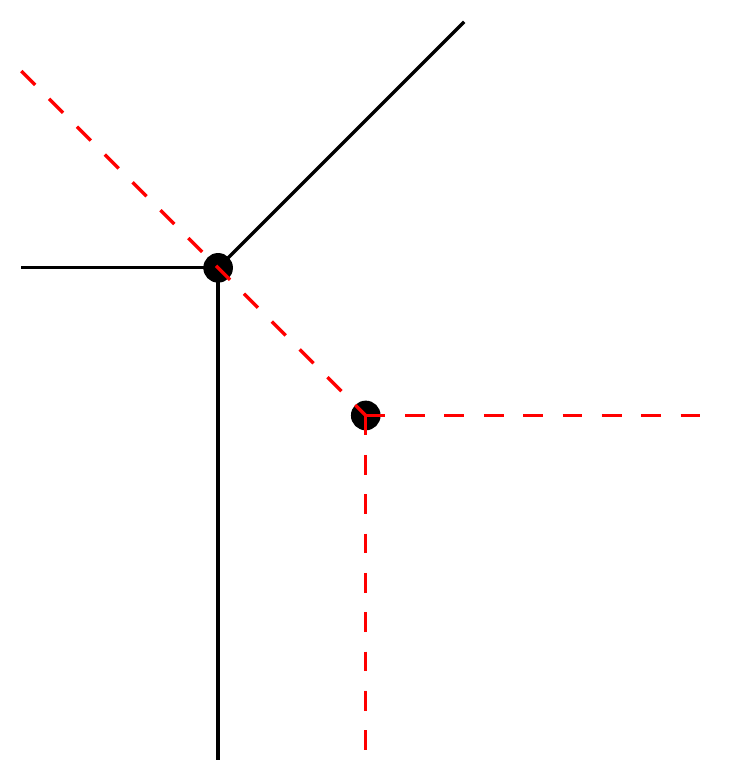}&  \hspace{10ex} & 
\includegraphics[height=4.5cm, angle=0]{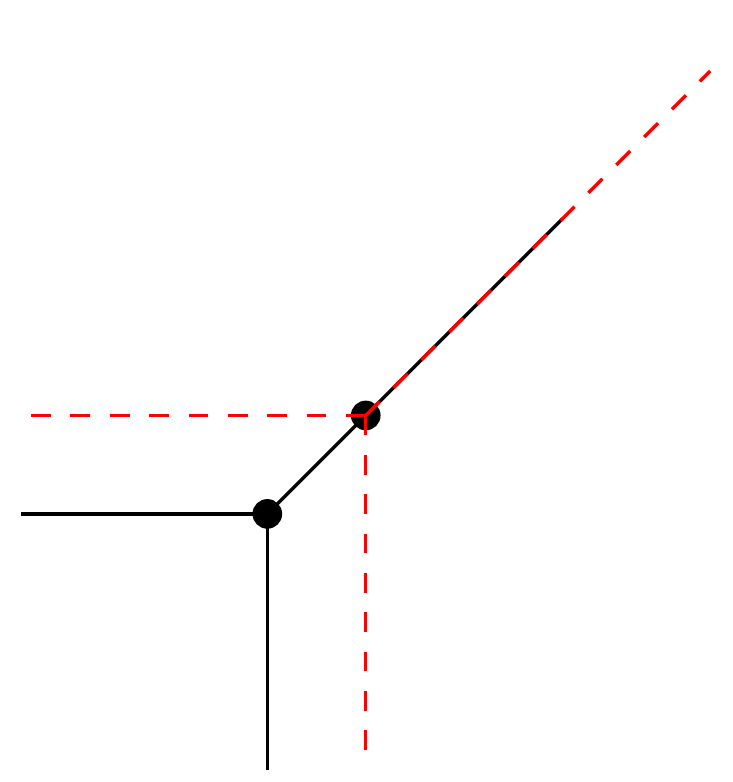}&  \hspace{10ex} & 
\includegraphics[height=4.5cm, angle=0]{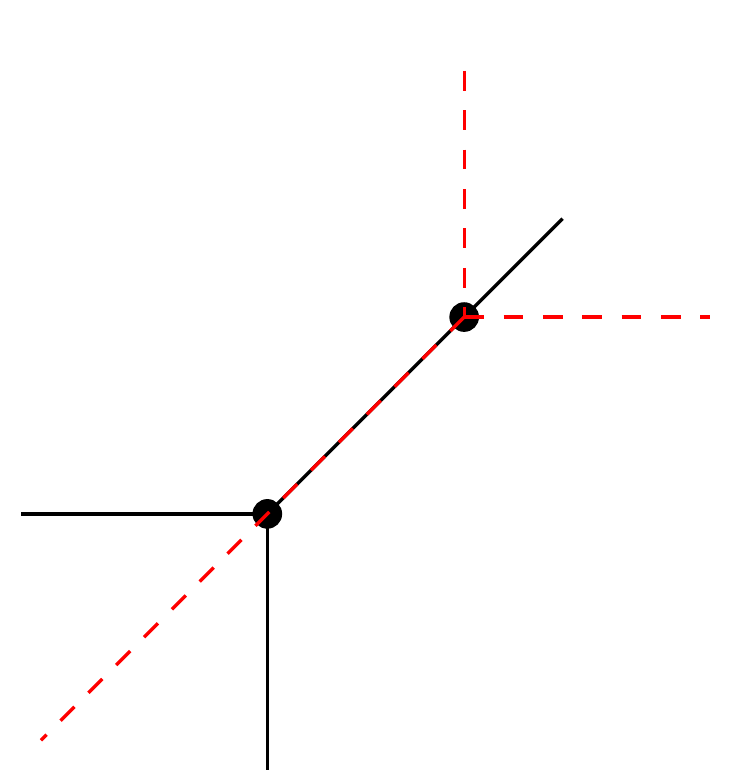}
\\ a)& & b)& & c) \\
\end{tabular}

\begin{tabular}{ccccc}
& 
\includegraphics[height=4.5cm, angle=0]{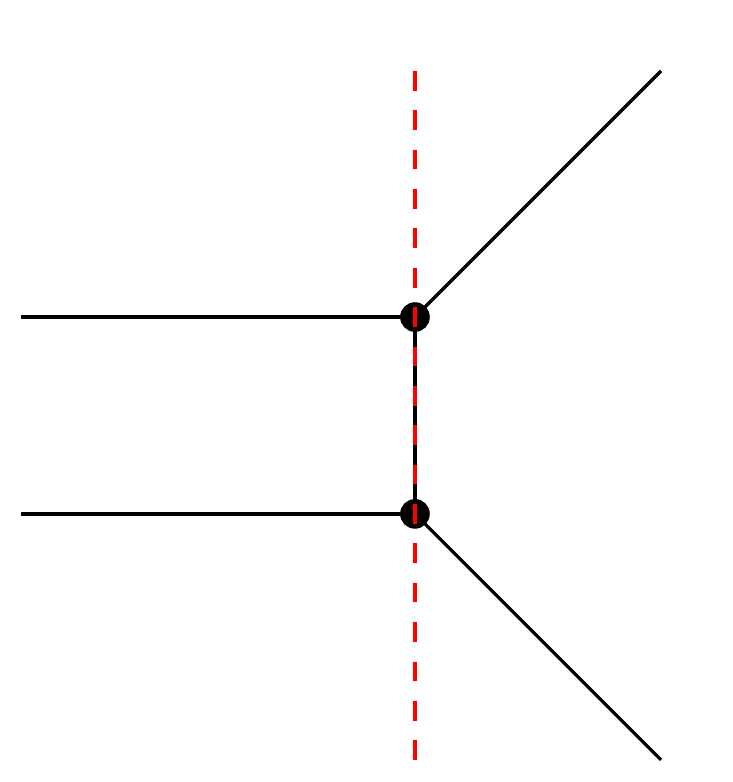}&  \hspace{10ex} & \includegraphics[height=4.5cm, angle=0]{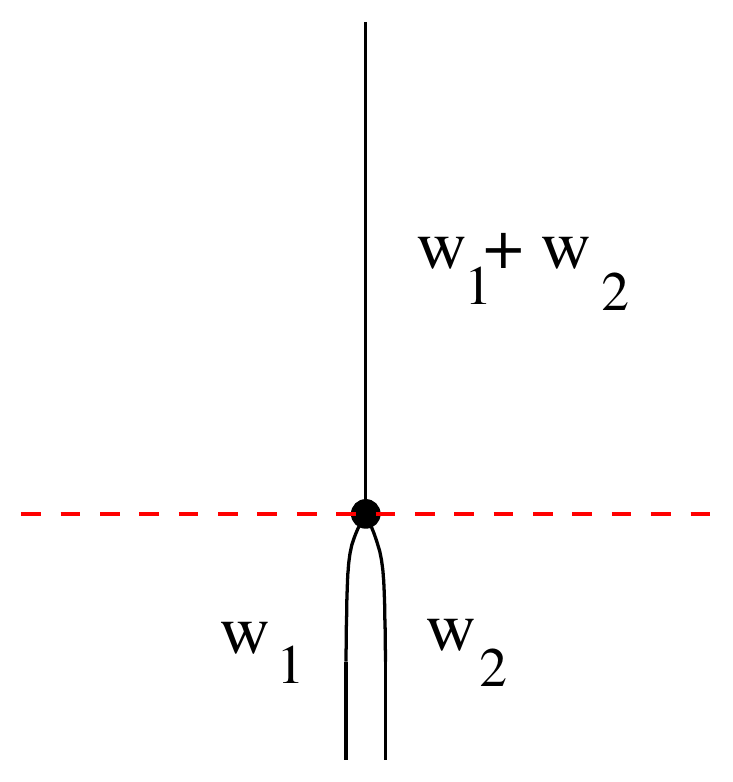}\\
 &d) & & e) & 
\end{tabular}

\caption{Pretangent morphisms}
\label{Fig:pretang}
\end{figure}

\begin{exa}
In  Figure~\ref{Fig:pretang} we  depicted several examples 
of the image of a morphism pretangent to a cycle which is represented by doted lines.
\end{exa}

It is clear that not any pretangency set corresponds to some
classical tangency point. For example, the two tropical lines in Figure
\ref{Fig:pretang}b are
pretangent, but this pretangency set doesn't correspond to
any tangency point 
between
two complex algebraic lines in $\CC P^2$. 
However, given any approximation of $f$ (if one exists) and 
any approximation of $L$ by algebraic curves, 
the accumulation set of tangency points of these approximations 
must lie
 inside the pretangency sets of $f$
and $L$ (see Section \ref{sec:proof corres}, or \cite{Br12}).

\subsection{Correspondence}\label{corres ns}

As in Section \ref{intro} let $d\ge 1$, $ k\ge 0$, and
$d_1,\ldots, d_{3d-1-k}>0$ be some integer numbers,
 and choose $\P=\{p_1,\ldots, p_k \}$ a set of $k$ points in
$\RR^2$, 
and 
 $\L=\{L_1,\ldots, L_{3d-1-k} \}$ a set of $3d-1-k$
non-singular effective tropical cycles
in $\RR^2$ such that $L_i$ has degree $d_i$. 
We denote by $\S^\TT(d,\P,\L)$ the set
of minimal  tropical morphisms 
$f:C\to \RR^2$ of degree $d$, where $C$ is a rational tropical curve with 
$\partial C=\emptyset$,  
passing through all 
points $p_i$ and pretangent to all curves $L_i$.

We suppose now that the configuration $(\P,\L)$ is \textit{generic}. 
The precise definition of this word will be given in 
Section \ref{generic conf}, Definition \ref{def generic}.
 For the moment, it is sufficient to have
in mind that the set of generic configurations
is a dense open subset of the set of all configurations with a given
number of points and tropical cycles.
The proof of the next three statements will be given in Section
\ref{generic conf}.
 Proposition \ref{generic curves} and Lemma
\ref{inter constr}
are direct consequences of Corollary \ref{finite2}

\begin{prop}\label{generic curves}
The set $\S^\TT(d,\P,\L)$ is finite, and any of its element
 $f:C\to\RR^2$ satisfies
\begin{itemize}
\item $C$ is a 3-valent curve with exactly 3d leaves, all of them
  of weight 1;
\item  $f(\Ve^0(C))\bigcap\Big(\bigcup_{L\in\L} \Ve^{0}(L)\cup\P
  \Big)=\emptyset $, i.e. no vertex of $C$ is mapped to a vertex of
  a curve  in $\L$ nor to a point in $\P$; 

\item  given $p\in\P$, if $x$ and $x'$
 are in $f^{-1}(p)$, then the (unique) path $\gamma$ in $C$ from
 $x$ to $x'$ is mapped to segment in $\RR^2$ by $f$;

\item  given $L\in\L$, there exists a connected subgraph
  $\Gamma\subset C$ which contains all pretangency components of $f$ with
  $L$, and such that $f(\Gamma)$ is a segment in $\RR^2$.

\end{itemize} 
\end{prop}

Let
$f:C\to\RR^2$ be an element of  $\S^\TT(d,\P,\L)$, 
and let us denote by $\alpha$ its combinatorial type.
Given $p\in\P$ (resp. $L\in \L$),
 we denote by $\lambda_p$ (resp. $\lambda_L$) the set of elements
of $\M_{\alpha}$ in a small neighborhood $U_f$ of $f$ 
which pass through $p$ (resp. are pretangent to $L$).
\begin{lemma}\label{lem1}
If $U_f$ is small enough, then
 $\lambda_q$ 
 spans a classical affine hyperplane $\Lambda_q$ defined
over $\ZZ$
 in
$\M_{\alpha}$ for any $q$ in $\P\cup\L$.
\end{lemma}
\begin{proof}
This is an immediate consequence of the $\ZZ$-linearity of
 the evaluation and forgetful maps $ev$ and $ft$
 (see Section \ref{generic conf}).
\end{proof}

\begin{lemma}\label{inter constr}
We have
$$\bigcap_{q\in\P\cup\L}\Lambda_q=\{f\}. $$
\end{lemma}

Let us associate a weight to the (classical) linear spaces $\Lambda_q$.
Given $p\in\P$,
 we denote by $\E(p)$ the set of
  edges of $C$  which contain a point of $f^{-1}(p)$ and we define
$$w_p=\sum_{e\in\E(p)} w_{f,e}.$$
Given $L\in\L$, we denote by  
$E_L$ the union of all pretangency components of $f$ with $L$,
by $\mu$ the cardinal of $E_L\cap \Ve^0(C)$,
 and by $\lambda$ the number of ends of $C$ contained in
$E_L$. If $v\in\Ve^0(L)$, we denote by 
$\E(v)$ the set of
 edges of $C$  which contain a point of $f^{-1}(v)$, and we define
\begin{equation}\label{wL}
w_L = \left(\sum_{v\in\Ve^0(L)}\quad \sum_{e\in\E(v)} w_{f,e}\right)  + \mu  -\lambda.
\end{equation}
Equivalently, we can define $w_L$ as follows
$$
w_L= \sum_{e\in\E(v)} w_{f,e}  \quad \quad \text{if
 $f(E_L)=v\in\Ve^0(L)$,}
$$

$$
w_L=\left(\sum_{v\in\Ve^0(L)}\quad \sum_{e\in\E(v)} (w_{f,e} +1)\right) +
 \kappa -2b_0(E_L)  \quad \text{otherwise}
$$
where $\kappa$ is the number of edges of $C\setminus E_L$ 
 adjacent to a vertex of $C$ in $E_L$.

\begin{defi}\label{def trop mult}
The  \emph{$(\P,\L)$-multiplicity of $f:C\to\RR^2$}, denoted by $\mu_{(\P,\L)}(f)$, is
defined as the tropical intersection number in
$\M_{\alpha}$ of all the tropical
hypersurfaces $\mu_{q}\Lambda_{q}$
 divided
by the number of automorphisms of $f$. That
is to say, 
$$\mu_{(\P,\L)}(f)=\frac{\left|\det\left(\Lambda_{p_1},\ldots,
\Lambda_{p_k},\Lambda_{L_1},\ldots, \Lambda_{L_{3d-1-k}} 
\right) \right| }{|Aut(f)|}\prod_{q\in\P\cup\L}w_q.
$$
The morphism $f$ is said to be \emph{tangent} to $\L$ if $\mu_{\P,\L}(f)\ne 0$.
\end{defi}

\begin{figure}[h]
\centering
\includegraphics[height=9cm, angle=0]{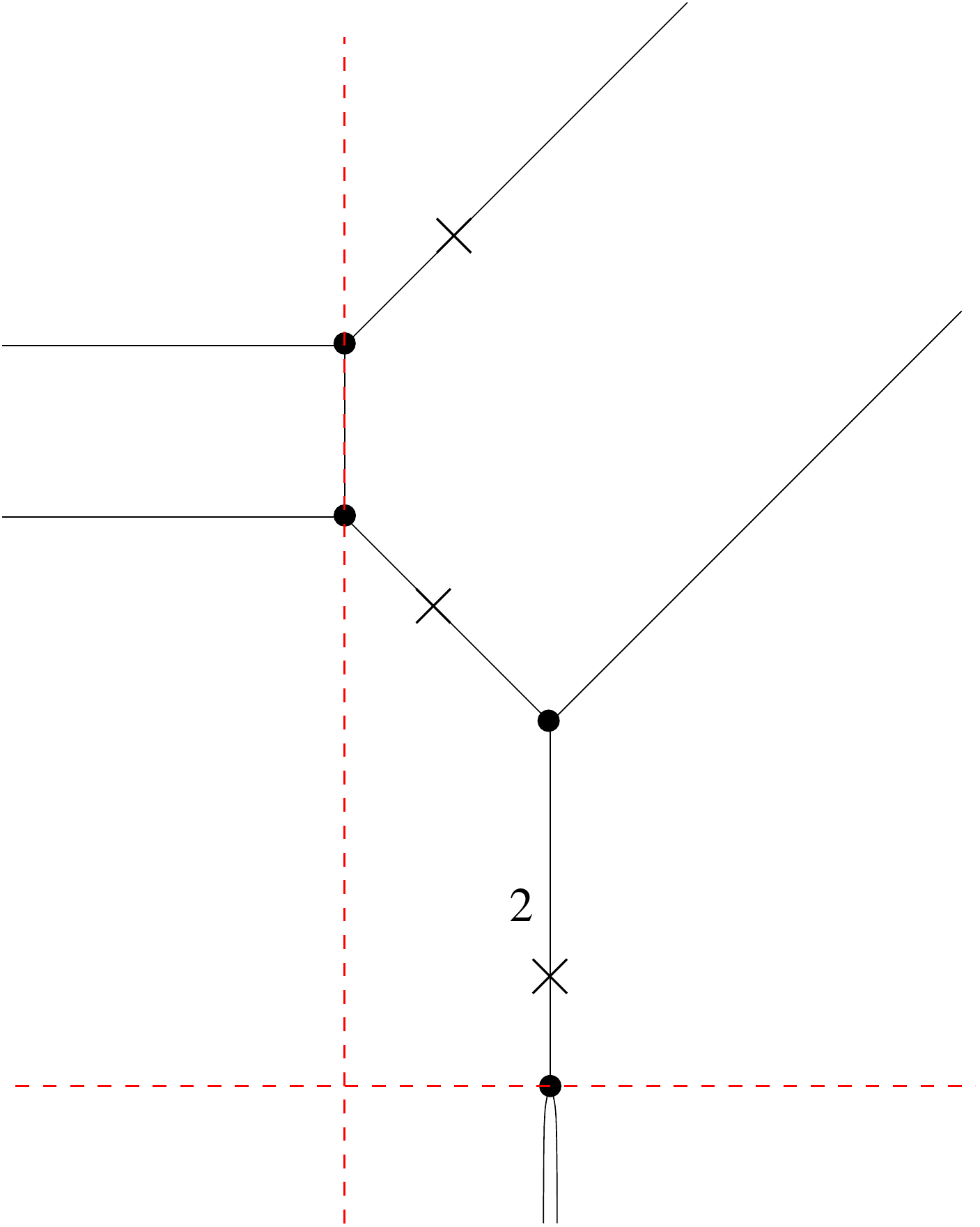}
\caption{A conic tangent to two lines and passing through three points}
\label{Fig:conic-with-constraints}
\end{figure}

\begin{exa} \label{Exa:conic}
For the morphism of Figure~\ref{Fig:conic-with-constraints}, in the
coordinate system 
described in Example \ref{coord syst ex},
let $y = \alpha_1$ and $x = \alpha_2$
be the equations respectively of the horizontal and vertical lines
and $(a_1,b_1), (a_2,b_2),$ and $(a_3,b_3)$ be the coordinates of the
three points from bottom to top. Then the hyperplanes $\Lambda_1, \dotsc, \Lambda_5$ have equations

$$\left\{
\begin{array}{rl}
y  & = \alpha_1 \\
x  & = a_1 \\
y + 2 l_1 + a_2 -a_1 & = b_2 \\
y + l_2 & = \alpha_2\\
y + 2 l_1 + l_2 + l_3 + a_3 -\alpha_2 & = b_3\\  
\end{array}
\right.$$

and  $$ 
\left|\det\left(\Lambda_1 \dotsc \Lambda_5\right)  \right| =  
\begin{vmatrix}
0 & 1 & 0 & 1 & 0\\ 
1 & 0 & 1 & 0 & 1\\
0 & 0 & 2 & 0 & 2\\
0 & 0 & 0 & 1 & 1\\
0 & 0 & 0 & 0 & 1
 \end{vmatrix}=2. $$

All the weights are $1$ except the ones associated to the bottom most point and the one associated to the tangency to the vertical line which are $2$ thus the formula of Definition~\ref{def trop mult} becomes $
\mu(f)  = \frac{2}{2}\times 2^2 =4$.

\end{exa}

\begin{thm}[Correspondence Theorem]\label{Corres}
With the hypothesis above, 
we have
$$N_{d,0}(k;d_1,\ldots,d_{3d-1-k})= \sum_{f\in\S^\TT(d,\P,\L)}\mu_{(\P,\L)}(f).$$ 

\end{thm}

\begin{exa}
For the configuration of points and lines considered in
Example~\ref{Exa:conic}, there is only one morphism in
$S^\TT(2,\P,\L)$ and  $N_{2,0}(3;1,1)$ is indeed $4$. Other examples
are given in Section~\ref{Sec:Exa}.

\end{exa}

\begin{rem}
In the proof of Theorem \ref{Corres} (see Section \ref{sec:proof corres}),
we establish not only equality of both numbers, but we also give a
correspondence between \textit{phase-tropical curves} (see Section
\ref{sec:phase} for the definition) and complex curves 
close to the tropical limit
in the sense of Section \ref{sec:trop limit}. 
In
particular, if we choose real phases,
in the sense of Definition \ref{def:phase point} and Remark \ref{rem:real phase morphism},
 for all constraints in
$(\P,\L)$, it is possible to recover all real algebraic curves passing
through a configuration of real points and tangent to a configuration
of real lines when these points and lines are close to the tropical
limit. 
See Section \ref{sec:real enumerative} for a few examples.
\end{rem}

The definition of $\mu_{(\P,\L)}(f)$ we gave so far is not very convenient for
actual computations. 
We will give in Section \ref{practical} a practical way to compute this
tropical multiplicity. We first  clarify the notion of generic
configurations in Section \ref{generic conf}.

\subsection{Generalization to immersed constraints}
In this section we generalize Theorem \ref{Corres} to the case when
the constraints are not necessarily non-singular tropical or complex
curves but any immersed curves. 
 Instead of considering 
curves of degree $d$ in the projective plane we now consider curves with a given Newton fan $\N$ in the two-dimensional torus ${\left(\CC^*\right)}^2$.

We first pose the problem in complex geometry.
Let $s \ge 2$, $k\le s-1, g_1,\ldots, g_{s-1-k}$ be some non negative integer
numbers and $\N$, $\N_1,\ldots, \N_{s-1-k}$ some Newton fans such that
the number of elements of $\N$ is $s$ (recall that a Newton fan is a multiset). Choose a set $\P=\{p_1,\ldots,
p_k \}$ of $k$ points in ${\left(\CC^*\right)}^2$, and a set $\L=\{L_1,\ldots,
L_{s-1-k} \}$ of $s-1-k$ immersed complex curves such that
$L_i$ is of genus $g_i$ and has Newton fan $\N_i$.

  We denote by $\S(\N,\P,\L)$ the set of all rational
complex algebraic maps $f:\CC P^1\setminus\{s\mbox{ points}\} \to
{\left(\CC^*\right)}^2$ with 
Newton fan $\N$, 
passing through all points $p_i\in\P$ and tangent to all curves
$L_i\in\L$.  The cardinal of the set $\S(\N,\P,\L)$ is finite as long
as $\P$ and $\L$ are chosen generically,
and we define
$$N_{\N,\P,\L}(k;\N_1,g_1,\ldots ,\N_{3d-1-k},g_{3d-1-k})=
\sum_{f\in\S(\N,\P,\L)}\frac{1}{|Aut(f)|}.$$

\vspace{2ex}
The problem in the tropical set-up is similar.
Choose now $\P^\TT=\{p_1^\TT,\ldots, p_k^\TT \}$ a set of $k$ points in
$\RR^2$, 
and 
 $\L^\TT=\{L_1^\TT,\ldots, L_{s-1-k}^\TT \}$ a set of $s-1-k$
simple effective tropical cycles
in $\RR^2$ such that $L_i^\TT$ has Newton fan $\N_i$ and genus $g_i$. 
We denote by $\S^\TT(\N,\P^\TT,\L^\TT)$ the set
of minimal  tropical morphisms 
$f:C\to \RR^2$ has Newton fan $\N$, where $C$ is a rational tropical curve with 
$\partial C=\emptyset$,   
passing through all 
points $p_i^\TT$ and pretangent to all curves $L_j^\TT$.
Next proposition is a
 direct consequence of Corollary \ref{finite2}

\begin{prop}
The set $\S^\TT(\N,\P^\TT,\L^\TT)$ is finite and its elements are extreme in the sense of Definition~\ref{def extreme}.
\end{prop}
Moreover, Lemmas \ref{lem1} and \ref{inter constr} still hold in this situation.
  If $p^\TT\in\P^\TT$, we define $w_p$
in the same way as in Section \ref{corres ns}. If $L^\TT\in\L^\TT$, 
we define $w_L$ as follows 
$$w_{L^\TT}= \left(\sum_{v\in\Ve^0(L^\TT)}\quad \sum_{e\in\E(v)} w_{f,e}\right)
+ \delta(\mu  -\lambda)$$ 
where $\delta$ is the weight of the edge of $L$ containing 
the tangency set of $f$ and $L^\TT$ (if any, otherwise $\mu=\lambda=0$).
The $(\P^\TT,\L^\TT)$-multiplicity of a tropical morphism $f$ in
$\S^\TT(d,\P^\TT,\L^\TT)$ is given by Definition \ref{def trop mult}.

\begin{thm}[Correspondence Theorem in ${\left(\CC^*\right)}^2$]\label{Corres2}
Let $\N$, $s$, $k$, $g_1,\ldots, g_{s-1-k}$, $\N_1,\ldots, \N_{s-1-k}$, $\L^\TT$ and $\P^\TT$ be as above.

There exists a generic configuration of points $\P=\{p_1,\ldots, p_k
\}$ and immersed complex curves $\L =\{L_1,\ldots, L_{s-1-k} \}$ in
${\left(\CC^*\right)}^2$ having respective Newton fans $\N_1,\ldots,
\N_{s-1-k}$ and genera $g_1,\ldots, g_{s-1-k}$ such that

$$N_{\N,\P,\L}(k;\N_1,g_1,\ldots,\N_{3d-1-k},g_{s-1-k})=
 \sum_{f\in\S^\TT(\N,\P^\TT,\L^\TT)}\mu_{(\P^\TT,\L^\TT)}(f).$$ 
\end{thm}

Theorem \ref{Corres2} we will be proved in Section~\ref{sec:proof corres}.

\section{Generic configurations of constraints}\label{generic conf}

\subsection{Marked tropical curves}
In order to prove Proposition \ref{generic curves}, 
it is convenient to consider marked
tropical curves as a technical tool.

\begin{defi}
A 
\emph{tropical curve  with $n$ marked points} is a $(n+1)$-tuple
$(C,x_1,\ldots,x_n)$ where $C$ is a 
tropical curve and the $x_i$'s are
$n$  points on $C$.

A \emph{marked  tropical morphism with $n$ marked points} is a $(n+2)$-tuple
$(C,x_1,\ldots,x_n,f)$ where $(C,x_1,\ldots,x_n)$ is a 
tropical curve
with $n$ marked points, and $f :C\to \RR^2$ is a tropical morphism.
\end{defi}

Note that we do not require the marked points to be distinct.
As in the case of unmarked tropical curves, we have the notion of
\textit{isomorphic marked tropical morphisms} and 
\textit{combinatorial  types of marked tropical morphisms}. 
The definition is the same as in Section \ref{defi trop curve}, 
where we
require in addition that the map $\phi:C_1\to C_2$ sends the $ith$
marked point of $C_1$ to the $ith$
marked point of $C_2$.

Lemma \ref{exp dim 1} has a straightforward extension to the case of
marked tropical morphisms.

\begin{lemma}\label{exp dim}
Let $\alpha$ be a combinatorial type of 
 marked tropical morphisms $f:C\to
\RR^2$ with $n_\alpha$ 
marked points, $m_{\alpha}$ of which lie on edges of $C$, and where
$C$ is a rational curve with $\partial C=\emptyset$.
Then the space  $\M_\alpha$ is an open convex polyhedron in
the vector space $\RR^{2+|\Ed^0(\alpha)|+m_{\alpha}}$, and 
$$  \dim \M_\alpha =   |\Ed^\infty(\alpha)|  -1 
 -
\sum_{v\in \Ve^0(\alpha)}(val(v)-3) + m_{\alpha}.$$
\end{lemma}

Let $ \alpha$ be a  combinatorial type of
marked minimal tropical morphisms $f:C\to\RR^2$ where $C$ is a rational  tropical curve
 with $\partial C=\emptyset$, with $s$ (open) ends, and with $n_\alpha$ 
marked points, $m_{\alpha}$ of which lie on edges of $C$.
 We denote by
$ \overline \alpha$ the combinatorial type of unmarked tropical morphisms
underlying $\alpha$.
We define  the evaluation map $ev$   and the forgetful map $ft$ by
$$\begin{array}{cccc}
ev: & \M_{\alpha} &\longrightarrow &(\RR^2)^{n_\alpha}
\\ & (C,x_1,\ldots, x_{n_\alpha},f) & \longmapsto & (f(x_1),\ldots,f(x_{n_\alpha}))
\end{array}
\ \ \ \ \ \
\begin{array}{cccc}
ft: & \M_{ \alpha}  &\longrightarrow & \M_{\overline \alpha}
\\ & (C,x_1,\ldots, x_{n_\alpha},f) & \longmapsto & (C,f)
\end{array}.$$

\begin{lemma}\label{inj}
The maps $ev$ and $ft$ are $\ZZ$-affine linear on $ \M_{ \alpha}$, and 
$$\dim ev(\M_{\alpha})\le s-1+m_{\alpha}. $$
Moreover, if  equality holds then $\alpha$ is 3-valent,
 and the evaluation map is injective
on $\M_{ \alpha}$.
\end{lemma}
\begin{proof}
The proof of the $\ZZ$-linearity is the same as \cite[Proposition
  4.2]{GM1}, and the dimension is given
by Lemma \ref{exp dim}
\end{proof}

\subsection{Generic configurations}
A \textit{parameter space} is a space of the form
$$Par(k,\beta_1,\ldots,\beta_{s-1-k})=\left(\RR^{2}\right)^k\times
\M_{\beta_1}\times\ldots\times\M_{\beta_{s-1-k}}$$
where $s\ge 2$   and $0\le k\le s-1$ 
 are two integer numbers,  
 and 
 $\beta_1,\ldots ,\beta_{s-1-k}$ are  $s-1-k$ 
combinatorial types of 
simple
effective tropical cycles
in $\RR^2$. 
 According to
Lemma \ref{dim comb type ns}, $Par(k,\beta_1,\ldots,\beta_{s-1-k})$
is an open convex polyhedron in some Euclidean  vector space. In
particular, it has a natural topology induced by this  Euclidean
space. 
Note that if $\Ve^0(\beta_i)=\emptyset$ (i.e. $\beta_i$ is just a ray), then
$\M_{\beta_i}=\RR$.

Given an element $(\P,\L)$ of
$Par(k,\beta_1,\ldots,\beta_{s-1-k})$, where 
 $\P=\{p_1,\ldots, p_k \}$ 
and $\L=\{L_1,\ldots, L_{s-1-k} \}$,
we denote by $\C(s,\P,\L)$ the set
of all minimal tropical morphisms $f:C\to \RR^2$ where $C$ is a rational
tropical curve with $s$
ends and $\partial
C=\emptyset$, and $f$   
passes through all 
points $p_i$, and is pretangent to all curves $L_j$.

\begin{defi}\label{def extreme}
A tropical morphism $f:C\to\RR^2$ in $\C(s,\P,\L)$
is said to be \emph{extreme} if it
satisfies the four following properties

\begin{itemize}
\item[(1)] $C$ is a 3-valent curve;
\item[(2)]  $f(\Ve^0(C))\bigcap\Big(\bigcup_{L\in\L} \Ve^{0}(L)\cup\P
  \Big)=\emptyset $, i.e. no vertex of $C$ is mapped to a vertex of
  a curve  in $\L$ nor to a point in $\P$; 

\item[(3)]   given $p\in\P$, if $x$ and $x'$
 are in $f^{-1}(p)$, then the path $\gamma$ in $C$ from
 $x$ to $x'$ is mapped to a segment by $f$;

\item[(4)]    given $L\in\L$, there exists a connected subgraph
  $\Gamma\subset C$ which contains all pretangency components of $f$ with
  $L$, and such that $f(\Gamma)$ is a segment in $\RR^2$.

\end{itemize}
\end{defi}

Note that in the case where $\Ve^0(L)=\emptyset$, property (4) is
equivalent to the fact that $E_L$ is connected.

\begin{prop}\label{finite}
Suppose that $\Ve^0(\beta_i)=\emptyset$ for all $i$. Then 
there exists a dense open subset $\widehat{Par}
(k,\beta_1,\ldots,\beta_{s-1-k})$ in the parameter space 
$Par(k,\beta_1,\ldots,\beta_{s-1-k})=\RR^{s-1+k}$
 such that the set $\C(s,\P,\L)$ is finite and contains only extreme
 tropical morphisms.
\end{prop}
\begin{proof}
 Denote by  
 $\C'(s,\P,\L)$ the set of all
 marked minimal tropical  morphisms $f:(C,x_1,\ldots,x_{s-1})\to
\RR^2$  where $(C,x_1,\ldots,x_{s-1})$ is a rational
tropical curve with $s$
leaves, with $\partial
C=\emptyset$,  and with $s-1$ marked points,
$x_{k+1},\ldots, x_{s-1}$ being on vertices of $C$,
 such that
 $f(x_i)=p_i$ for $1\le i\le k$ and  $f(x_i)\in L_{i-k}$ for $k+1\le
i\le s-1$.
By Definition \ref{defi tang}, the set $\C(s,\P,\L)$
is contained in 
 the set $ft(\C'(s,\P,\L))$. Hence,
it is sufficient to prove that 
there exists a dense open subset of the parameter space such that
the set $\C'(s,\P,\L)$ is finite and contains only
extreme tropical morphisms.

\vspace{1ex}
Let $ \alpha$ be a combinatorial type  of
marked   tropical  morphisms $f:(C,x_1,\ldots,x_{s-1})\to
\RR^2$ where $(C,x_1,\ldots,x_{s-1})$ is a rational
tropical curve with $s$
leaves, with $\partial
C=\emptyset$,  and with $s-1$ marked points,
$x_{k+1},\ldots, x_{s-1}$ being on vertices of $C$.
Let us
define the incidence variety $\I_{ \alpha}\subset
\M_{ \alpha}\times 
\RR^{s-1+k}  $ containing elements 
$(f, p_1,\ldots,p_k,L_{1},\ldots, L_{s-1-k})$  
where $f$ is a  tropical morphisms with
 combinatorial type $ \alpha$ 
such that
\begin{itemize}
\item[(i)] for $j\le k$, $f(x_i)=p_i$,
\item[(ii)] for $k+1\le j\le s-1$,  $f(x_j)\in L_{j-k}$.
\end{itemize}
 
Any of the above conditions (i) (resp. (ii)) demands 2 (resp. 1)
affine conditions on elements
of $\I_{\alpha}$. 
Moreover, all these conditions are independent since any
variable $p_i$ or $L_i$ is contained in exactly one of these
equations. Hence, 
the set $\I_{\alpha}$ is an open polyhedral complex of
dimension at most $\dim  \M_{\alpha}$. 
Let us consider the two natural projections 
$\pi_1:\I_{ \alpha}\to \M_{ \alpha}$ and 
 $\pi_2:\I_{ \alpha} \to Par(k,\beta_1,\ldots,\beta_{s-1-k})$.
For each  $1\le i\le s-1-k$, there is a natural linear isomorphism
 between
$\M_{\beta_i}$ and the quotient of $\RR^2$ by the linear direction of
elements of $\beta_i$. This provides a natural map 
$\psi:(\RR^2)^{s-1}\to   Par(k,\beta_1,\ldots,\beta_{s-1-k})$ by
taking the quotient of the $s-1-k$ lasts copies of $\RR^2$ by the direction of
the corresponding $\beta_i$. We have $\pi_2=\psi\circ ev
\circ \pi_1$. Hence according to Lemma \ref{inj}, if $\dim
\pi_2(\I_\alpha)=s+k-1$, 
 then 
$\alpha$ is necessarily trivalent, the evaluation map is injective
on $\M_{\alpha}$, and  $x_i\notin \Ve^0(C)$ for $i\le k$.

We define $\widetilde Par(k,\beta_1,\ldots,\beta_{s-1-k})$ as the
complement in $Par(k,\beta_1,\ldots,\beta_{s-1-k})$
 of the union of the sets $\pi_2(\I_\alpha)$ where $\alpha$
ranges over all  combinatorial types such that
 $\dim \pi_2(\M_\alpha)<s-1+k$.
This is an open and dense subset of
$Par(k,\beta_1,\ldots,\beta_{s-1-k})$.
Moreover, by injectivity of the evaluation map, 
if 
$(\P,\L)$
is in 
$\widetilde Par(s,k,\beta_1,\ldots,\beta_{s-1-k})$, 
then the fiber
$\pi_2^{-1}(\P,\L)$
consists of at most 1 point for any 
combinatorial type.
The number of possible combinatorial types 
$\alpha$
is finite, so is the set 
$\C'(s,\P,\L)$.
Moreover, all its elements satisfy properties (1) and (2) of
Definition \ref{def extreme}.

\vspace{1ex}
Let us now find an open dense subset of $\widetilde
Par(k,\beta_1,\ldots,\beta_{s-1-k})$ which will ensure
properties (3) and (4).
Let $\alpha$ be a combinatorial type considered above of 
marked tropical morphisms, and let
$ \alpha'$ be a combinatorial type of marked tropical
morphisms  consisting of $\alpha$ with an additional marked point
$x_{s}$. We also fix $1\le i\le s-1$, and we suppose that the path
$\gamma$ in $C$ joining $x_s$ to $x_i$  is not mapped to a
segment by $f$.

If $i\le k$, we define $\I_{ \alpha',i}\subset \M_{ \alpha'}\times
\RR^{s-1+k}\times \RR^2$ by the conditions (i) and (ii) above and
the condition $f(x_{s})=p_{s}$ ($p_{s}$ is the coordinate
corresponding to the copy of $\RR^2$ we added
 to $Par(k,\beta_1,\ldots,\beta_{s-1-k})$).
We still have the two projections $\pi_1:\I_{ \alpha',i}\to
\M_{\alpha'}$ and  
 $\pi_2:\I_{ \alpha',i} \to
Par(k,\beta_1,\ldots,\beta_{s-1-k})\times\RR^2$, in addition to
the projection 
$\pi: Par(k,\beta_1,\ldots,\beta_{s-1-k})\times\RR^2\to
Par(k,\beta_1,\ldots,\beta_{s-1-k})$.
According to the previous study,  the set
$\pi_2(\I_{ \alpha',i})$ has codimension at least 1 in 
$Par(k,\beta_1,\ldots,\beta_{s-1-k})\times\RR^2$, and
 none of  the two sets
 $\pi_2(\I_{ \alpha',i})$ 
and  $\{p_i=p_{s}\}$ contains the
other. Since the latter has codimension 2 in 
$Par(k,\beta_1,\ldots,\beta_{s-1-k})\times\RR^2$,
 the intersection  $X_{ \alpha',i}$  of
 $\pi_2(\I_{ \alpha',i})$ 
and  $\{p_i=p_{s}\}$ has codimension at least 3, and
$\pi(X_{ \alpha',i})$ has codimension at least 1 in 
$Par(k,\beta_1,\ldots,\beta_{s-1-k})$.

\vspace{1ex}
If $k+1\le i\le s-1$, we define $\I_{\alpha',i}\subset \M_{ \alpha'}\times
\RR^{s-1+k}\times \M_{\beta_{i-k}}$ by the conditions (i) and (ii) above and
the condition $f(x_{s})\in L_{s}$ ($L_{s}$ is the coordinate
corresponding to the copy of $ \M_{\beta_{i-k}}$ we added
 to $Par(k,\beta_1,\ldots,\beta_{s-1-k})$). 
We still have the two projections $\pi_1:\I_{\alpha',i}\to
\M_{\alpha'}$ and  
 $\pi_2:\I_{\alpha',i} \to
Par(k,\beta_1,\ldots,\beta_{s-1-k})\times\M_{\beta_{i-k}} $, in addition to
the projection 
$\pi: Par(k,\beta_1,\ldots,\beta_{s-1-k})\times \M_{\beta_{i-k}} \to
Par(k,\beta_1,\ldots,\beta_{s-1-k})$.
According to the previous study,  the set
$\pi_2(\I_{\alpha',i})$ has codimension at least 1 in 
$Par(k,\beta_1,\ldots,\beta_{s-1-k})\times\M_{\beta_{i-k}} $. Moreover,
since the path
$\gamma$ in $C$ joining $x_s$ to $x_i$  is not mapped to a
segment by $f$,
 none of  the two sets
 $\pi_2(\I_{ \alpha',i})$ 
and  $\{L_{i-k}=L_{s}\}$ contains the
other. Since the latter has codimension 1 in 
$Par(k,\beta_1,\ldots,\beta_{s-1-k})\times \M_{\beta_{i-k}} $,
 the intersection  $X_{ \alpha',i}$  of
 $\pi_2(\I_{ \alpha',i})$ 
and  $\{L_{i-k}=L_{s}\}$ has codimension at least 2, and
$\pi(X_{ \alpha',i})$ has codimension at least 1 in 
$Par(k,\beta_1,\ldots,\beta_{s-1-k})$.

\vspace{1ex}
We define
$$\widehat{Par}(k,\beta_1,\ldots,\beta_{s-1-k})=
\widetilde Par(k,\beta_1,\ldots,\beta_{s-1-k})\setminus
\left(\bigcup_{\alpha', i}\pi(X_{ \alpha',i})\right) $$
where $ \alpha'$ and  $i$ range over all possible choices in the
preceding construction.
Since the number of choices is finite, 
$\widehat{Par}(k,\beta_1,\ldots,\beta_{s-1-k})$ is a dense open
subset of $Par(k,\beta_1,\ldots,\beta_{s-1-k})$, and
for $\big( \P,\L \big)$ in this set, 
 the set 
$\C'(s,\P,\L)$ is finite and all its
 elements are extreme.
\end{proof}

\begin{cor}\label{finite2}
There exists a dense open subset $\widehat{Par}
(k,\beta_1,\ldots,\beta_{s-1-k})$ in the  parameter space
$Par(k,\beta_1,\ldots,\beta_{s-1-k})$
 such that the set $\C(s,\P,\L)$ is finite and only contains extreme
 tropical morphisms.
\end{cor}
\begin{proof}
Choose $0\le k'\le s-1-k$, and $1\le i_1<\ldots < i_{k'}\le k$. We
denote by $\{j_1,\ldots,j_{s-1-k-k'}\}$ the complement of $\{i_1,
\ldots,  i_{k'}\}$ in $\{1,\ldots, s-1-k\}$.
 Choose a vertex $v_{i_n}$ on each $\beta_{i_n}$, and an edge $e_{j_n}$ on
  each  $\beta_{j_n}$.
Given an element $L_{j_n}$ in  $\M_{\beta_{j_n}}$,
 we denote by $L_{j_n}'$ the  
effective tropical cycle in $\RR^2$ spanned by $e_{j_n}$ (this is just
the classical 
line of $\RR^2$ spanned by $e_{j_n}$), and
 denote by $\delta_{j,n}$ its
 combinatorial type.
 Then we have
  a natural surjective linear map
$$ \begin{array}{cccc}
\chi:& Par(k,\beta_1,\ldots,\beta_{s-1-k})&
\longrightarrow&
Par(k+k',\delta_{j_1},\ldots,\delta_{j_{s-1-k-k'}})
\\ & \big( p_1,\ldots,p_k,L_1\ldots 
,L_{s-1-k}\big)& \longmapsto& \big(
p_1,\ldots,p_k,v_{i_1},\ldots,v_{i_{k'}},
L'_{j_1}\ldots 
,L'_{j_{s-1-k-k'}}\big)
\end{array}.  $$
We define
$$ \widetilde Par(k,\beta_1,\ldots,\beta_{s-1-k}) =
\bigcap \chi^{-1}(\widehat{Par}(k+k',\delta_{j_1},\ldots,\delta_{j_{s-1-k-k'}})$$
 where 
$k', i_1,\ldots,i_{k'}$, $v_{i_1},\ldots
,v_{i_{k'}},e_{j_1},\ldots , e_{j_{s-1-k-k'}}$
 run over all possible
choices. Note that this number of choices is finite, and that 
$\widetilde Par(k,\beta_1,\ldots,\beta_{s-1-k})$ is open and
dense in $Par(k,\beta_1,\ldots,\beta_{s-1-k})$. 

\vspace{1ex}
Let $(\P,\L)\in \widetilde Par(k,\beta_1,\ldots,\beta_{s-1-k})$, and
$f:C\to\RR^2$ an element of $\C(s,\P,\L)$. 
By Definition \ref{defi tang}, the tropical morphism 
$f:C\to\RR^2$ is also an element of some set

$$\C(s,\{p_1,\ldots,p_k,v_{i_1},\ldots,v_{i_{k'}}\},\{L'_{j_1}\ldots 
,L'_{j_{s-1-k-k'}}\})$$
 constructed as above.
Hence, according to
 Proposition \ref{finite}, 
 the set $\C(s,\P,\L)$ is finite
 and all its elements satisfy properties (1), (2) and (3) of
 Definition \ref{def extreme}.

Moreover, for any tangency
component $E$ of $f$ with an element $L$ of $\L$, the set $f(E)$ 
is contained in a classical line of $\RR^2$. Indeed, 
otherwise there would exist $v\in\Ve^0(C)$ such that
$f(v)\in\Ve^0(L)$,
 in contradiction with the fact that $f$ satisfies
property (2) of Definition \ref{def extreme}.
Hence,
we can use  the same technique we used at the end of the proof of
Proposition \ref{finite} to construct an open dense subset
$\widehat{Par}
(k,\beta_1,\ldots,\beta_{s-1-k})$ of $\widetilde Par
(k,\beta_1,\ldots,\beta_{s-1-k})$ such that if $(\P,\L)\in\widehat{Par}
(k,\beta_1,\ldots,\beta_{s-1-k})$, then all elements of  $\C(s,\P,\L)$
satisfy property (4) of
 Definition \ref{def extreme}.
That is to say, all elements of  $\C(s,\P,\L)$ are extreme.
\end{proof}

\begin{defi}\label{def generic}
Let $p_1,\ldots,p_k$ be $k$ points in $\RR^2$ and $L_1,\ldots,L_l$
be $l$ 
simple
effective tropical cycles in $\RR^2$ of combinatorial types
$\beta_1,\ldots,\beta_l$.  
The $(k+l)$-tuple 
$(p_1,\ldots,p_k,L_1,\ldots,L_{l})$ is called  \emph{weakly generic}
 if it is an element of 
$\widehat{Par}(k,\beta_1,\ldots,\beta_{l})$.

The  $(k+l)$-tuple
$(p_1,\ldots,p_k,L_1,\ldots,L_{l})$ is called  \emph{generic}
 if any of its sub-tuple is weakly
generic.
\end{defi}

\vspace{1ex}
Automorphisms of elements of $\C(s,\P,\L)$ are pretty simple when
dealing with generic configurations.
\begin{lemma}\label{automorphism}
Let $(\P,\L)$ be a generic configuration of $s$ points and simple
effective 
tropical cycles, and
let $f:C\to \RR^2$ be an element of $\C(s,\P,\L)$. Denote by
$(e_1^1,e_2^1),\ldots , (e_1^l,e_2^l)$ all pairs of open ends of $C$
with $u_{f,e_1^i}=u_{f,e_2^i} $ and
adjacent to a common vertex in $\Ve^0(C)$. Denote also by  $\phi_{e_1^i,e_2^i}$
the automorphism of $C$ such that $\phi_{e_1^i,e_2^i}(e_1^i)=e_2^i$
and $\phi_{e_1^i,e_2^i|C\setminus\{e_1^i,e_2^i\}}=Id$.
 Then
$Aut(f)$ is the abelian group generated by the automorphisms
$\phi_{e_1^i,e_2^i}$, i.e.
$$Aut(f)=<\phi_{e_1^i,e_2^i}, \ i=1,\ldots, l> \simeq \left(\ZZ/2\ZZ\right)^l. $$
\end{lemma}
\begin{proof}
It is clear that the automorphisms $\phi_{e_1^i,e_2^i}$ commute and
are of order 2. We just have to prove that they generate
$Aut(f)$. Suppose that this group is non-trivial and let
$\phi\ne Id$ be an element of $Aut(f)$. Since $\phi$ is non-trivial,
there exist two distinct edges $e_1$ and $e_2$ in $\Ed(C)$ such that
$\phi(e_1)=e_2$.  Since two vertices of $C$ cannot have the same image
by $f$, the edges $e_1$ and $e_2$ must be adjacent to the same
vertices. The tropical curve $C$ is rational, so $e_1$ and $e_2$ must
be adjacent to exactly 1 vertex, which means that they are open ends
of $C$. 
\end{proof}

\section{Practical computation of tropical multiplicities}\label{practical}

\subsection{Combinatorial multiplicity}
Here we give a practical way to compute the determinant in Definition
\ref{def trop mult} by a standard cutting procedure (see for
example \cite{GM3}).
Let us fix a generic configuration 
$(\P,\L)\in Par(k,\beta_1,\ldots,\beta_{3d-1-k})$ and an element
$f:C\to\RR^2$ of  $\S^\TT(d,\P,\L)$ tangent to $\L$. We choose a marking of
$C$ such that $f:(C,x_1,\ldots,x_{3d-1})\to\RR^2$ 
is an element of one of the sets $\C'(3d,\P',\L')$ defined in
the proof of 
Proposition \ref{finite}, and we define 
$\oC=C\setminus\{x_1,\ldots,x_{3d-1}\}$.

First, we  define an orientation on $\oC$. Let $x$ be a point on an edge of 
$\oC$.
 Since $C$ is rational,  $C\setminus\{x\}$ has 2 connected components
 $C_1$ and $C_2$
containing respectively $s_1$ and $s_2$ ends, and 
$s_1+s_2=3d+2$. Moreover, since $(\P,\L)$ is generic,
 $C_1$ (resp. $C_2$) contains  $k_1\le s_1 -1$ (resp. $k_2\le
s_2-1$) marked points. Since $k_1+k_2=3d-1=s_1+s_2-3$, up to exchanging
$C_1$ and $C_2$, we have 
$k_1=s_1-1$ and $k_1=s_2-2$. We orient $\oC$ at $x$ from $C_1$ to $C_2$.
Note that $\oC$ and its orientation  depends  on the choice of
the marking of $C$ we have chosen,
 but this won't play a role in what follows.

\vspace{1ex}
Now we define a multiplicity $\mu_{(\P,\L)}(v)$ for each vertex $v$ in
$\Ve^0(C)$. 
If $f(v)\notin\bigcup_{L\in\L} L$, then the genericity of $(\P,\L)$ implies
that there 
exist
two edges $e_1,e_2\in\Ed(C)$ adjacent to $v$ and oriented toward $v$. We define
$$\mu_{(\P,\L)}(v)=|\det(u_{f,e_1},u_{f,e_2})|.$$
If $f(v)\in L_i$, we denote by $u_{L_i}$ the
primitive integer direction of the edge of $L_i$ containing $f(v)$.
If $f(v)\in L_i\setminus \bigcup_{L\ne L_i}  L$,
then the genericity of $(\P,\L)$ implies that there exists
exactly one edge $e\in\Ed(C)$ oriented toward $v$ with $u_{f,e}\ne u_{L_i}$,
and we define 
$$\mu_{(\P,\L)}(v)=|\det(u_{f,e_1},u_{L_i})|.$$
If $f(v)\in L_i\cap L_j$,
we define
$$\mu_{(\P,\L)}(v)=|\det(u_{L_i},u_{L_j})|.
$$

\begin{prop}\label{practical comp}
For any tropical morphism $f:C\to\RR^2$ in $\S^\TT(d,\P,\L)$, we have
$$\mu_{(\P,\L)}(f)=\frac{1}{|Aut(f)|}\prod_{q\in\P\cup\L} w_q   
\prod_{e\in\Ed^0(C)}w_{f,e} \prod_{v\in\Ve^0(C)}\mu_{(\P,\L)}(v).$$
\end{prop}

We prove Proposition \ref{practical comp} by writing down explicitly the
linear part of the equations of the hyperplanes $\Lambda_q$.
Before going deeper into the details, we remark
 that the definition of the multiplicity of a
tropical morphism in $\S^\TT(d,\P,\L)$ and of the orientation of $\oC$
are
 based only on the fact that all
elements of $\C(3d,\P,\L)$ are extreme. Hence, we can extend Definition
\ref{def trop mult} and the orientation of $\oC$
 to any tropical morphism in $\C(s,\P,\L)$ for any
$s$, as long as $(\P,\L)$ is generic. 

\vspace{1ex}
For the rest of this section, we fix some positive integer $s$,
 a generic configuration $(\P,\L)$ of constraints containing $s$
elements, and an element $f:C\to\RR^2$ of $\C(s,\P,\L)$ tangent to
$\L$. In particular, $\mu_{(\P,\L)}(f)\ne 0$.
Let us choose a vertex $v_1\in\Ve^0(C)$, and some ordering of edges in
$\Ed^0(C)$. If $v\in\Ve^0(C)$, we denote by $(v_1v)$ the path
joining $v_1$ to $v$ in $C$. In particular we have
\begin{equation}\label{equ f}
f(v)=f(v_1) + \sum_{e\in(v_1v)}l_ew_{f,e}u_{f,e}
\end{equation}
where the vectors $u_{f,e}$ are oriented toward $v_1$ (recall that
$w_{f,e}$  and $l_e$ are respectively  the weight and the length of $e$).

Given $p\in\P$, we choose $v\in\Ve^0(C)$ adjacent to an edge $e$ of
$C$ such that $e\cap f^{-1}(p_i)\ne\emptyset$. Then, in the coordinate
used in the proof of Lemma \ref{exp dim 1}, the linear part of
the equation of
$\Lambda_p$ is given by
\begin{equation}\label{equ pt}
|\det\left(f(v),u_{f,e} \right)|=0.
\end{equation}
Let $L$ be an element of $\L$.
 If there exist $v_0\in\Ve^0(L)$, 
 $e\in \Ed(C)$ adjacent to  $v\in\Ve^0(C) $, such that  
$e\cap f^{-1}(v_0)\ne\emptyset$, then the linear part of the equation of
$\Lambda_L$ is given by
\begin{equation}\label{equ dte1}
|\det\left(f(v),u_{f,e} \right)|=0.
\end{equation}
If there exists a vertex
 $v\in\Ve^0(C)$ such that $f(v)\in L$, 
then the linear part of the equation of
$\Lambda_L$ is given by
\begin{equation}\label{equ dte2}
|\det\left(f(v),u_L \right)|=0
\end{equation}
where $u_L$ is the primitive integer direction of the edges of $L$
containing $f(v)$.

Equations (\ref{equ pt}), (\ref{equ dte1}), and (\ref{equ
  dte2})
 do not depend on the choice of $e$ and $v$
thanks to properties  (3) and (4) in Definition \ref{def extreme}.

To sum up, if $q\in\P\cup\L$, the linear part of the equation of
$\Lambda_q$ 
is
always of the form 
$|\det\left(f(v_{q}),u_q \right)|=0$, for some  $v_{q}\in\Ve^0(C)$ and
$u_q=(u_{q,1},u_{q,2})\in\RR^2$.
Thus, the coefficients  of 
 the matrix $M(f)$ of intersection of all the $\Lambda_q$
 are given by equation
 (\ref{equ f}): 
 \begin{itemize}
\item the
coefficient corresponding to the hyperplane $\Lambda_q$ and $x_{v_1}$ is $u_{q,2}$;
\item the
coefficient corresponding to the hyperplane $\Lambda_q$ and $y_{v_1}$ is $-u_{q,1}$;
\item the
coefficient corresponding to the hyperplane $\Lambda_q$ and 
$e\in\Ed^0(C)$ is 
\begin{itemize}
\item $w_{f,e}\det(u_{f,e},u_q)$ if $e\in (v_1v_{q})$;
\item 0 otherwise.
\end{itemize}

\end{itemize}

Clearly, the matrix $M(f)$ depends on the choice of the coordinates
we choose on  the deformation space of $f$. However, $M(f)$ is well
defined up to a multiplication by a matrix in $GL_{s-1}(\ZZ)$, hence the
absolute value of its determinant does not depend on this choice.

\begin{lemma}\label{cut3}
Suppose  that there exists $L\in\L$ such that we are in one of the two
following situations
\begin{itemize}
\item there exists a vertex $v\in
\Ve^0(C)$ adjacent to two edges $e_1,e_2\in\Ed^0(C)$ with
$u_{f,e_1}=u_L$ and $u_{f,e_2}\ne u_L$, and such that 
$f(v)\in L$; in this case, choose $p$ in $f(e_1)\cap L$;

\item there exists a point $x\in C$ 
 such that $f(x)\in\Ve^0(L)$; in this case,  put $p=f(x)$.
\end{itemize} 
Define $\P'=\P\cup\{p\}$
and $\L'=\L\setminus\{L\}$. Then
$$\mu_{(\P,\L)}(f)= \frac{w_{L}}{w_p}\mu_{(\P',\L')}(f).$$
In particular,   Proposition \ref{practical
  comp} for $\mu_{(\P,\L)}(f)$ follows from Proposition \ref{practical
  comp} applied to $\mu_{(\P',\L')}(f)$.
\end{lemma}
\begin{proof}
The linear part of the equations of $\Lambda_L$ and $\Lambda_p$ are the same. 
\end{proof}

Let us now explain the cutting procedure to compute $\mu_{(\P,\L)}(f)$
in general.
Suppose that there exist $e\in\Ed^0(C)$ and $x\in e$ 
such that $f(x)\notin\P\bigcup_{L\in\L} L$.
Recall that we have defined an orientation on $C$ at $x$. 
The space $C\setminus\{x\}$ has two connected components, 
$C_1$ and $C_2$,
containing respectively $s_1$ and $s_2$ ends.
 We choose $C_1$ and $C_2$  so that $C$ is
oriented from $C_1$ to $C_2$ at $x$. It is clear that 
$s_1+s_2=s+2$.
The graph $C_i$ inherits a tropical structure from the tropical curve
$C$, and there is a unique way to extend $C_i$ to a rational tropical
curve $\overline C_i$ without boundary components.
Note that the tropical morphism $f:C\to\RR^2$ 
 induces tropical morphisms $f_i:\overline C_i\to \RR^2$. 

Let $\P_i\subset\P$ be the points of $\P$ through which $f_i$ passes,
and $\L_i\subset\L$ be the curves of $\L$ which are
pretangent to $f_i$. The configuration $(\P,\L)$ is generic, so we have
$|\P_i|+|\L_i|\le s_i-1$. Since we also have $s_1+s_2=s+2$ and 
$|\P|+|\L|= s-1$, we deduce that 
$\left|\left(\P_1\cup\L_1 \right)\cap\left(\P_2\cup\L_2 \right) \right|\le 1$.

If $\left(\P_1\cup\L_1 \right)\cap\left(\P_2\cup\L_2 \right)
=\emptyset$, then $|\P_2|+|\L_2|= s_2-2$ because of the orientation of
$C$ at $x$. In this case we define $\P_2'=\P_2\cup\{f(x)\}$ and $\nu=1$.

If $\left|\left(\P_1\cup\L_1 \right)\cap\left(\P_2\cup\L_2 \right)
\right|= \{q_0\}$, then we define $\P_2'=\P_2$ and
$$\nu=\frac{w_{q_0}(f)w_{f,e}}{w_{q_0}(f_2)w_{q_0}(f_1)}.$$
Since the 3 distinct tropical morphisms $f$, $f_1$, and $f_2$ pass
through or are
tangent to $q_0$, we precised to which morphism the quantity $w_{q_0}$
refers in the previous formula.

\begin{lemma}\label{cut1}
We have
$$\mu_{(\P,\L)}(f)=\nu \mu_{(\P_1,\L_1)}(f_1)\mu_{(\P'_2,\L_2)}(f_2).$$
In particular,   Proposition \ref{practical
  comp} for $\mu_{(\P,\L)}(f)$ follows from Proposition \ref{practical
  comp} applied to $\mu_{(\P_1,\L_1)}(f_1)$ and $\mu_{(\P'_2,\L_2)}(f_2)$.
\end{lemma}
 \begin{proof}
For the coordinates of the 
deformation space of $f_1$, we choose the root vertex 
 to be the vertex  of $C_1$ 
adjacent
to $e$.
For the deformation space of $f_2$, we choose the root vertex
 to be the vertex  of 
 $C_2$ adjacent
to $e$. We also suppose that the first line of $M(f_2)$ is given
by $\Lambda_{f(x)}$ or $\Lambda_{q_0}$. Choose an order
$q_2,\ldots,q_{s_2-1}$ on the other elements of $\P_2\cup\L_1$. 
Then we have
$$M(f_2)=\left(\begin{array}{ccc}
u_{e,2} & -u_{e,1}&0
\\ U_1& U_2&A
\end{array}\right)  $$
where  $A$ is a $(s_2-2)\times(s_2-3)$ matrix and
$$U_1=\left(\begin{array}{c} u_{q_2,2}\\ \vdots\\  u_{q_{s_2-1},2}\end{array}
\right) \quad
\textrm{and} \quad U_2=\left(\begin{array}{c}
  -u_{q_2,1}\\ \vdots\\  -u_{q_{s_2-1},1}\end{array} 
\right).$$
Hence, eliminating the coefficient $-u_{e,1}$ by elementary operations
on the column of $M(f_2)$ we get 
$$\det(M(f_2))= \det\left(\left(\begin{array}{cc}
 U_3&A
\end{array}\right)\right)$$
where 
$$U_3=\left(\begin{array}{c} \det(u_{f,e},u_{q_2})
  \\ \vdots\\ \det(u_{f,e},u_{q_{s_2-1}})\end{array} 
\right).$$

We choose coordinates on the deformation space of $f$ 
correspondingly to the one we chose for $f_1$ and $f_2$: the root
vertex is the root vertex we chose for $f_1$, and the order on the
edges in $\Ed^0(C)$ is given by first the edges in  $\Ed^0(C_1)$, then
$e$, and then  the edges in  $\Ed^0(C_2)$.
Then, 
we have
$$M(f)=\left(\begin{array}{ccc}
M(f_1)&0&0
\\ *&w_{f,e}U_3  &A
\end{array}\right).  $$
Hence $\det(M(f))= w_{f,e}\det(M(f_1))\det(M(f_2))$. To conclude, we
remark that according to
Lemma \ref{automorphism}, we have $Aut(f)=Aut(f_1)\times Aut(f_2)$.
\end{proof}

\vspace{2ex}
\begin{proof}[Proof of Proposition \ref{practical comp}]
Applying recursively Lemmas \ref{cut1} and \ref{cut3}, we reduce to
 cases when $\Ve^0(C)$ has at most 1 element, for which we can
easily check by hand that Proposition \ref{practical comp} is true.
\end{proof}

\subsection{Examples of computations}\label{Sec:Exa}

Let us first compute again the multiplicity of the conic of
Example~\ref{Exa:conic} pictured in
Figure~\ref{Fig:conic-with-constraints} using the combinatorial
procedure described above.

The automorphism group of the morphism is isomorphic to $\ZZ/2\ZZ$,
the multiplicities $\mu_q$ of the constraints is $1$ except for the
point sitting on an edge of weight $2$ and for the vertical line in
which case it is $2$. There is
 an edge of weight $2$
contributing for $2$ in the multiplicity and the vertices all have
multiplicity $1$ as can be seen using the formula given at the
beginning of previous section and the orientations described in
Figure~\ref{Fig:oriented-conic}. The multiplicity of this morphism is
thus $\mu_{(\P,\L)}(f) = \frac{1}{2}\times 2^2 \times 2 \times 1 = 4$
which is indeed the number of conics tangent to two lines and passing through
three points provided that the configuration is generic.

\begin{figure}[h]
\centering
\includegraphics[height=9cm, angle=0]{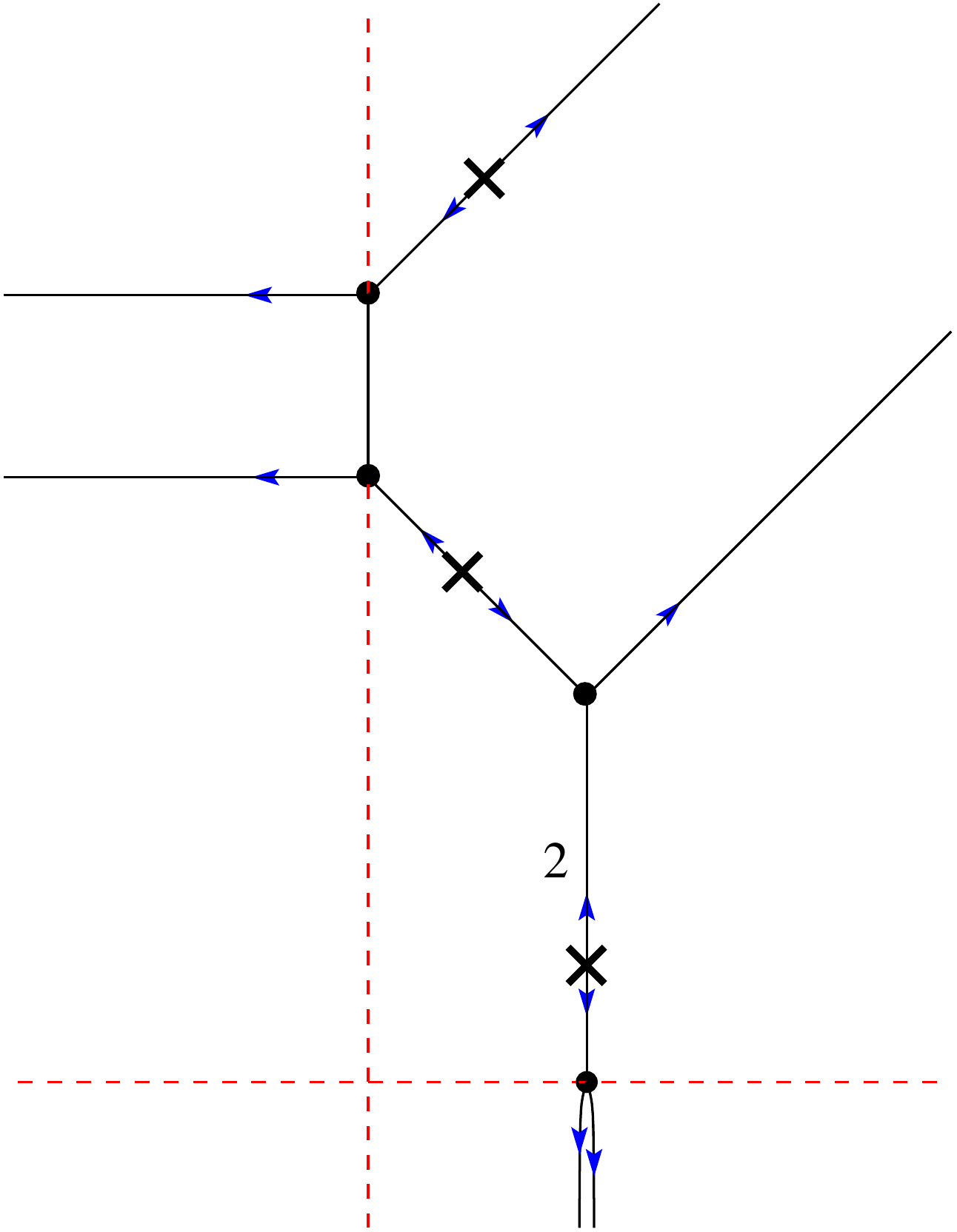}
\caption{Conic (with orientations) tangent to two lines}
\label{Fig:oriented-conic}
\end{figure}

One can check using the same techniques and
Figure~\ref{Fig:all-conics} the number of conics tangent to $5-k$
lines when $k$ varies from $0$ to $5$. In each case there is only one
tropical curve satisfying the constraints on which all the classical
conics degenerate. Their multiplicities are $\min\left( 2^k , 2^{5-k}\right)$.
\begin{figure}[h]
\centering
\begin{tabular}{ccccc}
\includegraphics[width=5cm, angle=0]{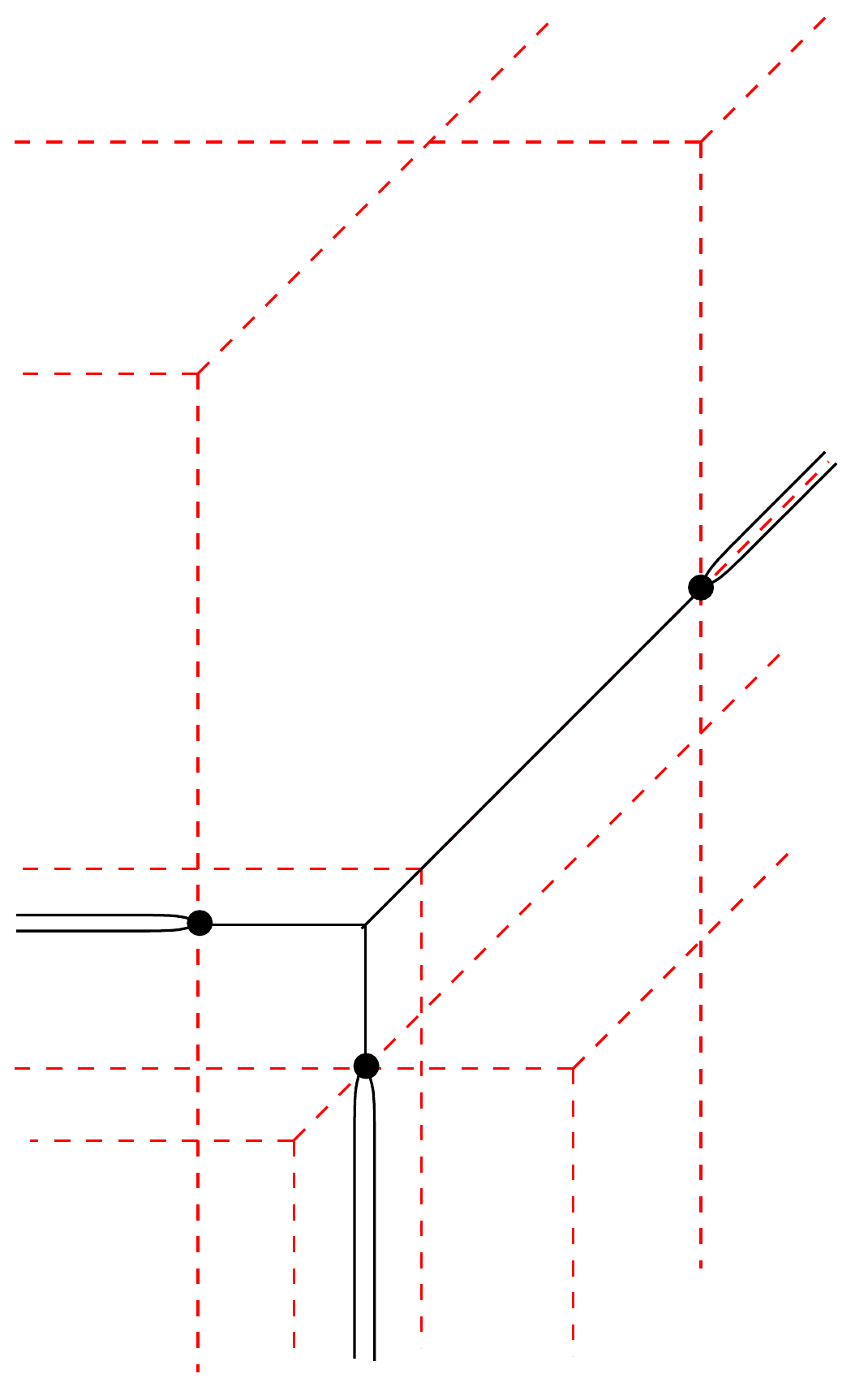}&
\includegraphics[width=5cm, angle=0]{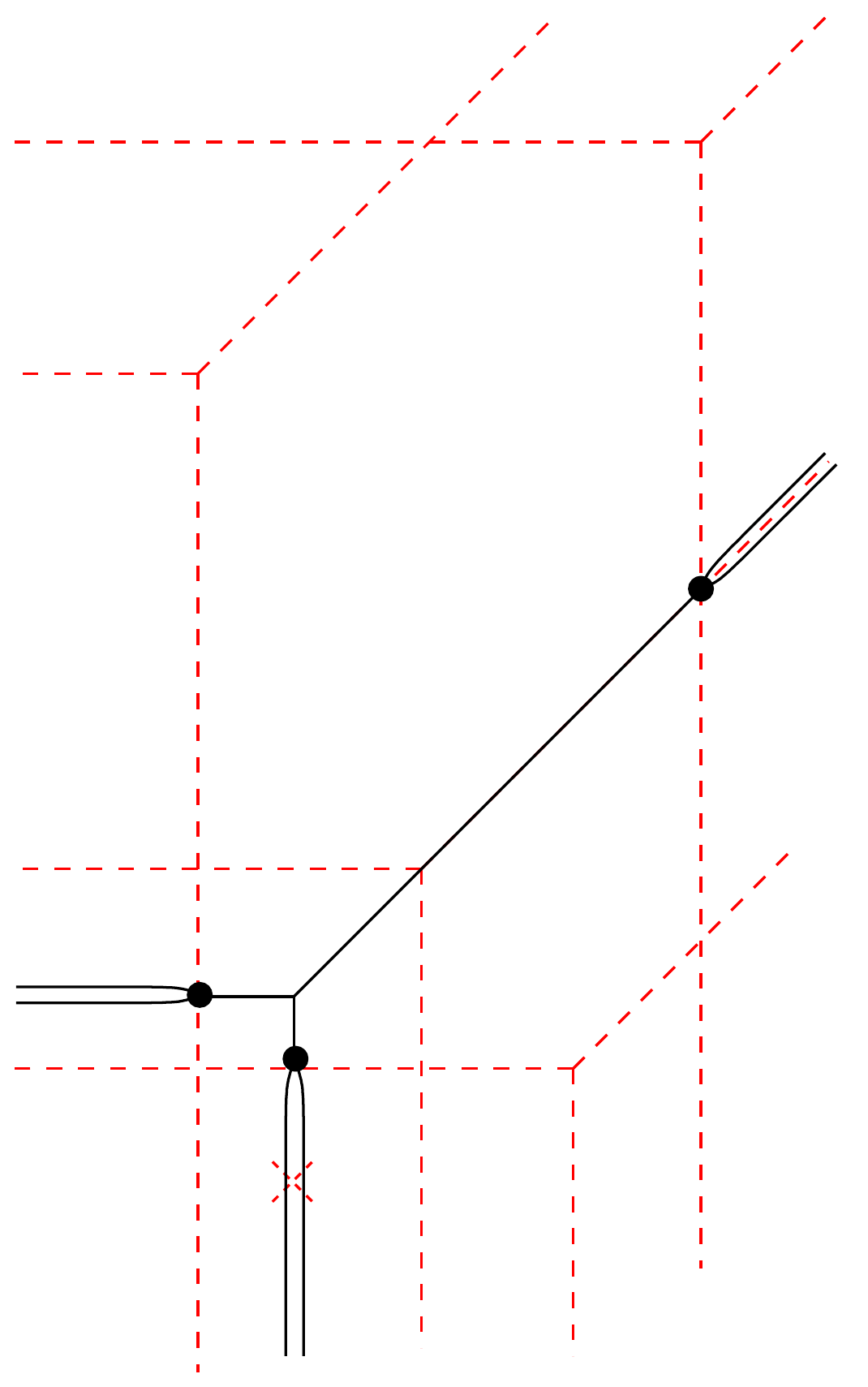}&
\includegraphics[width=5cm, angle=0]{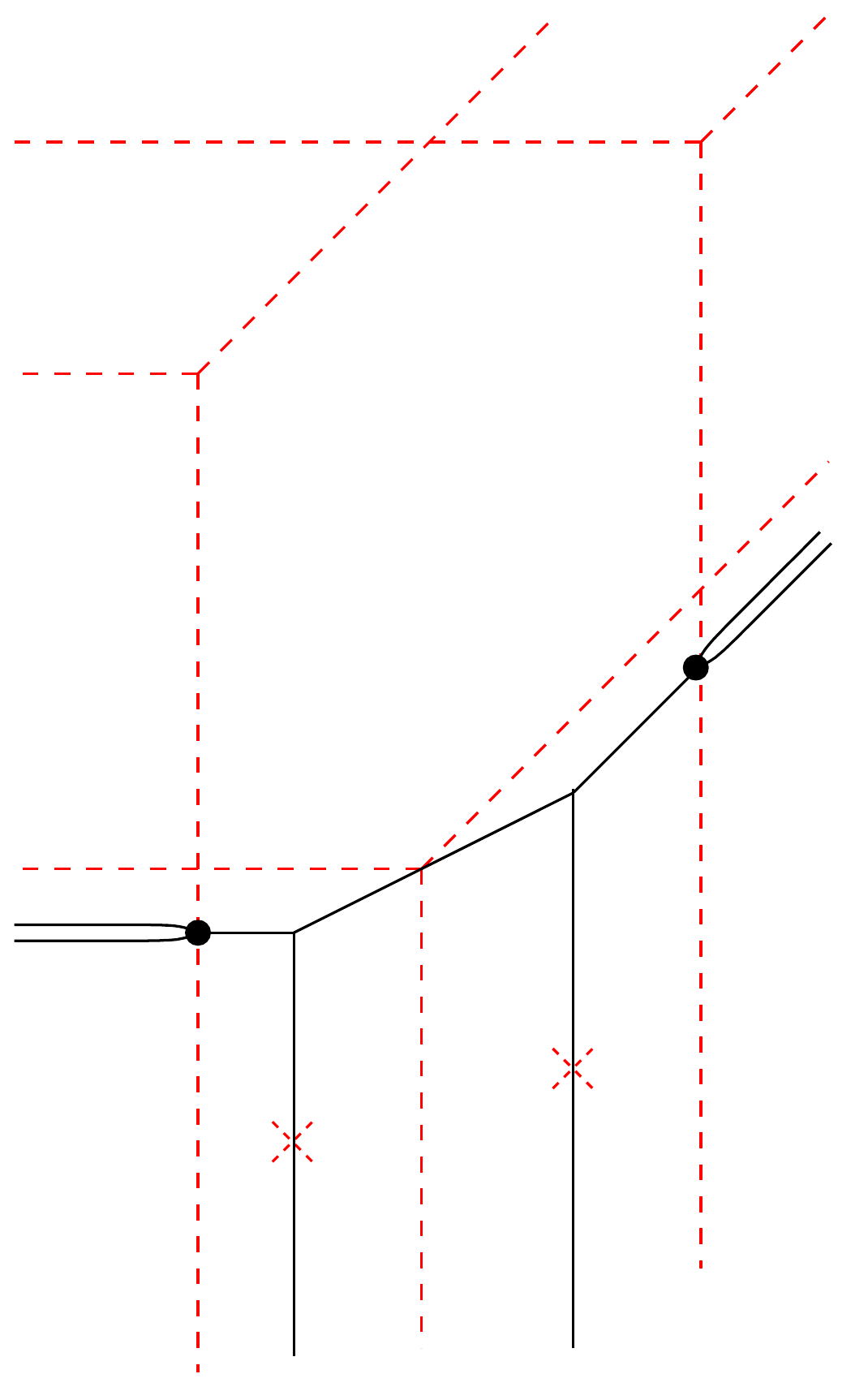}\\
\includegraphics[width=5cm, angle=0]{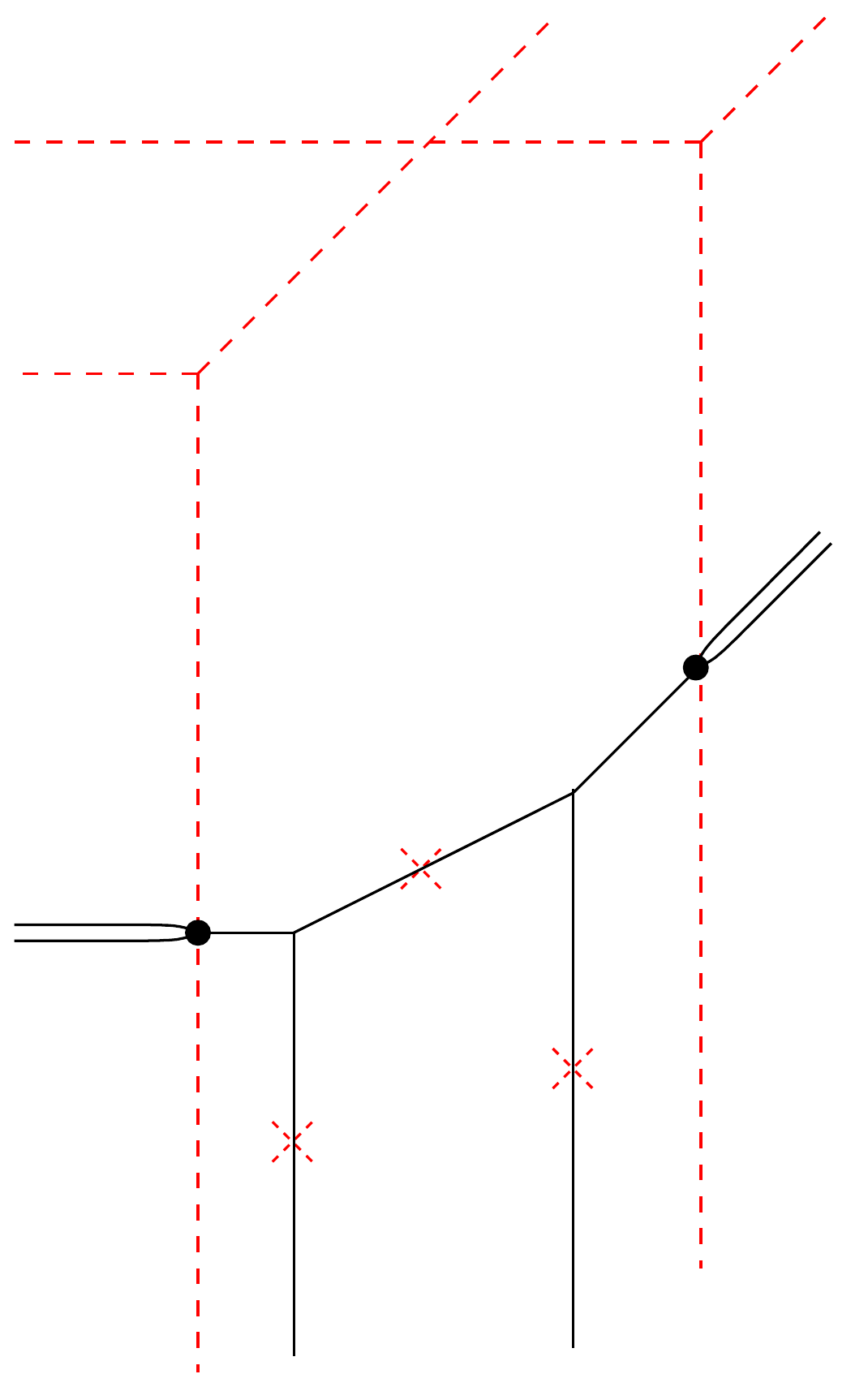}&
\includegraphics[width=5cm, angle=0]{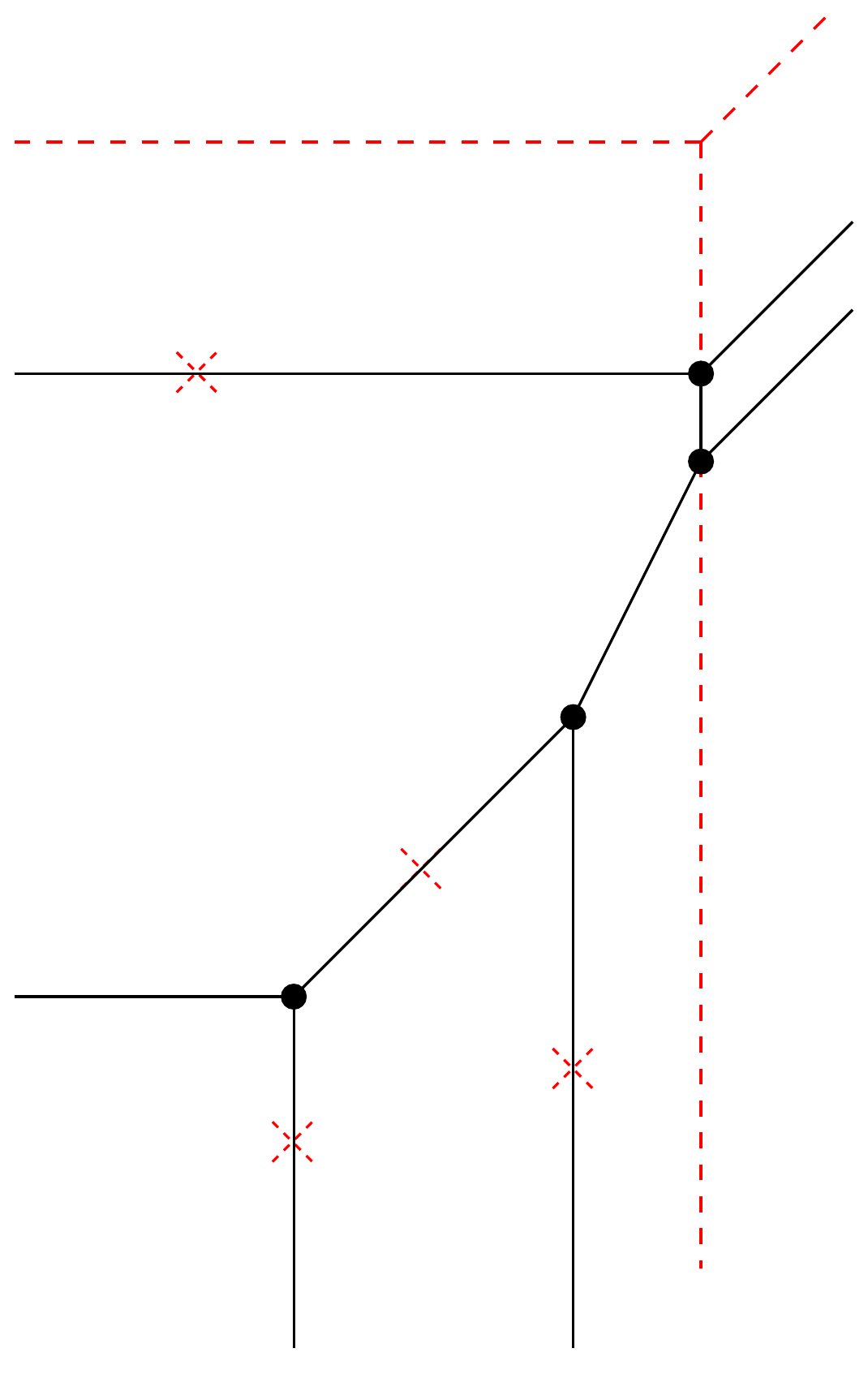}&
\includegraphics[width=5cm, angle=0]{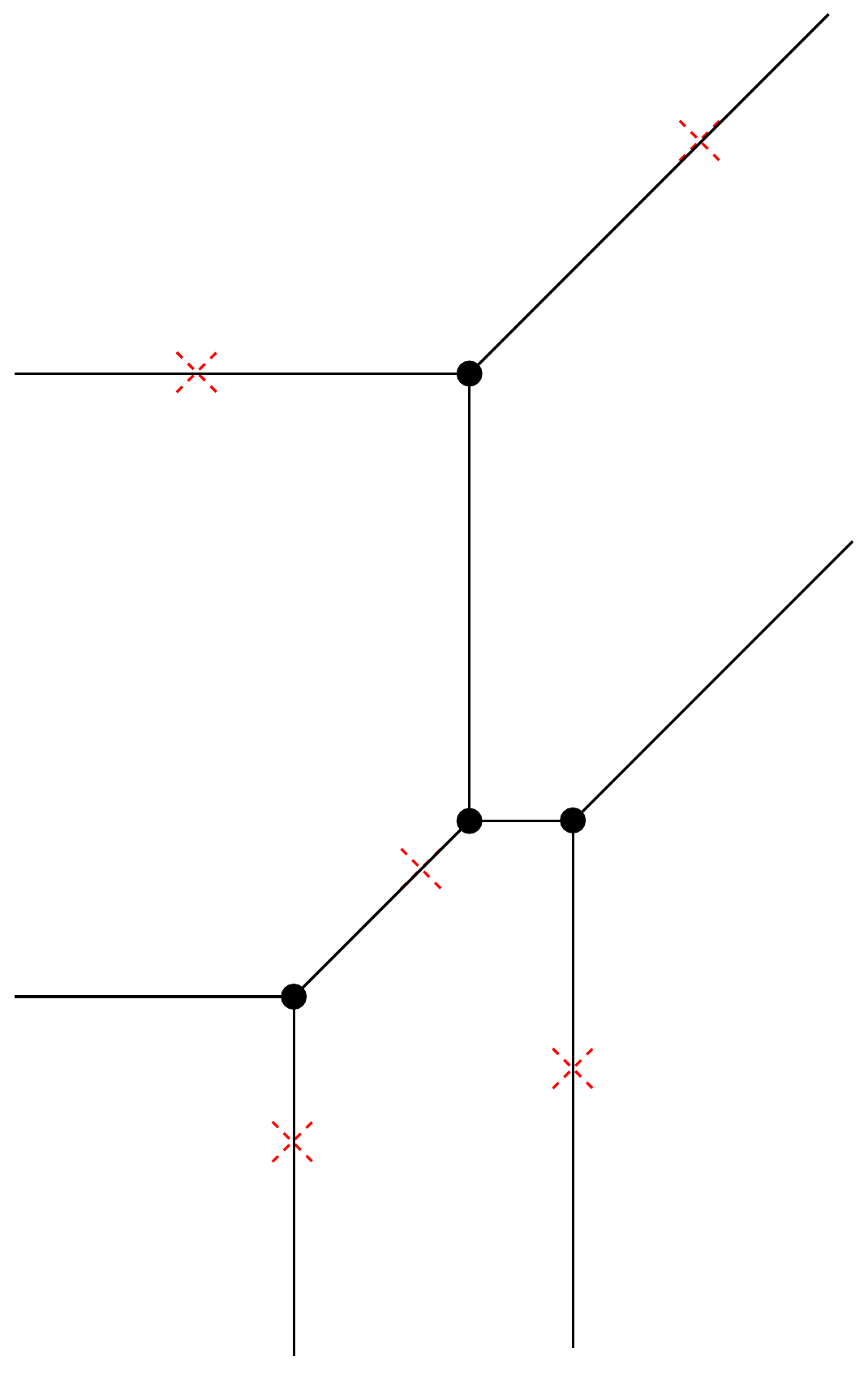}
\end{tabular}

\caption{Conics tangent to $5-k$ lines, $k = 0,\, \dots,\,5$}
\label{Fig:all-conics}
\end{figure}

Finally, we draw on Figure~\ref{Fig:cubic-oriented} one of the rational cubics tangent to seven lines and passing through one point.  

\begin{figure}[h]
\centering
\includegraphics[height=9cm, angle=0]{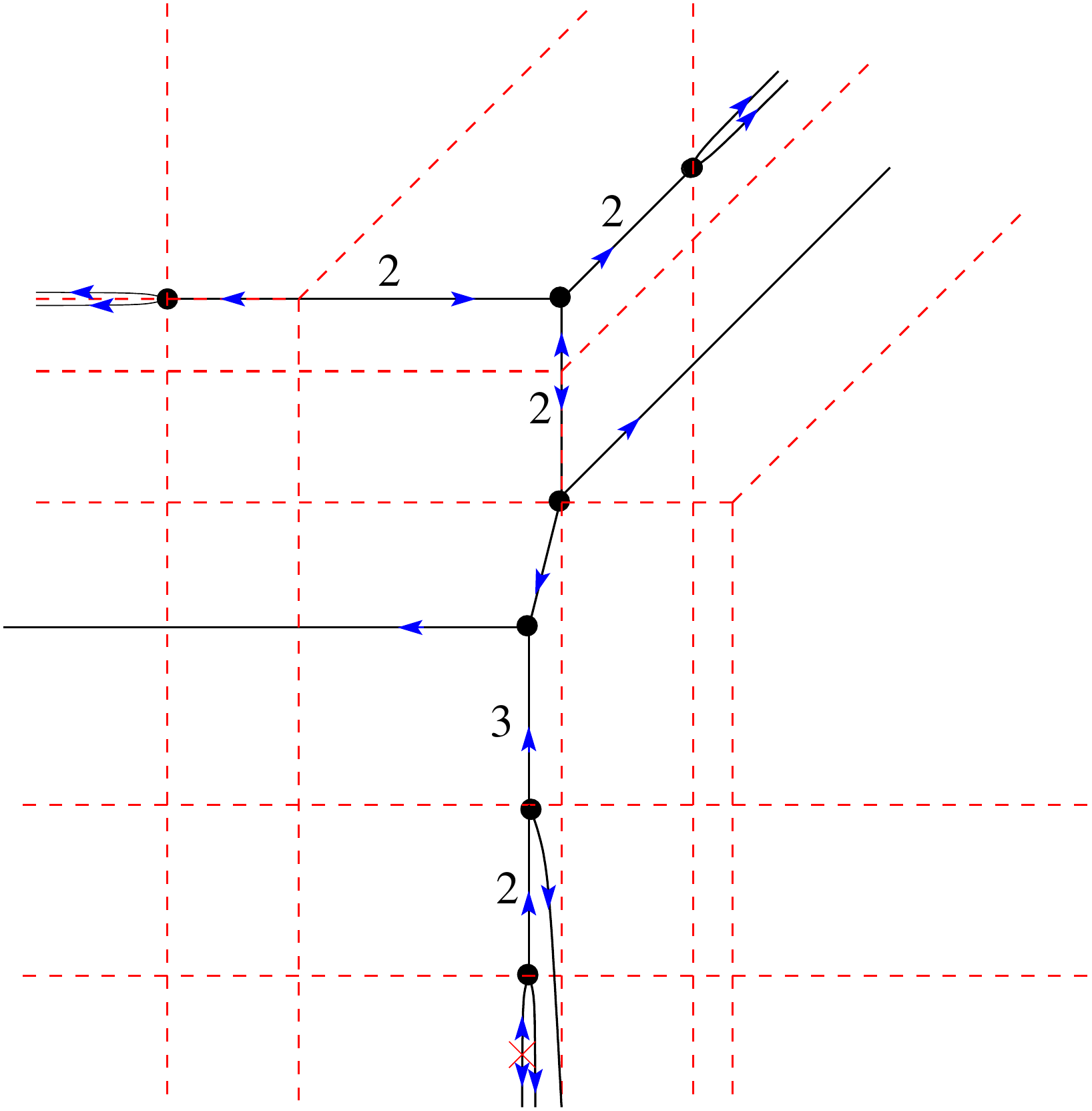}
\caption{Cubic (with orientations) tangent to seven lines}
\label{Fig:cubic-oriented}
\end{figure}

In this case the automorphism group is isomorphic to
${\left(\ZZ/2\ZZ\right)}^3$, all multiplicities of constraints are $1$
except that of the point and of the line having its vertex on a
vertical edge of weight $2$ which are $3$. The product of the weights
of interior edges is $2^4\times 3$ and all vertices have multiplicity
$1$.  This morphism thus contributes for $\mu_{(\P,\L)}(f)=54$ to the
seven hundreds classical rational cubics tangent to seven lines.

\section{Phases and tropical limits}\label{sec:phase}
Our strategy is to approximate our tropical ambient space with constraints
$(\RR^2,\P,\L)$ with a family of corresponding classical ambient
spaces.  In order to prove Correspondence Theorems with tangency
conditions as some of the constraints, we introduce the notion of
{\em phase-tropical structure} on a tropical variety and that of {\em
tropical limit}. For simplicity in this paper we define both concepts
only in the special case we need, namely for points and curves in
toric surfaces. See \cite{Mik08} for a more general case. Such an approach allows us to generalize correspondence statements started in 
\cite{Mik1} and followed in \cite{NishinouSiebert}, \cite{Nishinou},
\cite{Tyomkin} to curves tangent to given curves.

\newcommand{\ctor}{(\CC^*)^2}
\newcommand{\Arg}{\operatorname{Arg}}
\renewcommand{\C}{\CC}
\newcommand{\R}{\RR}
\newcommand{\Z}{\mathbb Z}
\newcommand{\MM}{\mathcal M}
\newcommand{\tp}{{\mathbb T}{\mathbb P}}
\newcommand{\cp}{{\mathbb C}{\mathbb P}}
\newcommand{\ctorn}{(\C^*)^n} \newcommand{\cto}{(\C^*)}
\newcommand{\U}{\mathcal U}
\newcommand{\G}{\mathcal G}
\newcommand{\T}{\mathbb T}
\newcommand{\sonen}{(S^1)^n}

All tropical curves $C$ considered in this section have no boundary
component, i.e. $\partial C=\emptyset$ (see Section \ref{defi trop
curve}). We do not make any assumption on the genus of the curves
 in
Section \ref{sec:trop limit}, the hypothesis of being rational will
be necessary only 
starting from Section \ref{sec:inv trop limit}.
\subsection{Phase-tropical structures and tropical
limits}\label{sec:trop limit}
Let $p\in\RR^n$ be a point.
\begin{defn}\label{def:phase point}
The \emph{phase of $p$} is a choice of a point $\phi(p)\in\sonen$.
Alternatively, we may think of it as a choice of point in $\ctorn$
as long as we identify two phases $\phi(p),\phi(p')\in\ctorn$ whenever they
have the same argument $$\Arg(\phi(p))=\Arg(\phi'(p))\in \sonen.$$

The phase of $p$ is \emph{real} if $\phi(p)\in \Arg((\RR^*)^n)=\{0,\pi\}^n $.
The phase of the collection $\P$ of points is a choice of phase for each point of $\P$.
\end{defn}
Clearly, in the case of points the phase is nothing else, but prescription of arguments to the coordinates.
In the case of curves we need to prescribe a phase for each vertex of the curve so that this prescription
is compatible at each edge.

Before defining the phase-tropical structure for curves we consider some motivations for it.
We refer to the upcoming paper \cite{Mik08} for a more thoroughful
treatment, but as
this paper has not appeared yet we give  below a preview relevant to
our purposes here. 
A tropical curve can be thought of as a certain degeneration of a sequence (or a family, which
can be thought of as a generalized sequence) of complex curves, i.e. Riemann surfaces $S_{t_j}$
whose genus $g$ and number of punctures $k$ do not depend on the parameter $t_j$.

From the hyperbolic geometry viewpoint each $S_{t_j}$ is a hyperbolic surface and is completely
determined by the length of any collection of $3g-3+k$ disjoint closed
embedded geodesics (see \cite{Thurston97}).
Such a collection defines a decomposition of $S_{t_j}$ into pair-of-pants. This means that every connected
component of the complement of this collection is homeomorphic to a sphere punctured three times.
Some of these punctures correspond to the punctures of $S$ while others correspond to a geodesic
from our collection. Note that each geodesic corresponds to two punctures of pairs-of-pants, cf. Figure~\ref{pair-of-pants}.
Conversely, each decomposition into pairs-of-pants gives a collection of $3g-3+k$ disjoint closed embedded geodesics
once we represent each cutting circle by a geodesic in the hyperbolic metric of $S_{t_j}$. 
Once all surfaces $S_{t_j}$ are marked (i.e. a homotopy equivalence with a ``standard surface" $S$ of genus $g$ with $k$ punctures
is fixed) the pairs-of-pants decomposition in $S_{t_j}$ with different $t_j$ can be chosen in a compatible way.

\begin{figure}[h]
\centering
\begin{tabular}{ccc}
\includegraphics[width=5cm,angle=0]{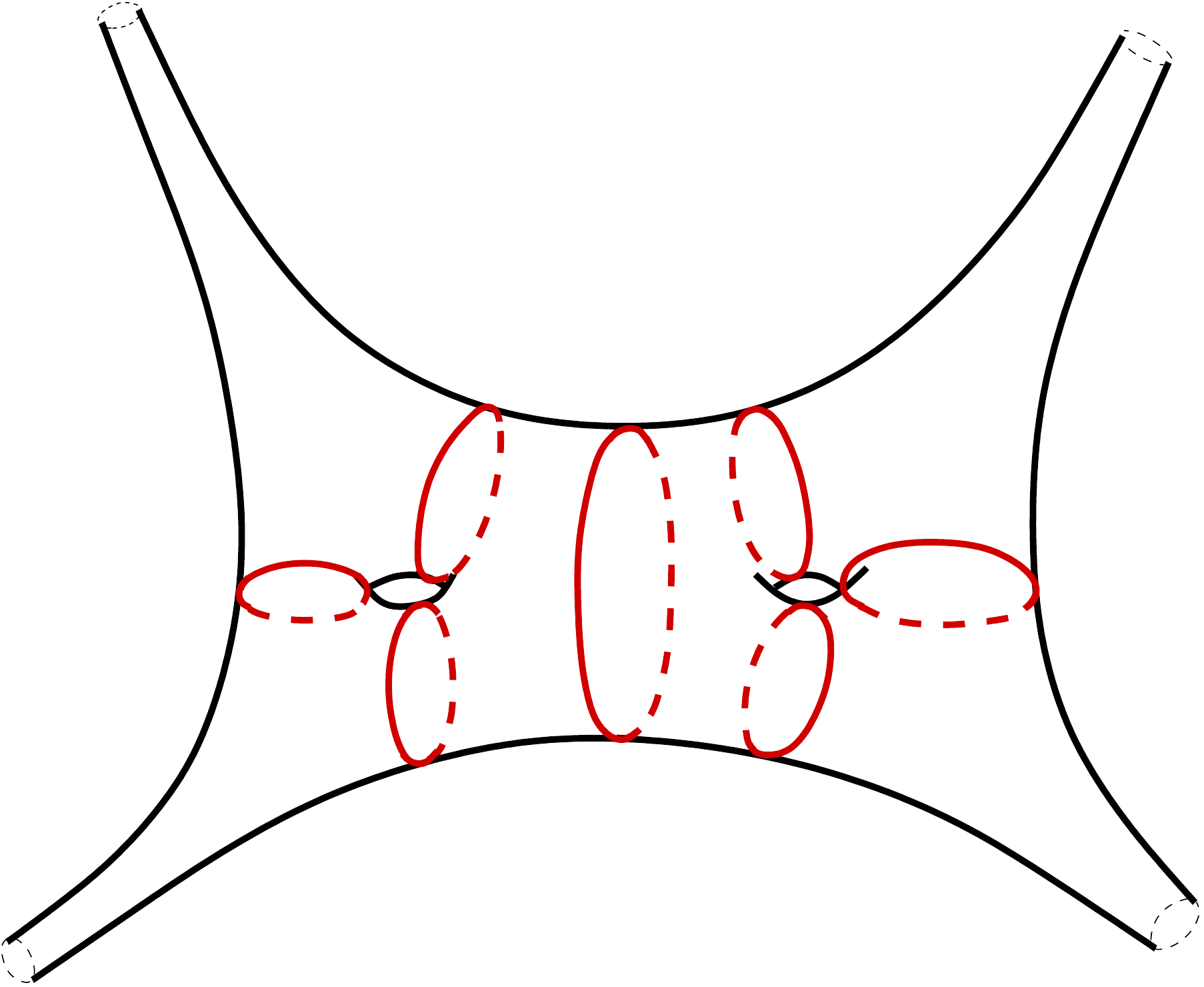}
& \hspace{8ex} &
\includegraphics[width=4cm, angle=0]{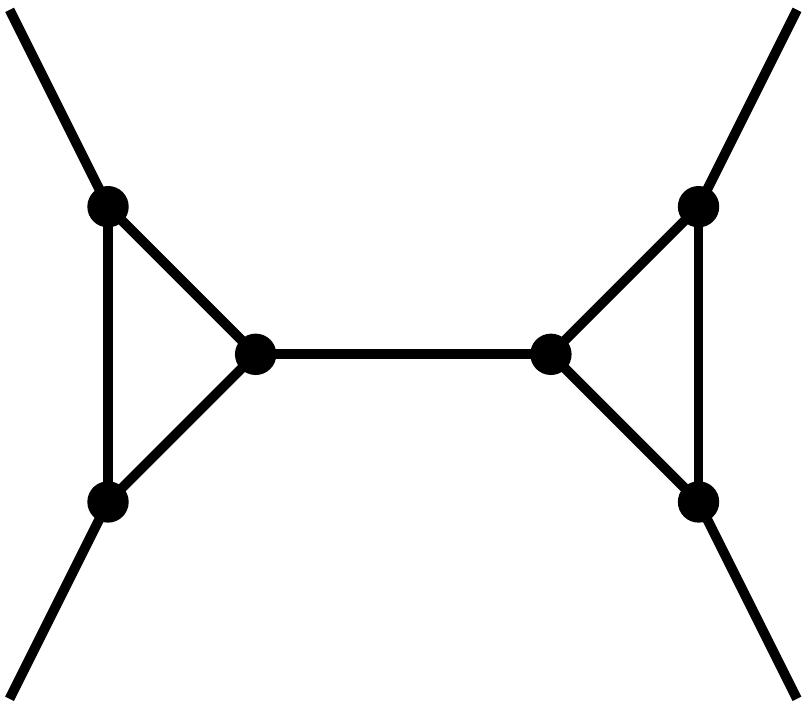}
\end{tabular}

\caption{Pairs of pants decomposition of a punctured Riemann surface and a tropical curve realizing the dual graph}
\label{pair-of-pants}
\end{figure}

There may not be more than $3g-3+k$ disjoint closed embedded geodesics,
but we may also consider collection consisting less than that number.
The result can be thought of as a generalized (or partial) pair-of-pants decomposition. 
Some component of the complement of such a collection are spheres punctures three times
while the others have a greater number of punctured or even a genus.
Thus the 
 conformal 
structure of some component is no longer determined by the lengths
of the boundary geodesic and we need to specify it separately. As we are
preparing the framework for the tropical limit where the hyperbolic lengths
of all these boundary geodesics vanish we need only to consider conformal
structures of finite type (i.e. such that all the end components correspond
to punctures conformally).

Any non-compact Riemann surface $S$ of finite type is obtained from a closed Riemann surface $\tilde S$
by puncturing it in finitely many points. It is convenient to define a compact surface $\bar S$ obtained
by an oriented real blowup of $\tilde S$ at the points of puncture as in  {\cite[Section 6]{MikhalkinOkounkov}}.
\newcommand{\dd}{\partial}
Then each puncture $\epsilon$ of $S$ gets transformed to a boundary component $b_\epsilon\subset\dd\bar S$ that is naturally oriented
as the boundary of $\bar S$.
We refer to the resulting $b_\epsilon\approx S^1$ as the {\em boundary circles} of $\epsilon$.

Similar considerations work in the case of so-called
{\em nodal} surfaces $S$. Topologically, such a surface $S$ is obtained from a (possibly disconnected)
punctured Riemann surface by choosing some number of distinct pairs of points and identifying the points in these pairs.
The points resulting in this procedure are called the {\em nodes}. The surface $S^\circ$ is defined from $S$ by
removing all nodes. Thus in addition to the punctures $S^\circ$ gets new punctures.
A nodal Riemann surface consists of $S$ and the choice of a conformal structure of finite type on $S^\circ$.

We set $\bar S=\bar S^\circ$ by oriented blowup of the closed (possibly disconnected) Riemann surface $\tilde S$
obtained by attaching back a point to each puncture. Each node $\delta$ contributes to two boundary
components $b_\delta',b_\delta''\subset\dd\bar S$ which we call the {\em vanishing boundary circles} of $\delta$.

We call a map $\Phi:S^\circ\to\sonen$ {\em pluriharmonic} if it can be obtained from a holomorphic map $\tilde\Phi:S^\circ\to\ctorn$
by composing it with the argument map $\Arg:\ctorn\to\sonen$.
It is easy to see (cf. e.g.  
{\cite[Section 6.2]{MikhalkinOkounkov}} and  {\cite[Section 2]{Mik08}}) that
any 
pluriharmonic map  
 $\Phi:S^\circ\to\sonen$ 
which is proper in a neighborhood of a  boundary circle $b$
induces a map 
\begin{equation}\label{Phi-epsilon}
\Phi^b:b\to \sonen.
\end{equation}
Its image is a geodesic on the flat torus $\sonen=(\RR/2\pi i)^n$.
The map $\Phi^b$ comes as a limit of the map $\Phi$ restricted to a small simple loop
around the corresponding puncture or node when this loop tends to $b$ (in the Hausdorff topology on closed subsets of $\bar S$).
This map allows us to define a natural translation-invariant metric of circumference $2\pi$ on each such $b$
(cf. \cite{Mik08}).

If a pluriharmonic map $\Phi:S^\circ\to\sonen$ 
has a removable singularity at some
puncture 
with corresponding boundary circle $b$,
then $\Phi^b(b)$ is a point of $\sonen$.
We can treat such point as a degenerate geodesic of slope $0$.
Otherwise we say that $b$ is an \emph{essential} boundary circle for $\Phi$.

Recall that
a tropical morphism $f:C\to\RR^n$ is called minimal if no edge of $C$ is contracted to a point by $f$.
From now on we assume that $f$ is minimal to simplify our definitions (we will only use minimal morphisms
in the applications of this paper, we refer to \cite{Mik08} for general
case). 
\begin{defn}\label{phase-curve}
The \emph{phase $\phi$ of $f$} is the following data
\begin{itemize}
\item a choice of nodal
Riemann surface $\Gamma_v$ 
of genus $g_v$
with $k$ punctures
for each inner vertex $v\in\Ve^0(C)$
of valence $k$;
\item a one-to-one correspondence between the punctures of $\Gamma_v$ and the edges of $C$ adjacent to $v$;
\item an orientation-reversing isometry 
\begin{equation}\label{rho-e}
\rho_e:b_\epsilon^v\approx b^{v'}_\epsilon
\end{equation}
between the boundary circles
of the punctures corresponding to $e$ for any edge $e$ connecting vertices $v,v'\in\Ve^0$;
\item
a pluriharmonic map $$\Phi_{v}:\Gamma^\circ_{v}\to\sonen$$
for each $v\in\Ve^0(C)$, where $\Gamma^\circ_v$ is obtained from $\Gamma_v$
by removing all nodes;
\item 
an orientation-reversing isomorphism 
\begin{equation}\label{rho-delta}
\rho_\delta:b'_\delta\approx b''_\delta
\end{equation}
between the vanishing boundary circles for each node $\delta$ of $\Gamma_v$; 
\end{itemize}
subject to the following properties.
\begin{enumerate}
\item For any edge $e$ adjacent to $v\in\Ve^0(C)$ and  the puncture $\epsilon$ corresponding to $e$ 
 we have the following identity for the homology class $[\Phi_v^{b_\epsilon}(b_\epsilon)]\in H_1(\sonen)=\ZZ^n$
$$[\Phi_v^{b_\epsilon}(b_\epsilon)]=w_{f,e}u_{f,e}\in\ZZ^n.$$
Here $u_{f,e}\in\ZZ^n$ is the outgoing (from $v$) primitive integer vector parallel to $f(e)\subset\R^n$.
\item For any edge $e$ connecting $v,v'\in\Ve^0(C)$ we have 
$$\Phi^{b_\epsilon}_v=\Phi^{b_\epsilon}_{v'}\circ\rho_e:b_\epsilon^v\to \sonen,$$
where $\epsilon$ is the puncture corresponding to $e$.
\item For any node $\delta$ of $\Gamma_v$ 
we have 
$$\Phi^{b_\delta'}_v=\Phi^{b_\delta''}_v\circ\rho_\delta:b'_\delta\to \sonen,$$
where the node $\delta$ is considered as puncture for the components of $\Gamma^\circ_v$
whose closures intersect at $\delta$.
\item Any node $\delta$ of $\Gamma_v$ 
with essential boundary circle 
is adjacent to two distinct connected components of $\Gamma_v^\circ$.
\end{enumerate}
A tropical morphism $f$ equipped with a phase $\phi$ is called the {\em phase-tropical morphism} $(f,\phi)$.
\end{defn}

Note that by the last property  the dual graph of $\Gamma_v$
does not have edges adjacent to the same vertex (loop-edges)
corresponding to a node with essential boundary circle. 
Such case may be equivalently treated via perturbing $v$
into several vertices via inserting a corresponding length 0 edge for
each node of $\Gamma_v$.
The slopes of the new edges are determined by the homology classes of the geodesics $\Phi^{b_\delta'}_v(b_\delta)$
and are allowed to be zero. In the simplification considered in this paper (restricting to minimal tropical morphisms
and coarse phase-tropical limits) we always treat such edges as having zero length, though in a refined version
a length of such edge may be positive if it has a zero slope.

\begin{rem}\label{rem:real phase morphism}
As in the case of points, one has the notion of  \emph{real phase} of
a tropical morphism: a phase is real if there exists a continuous
involution $\sigma:C\to C $ such that $\Phi_v$ has a real algebraic lift
if $\sigma(v)=v$, and $\Phi_v=-\Phi_{\sigma(v)}$ if $v\ne \sigma(v)$.
\end{rem}

\begin{exa}\label{ex:phase}
Consider a tropical curve $C$ that consists of 4 lines emanating from the same point
and a tropical morphism $f:C\to\R^2$ that maps this curve onto the union of the $x$- and $y$-coordinate
axes, see Figure~\ref{cross}. To specify the phase of $f$ we need to choose a conformal structure
on $\Gamma_v$, a sphere punctured 4 times corresponding to the only vertex $v\in C$
and a pluriharmonic map $\Phi_v:\Gamma_v^\circ\to S^1\times S^1$.
If $\Gamma_v$ is irreducible then this defines the phase $\phi$ completely.

\begin{figure}[h]
\centering
\includegraphics[width=8cm,angle=0]{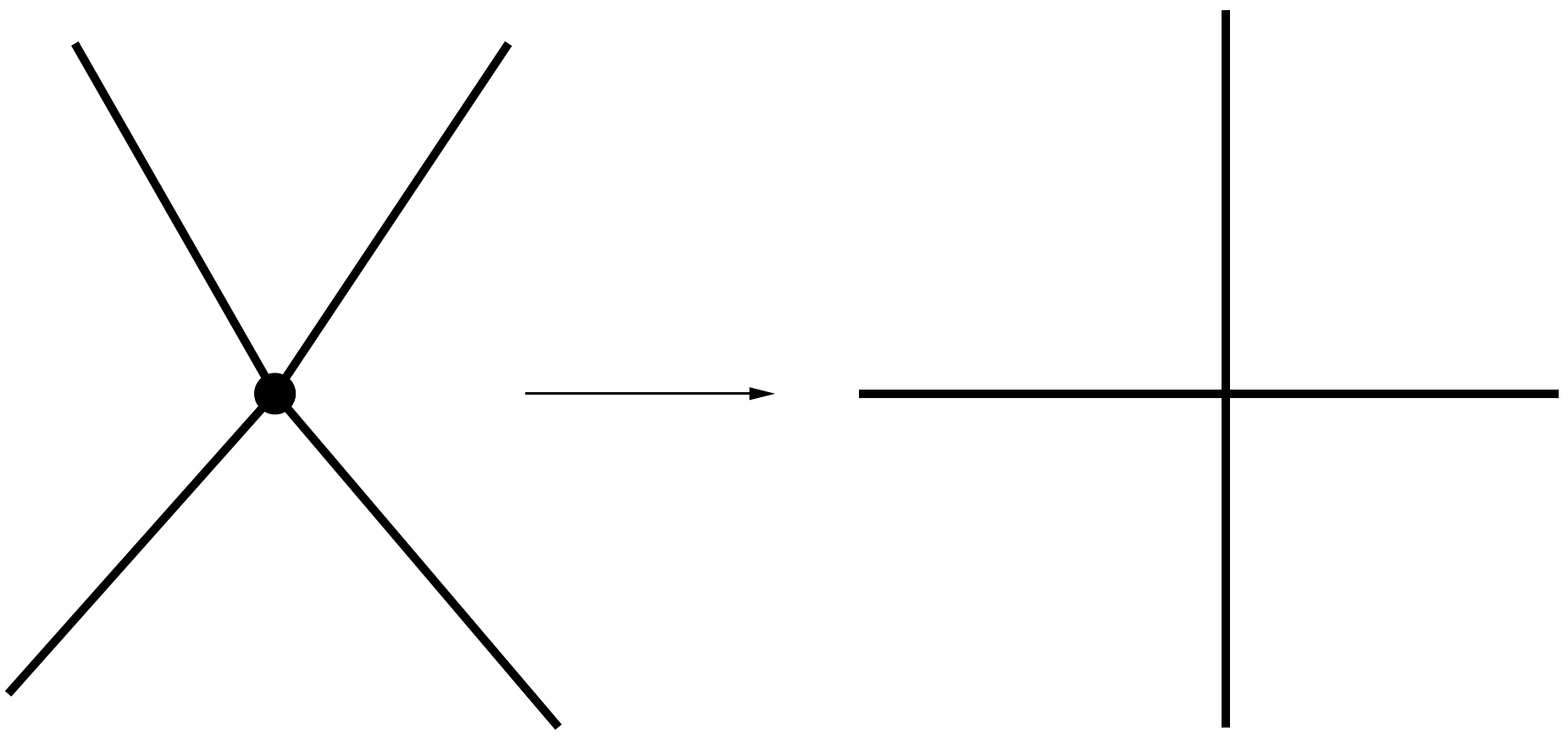}
\caption{Parameterization of a 4-valent node}
\label{cross}
\end{figure}

The choice of a conformal structure on $\Gamma_v$ is the choice of
 an element in $\mathcal M_{0,4}\simeq \mathbb CP^1$, the compactified moduli
 space of rational curves with 4 marked points. Exactly
three of these structures correspond to a reducible surface $\Gamma_v=\Gamma'\cup\Gamma''$ made
of two components intersecting at a node $\delta$. If $\Gamma'$ contains two punctures corresponding
to the horizontal rays of $C$ then the boundary circles $b'_\delta$ and $b''_\delta$ are contracted
by ${\Phi}_v^{b_\delta'}$ and ${\Phi}_v^{b_\delta''}$
and define a point in $S^1\times S^1$. The orientation-reversing isometry $\rho_\delta$ can be chosen arbitrarily.
For all other reducible cases these boundary circles
define closed geodesics on $S^1\times S^1$ which should coincide and thus induce an orientation-reversing
isomorphism $\rho_\delta$.
\end{exa}

Now we are ready to define tropical limits. As usual, it is especially easy to do
for the case of points. For $t>0$ we define the renormalization isomorphism $H_t:\ctorn\approx\ctorn$
(cf. \cite{Mik12}, \cite{Mik1}) by
\begin{equation}\label{Ht}
H_t(z_1,\dots,z_n)=(|z_1|^{\frac{1}{\log t}}\frac{z_1}{|z_1|},\dots,|z_n|^{\frac{1}{\log t}}\frac{z_n}{|z_n|})
\end{equation}
\newcommand{\tj}{t_j}
\newcommand{\Nn}{{\mathbb N}}
Let $p_{\tj}\in \ctorn$, $\tj>0$, $j\in\Nn$ be a sequence of points indexed by a sequence
of real numbers $t_j\to +\infty$.
\newcommand{\Log}{\operatorname{Log}}
\begin{defn}\label{troplim-points}
We say that 
a sequence $p_{\tj}$ {\em converges tropically} if the limit
$\lim\limits_{j\to+\infty} H_{\tj}(p_{\tj})$ exists as a point of $\ctorn$.
The \emph{phase-tropical limit} of this sequence is the point $p$ enhanced with the phase $\phi(p)$, where
$$p=\Log(\lim\limits_{j\to+\infty} H_{\tj}(p_{\tj})),\ \ \ \phi(p)=\lim\limits_{j\to+\infty} \Arg(H_{\tj}(p_{\tj}))
=\lim\limits_{j\to+\infty} \Arg(p_{\tj}).$$
Also we say that the point $p$ itself is the tropical limit of $p_{\tj}$.
\end{defn}
It is easy to see that in this case $p=\lim\limits_{j\to+\infty}\Log_{\tj}p_{\tj},$
where $\Log_t(z_1,\dots,z_n)=(\log_t|z_1|,\dots,\log_t|z_n|),$
since $$\Log_t=\Log\circ H_t.$$
Note that the limit depends not only on the sequence of points $p_{\tj}\in\ctorn$,
but also on the parameterizing sequence $\tj\in\RR$.

Let $f:C\to\R^n$ be a minimal tropical morphism enhanced with a phase $\phi$.
 Let $U\subset\R^n$ be a convex bounded open set with connected intersection $f(C)\cap U$;
 containing not more than a single vertex of $f(C)$
 (this vertex may be an image of more than one vertex of $C$) 
 and such that
its boundary does not contain
any such  vertex of $f(C)$.
Each connected component $W$ of $f^{-1}(U)$ contains not more than a single vertex $v$.
For any such vertex the phase
$\phi$ of $f$ associates to $v$ a Riemann surface $\Gamma_v$ and a
pluriharmonic map $\Phi_v:\Gamma_v^\circ\to\sonen$.
In the case $\Gamma_v$ is smooth,
we let $\Gamma_W=\Gamma_v$ and $\Phi_W=\Phi_v$.

If $\Gamma_v$ is not smooth but has nodal points then we prepare a new smooth surface $\hat\Gamma_v$
by resolving each node, i.e. replacing each node $\delta$ with the boundary circle $b_\delta$ corresponding to
this node, i.e. either side of the isometry \eqref{rho-delta}.
In other words $\hat\Gamma_v$ is obtained from $\bar\Gamma_v$ by identifying all pairs of vanishing boundary circles
with \eqref{rho-delta}.
Naturally we have an inclusion $\Gamma_v^\circ\subset\hat\Gamma_v$
and a surjective map
$$\hat\Gamma_v\to\Gamma_v.$$
that collapses each boundary circle to a point.
Also we have a map $$\hat\Phi_v:\hat\Gamma_v\to\sonen$$
that extends the map $\Phi_v$ from $\Gamma_v^\circ$ to $\hat\Gamma_v$ with the help of the boundary circle isomorphism \eqref{rho-delta}.
In this case we let $\Gamma_W=\hat\Gamma_v$ and $\Phi_W=\hat\Phi_v$.
Note that in this case $\Gamma_W$ is not a Riemann surface in the conventional sense, however 
$\Gamma_W^\circ=\Gamma_v^\circ$ is.

If the connected component $W$ does not contain a vertex then it is an open arc of an edge $e\subset C$.
Let $v$ be an endpoint vertex of $e$.
Consider the boundary circle $b_\epsilon\subset\bar\Gamma_v$ corresponding to the edge $e$.
The map $\Phi_\epsilon:b_\epsilon\to \sonen$ induces a holomorphic map 
$\CC^*\approx b_\epsilon\times\RR\to\ctorn$
by identifying 
 $\CC^*$
 with the tangent space of $b_\epsilon$ and 
 $\ctorn$
 with the tangent space
 of $\sonen$. 
Composing this map with $\Arg$, we obtain  a pluriharmonic map
$\Phi_e:\CC^*\to \sonen$. Note that $\Phi_e(\CC^*)=\Phi_\epsilon(b_\epsilon)$.
We denote $b_\epsilon\times\RR$ enhanced with this complex structure
(of $\CC^*$) with $\Gamma_e$. Note that also we have an embedding of 
 $b_\epsilon$ to $\Gamma_e$
as well as to $\Gamma_v$ for any endpoint $v$ of $e$.
In this case we let  $\Gamma_W=\Gamma_W^\circ=\Gamma_e$ and $\Phi_W=\Phi_e$.

Let $f_{\tj}:C_{\tj}\to\ctorn$ be a sequence of holomorphic maps from Riemann surfaces $C_{\tj}$
of finite type. Note that the inverse image $f_{\tj}^{-1}(\Log_{\tj}^{-1}(U))$ is a Riemann surface with finitely many ends
(as we can enlarge $U$ and extend the map $f|_{f_{\tj}^{-1}(\Log_{\tj}^{-1}(U))}$).

A map $\tau:\ctorn\to\ctorn$ defined by 
$(z_1,\dots,z_n)\mapsto(a_1z_1,\dots,a_nz_n)$ for some $a_1,\dots,a_n>0$,
is called a {\em positive multiplicative translation} in $\ctorn$.

\begin{defn}\label{wtroplimit}
We say that $f:C\to\RR^n$ enhanced with the phase $\phi$ is the
\emph{coarse phase-tropical limit} of $f_{\tj}$
if for any choice of open set $U\subset\R^n$ as above and
all sufficiently large $\tj$ there is a 1-1 correspondence between
connected components
$W_{\tj}$ of $f_{\tj}^{-1}(\Log_{\tj}^{-1}(U))$ 
and connected components $W$ of $f^{-1}(U)$
with the following properties
of the corresponding components.
\begin{itemize}
\item
There exists an open embedding $\Xi^W_{\tj}:W_{\tj}\to\Gamma_W$
and, for each connected component $\Gamma^\circ$ of $\Gamma^\circ_W$, a holomorphic map
$$\tilde\Phi_{\Gamma^\circ}:\Gamma^\circ\to\ctorn$$
such that $\Arg\circ\tilde\Phi_{\Gamma^\circ}=\Phi_W|_{\Gamma^\circ}$
and a sequence of positive multiplicative translations $\tau_{\tj}:\ctorn\to\ctorn$ such that
for any $z\in \Gamma^\circ$ 
$$\lim\limits_{t_j\to +\infty}\tau_{\tj}\circ f_{\tj}\circ (\Xi^W_{\tj})^{-1}(z) = \tilde\Phi_{\Gamma^\circ}(z).$$
In particular, we require that $z\in\Xi^W_{\tj}(W_{\tj})$ for large $\tj$.
\item
For any pair of boundary circles $b$ and $b'$ identified by \eqref{rho-e} (in the case when these
pair correspond to an edge of $C$, i.e. $b$ and $b'$ are boundary circles for $\Gamma=\Gamma_v$ and $\Gamma'=\Gamma_{v'}$
for distinct vertices $v,v'$) or \eqref{rho-delta} (in the case when they are vanishing boundary circles for components
$\Gamma$ and $\Gamma'$ of $\Gamma^\circ_v$ for the same vertex $v=v'$),
any point $z\in b$, any $\eta>0$
and a sufficiently large $\tj$
there exist 
\begin{itemize}
\item a point $z_\eta\in\Gamma$ and a point $z'_\eta\in\Gamma'$;
\item a path $\gamma_\eta\subset\bar\Gamma$ connecting $z_\eta$ and $z$
and a path $\gamma'_\eta\subset\bar\Gamma'$ connecting $z'_\eta$ and $z'=\rho_e(z)$ (see \eqref{rho-e})
such that the diameter of 
$\hat\Phi_v(\gamma_\eta)\subset \sonen$ 
and that of
$\hat\Phi_{v'}(\gamma'_\eta)\subset\sonen$ 
are less than $\eta$
(in the standard product metric on $\sonen$);
\item 
a path $\gamma_{\tj}\subset C_{\tj}$ connecting $(\Xi^v_{\tj})^{-1}(z_\eta)$ and $(\Xi^{v'}_{\tj})^{-1}(z'_\eta)$
such that the diameter of $\Arg(f_{\tj}(\gamma_{\tj}))\subset \sonen$ is less than $\eta$ and
the path $\Log_{\tj}\circ f_{\tj}\circ \gamma_{\tj}$ is contained in a small neighborhood of the interval connecting
$\Log_{\tj}(f_{\tj}(v))$ and $\Log_{\tj}(f_{\tj}(v'))$ in $\RR^n$.
\end{itemize}

\end{itemize}

\end{defn}

In plain words, if our open set  $U\subset\R^n$ is disjoint from $f(C)$ then it should
not intersect $\Log_{t_j}(f_{t_j}(C_{t_j}))$ for large $t_j$. If it intersects $f(C)$ along an interval
that is simply covered by an edge $e\subset C$ 
then $\Log_{\tj}^{-1}(U)\cap f_{t_j}(C_{t_j})$ should be an annulus 
whose image by $\Arg$ is
close to $\Phi_e(\Gamma_e)$.
Similarly, if this interval is covered by several edges then we should see an annulus for each such edge.
Finally, if $U$ contains a neighborhood of a vertex $v\in C$ then 
there should exist holomorphic maps $\tilde\Phi_{\Gamma^\circ}:\Gamma^\circ\to\ctorn$ whose argument
agrees with 
the argument of
the map to which $\Log_{\tj}^{-1}(U)\cap f_{t_j}(C_{t_j})$ accumulates.
All these components should glue in accordance with \eqref{rho-e} and \eqref{rho-delta}.

Altogether we may think of $C_{t_j}$ as decomposed into generalized pair-of-pants 
(cf. the discussion preceeding Definition \ref{phase-curve}) made from pair-of-pants
(or more general Riemann surfaces) close to $\Gamma_v$ that are glued along
embedded circles dual to the edges of $C$. Existence of paths $\gamma_\delta$ ensures
that the gluing of these pairs-of-pants agrees with the orientation reversing isometries specified by the phase $\phi$
in the case when these geodesics represent a divisible class in $H_1(\sonen)$, i.e.
when the weight of the corresponding edge is greater than 1. If this class is primitive then 
this condition holds automatically.

\begin{rem}
In Definition \ref{wtroplimit} we use the term {\em coarse} tropical limit as for simplicity
we ignore the issue of non-minimal morphisms by implicit identification of the limit
with the corresponding minimal morphism.
This procedure can actually be refined (although we do not need it for this paper)
so that the lengths and the phases of the contracted edges will also be well-defined in the limit,
see \cite{Mik08}.
\end{rem}

\begin{exa}
Let us consider the tropical morphism $f:C\to\RR^2$ from Example
\ref{ex:phase}, equipped with a phase
$\Phi_v:\Gamma_v^\circ\to S^1\times S^1$ at the only vertex $v$ of $C$.

We  may present this phase-tropical morphism 
 as a phase-tropical limit of a generalized sequence of holomorphic maps
 $f_t:S_t\to\ctor$ with $S_t$ a Riemann surface of genus $0$ with 4 punctures.
For simplicity we  rather present $(f,\phi)$ as the limit
of  a family of embedded curves in $\ctor$
defined by an implicit equation,  the passage from one point of
view to the other being straightforward. 
Note that by definition of $\Phi_v$, we can
 see
  $\Gamma_v^\circ$ as embedded in $\ctor$ and
$\Phi_v=\Arg_{|\Gamma_v^\circ}$.

If $\Gamma_v$ is irreducible, then 
  $\Gamma_v=\Gamma_v^\circ$ and  we choose $S_t=\Gamma_v$. 

Suppose that now $\Gamma_v$ is reducible and made of two components
$\Gamma'$ and $\Gamma''$ intersecting at a node $\delta$. Both
$\Gamma'$ and $\Gamma''$ are spheres with two punctures, and we may
assume without loss of generality that $\Gamma'$ (resp. $\Gamma''$)
has one puncture corresponding to the left-horizontal (resp. a vertical) ray
of
$C$.

If the boundary circle of $\delta$ has slope $0$, then the
pluriharmonic map $\Phi_v:\Gamma_v^\circ\to  S^1\times S^1$  extends to the
whole curve $\Gamma_v=\Gamma'\cup\Gamma''$.
In this case we can still
assume that  $\Gamma_v\subset \ctor$, and that it is given by the
equation  $(x-x_0)(y-y_0)=0$ with $(x_0,y_0)\in\ctor$. 
 We then
choose
$S_t$ to be the algebraic curve given by the
equation $(x-x_0)(y-y_0)+t^{-1}=0$.

If the boundary circle of $\delta$ has a non-zero slope, then one
cannot extend $\Phi_v$ at the node $\delta$. If the other
puncture of $\Gamma'$ corresponds to the down-vertical ray of $C$,
then $\Gamma'\setminus\delta$ (resp.  $\Gamma''\setminus\delta$)
has an equation of the form $1+ c_1x
+c_2y=0$ (resp. $ c_1x+c_2y +c_3xy=0$) with $c_1,c_2,c_3\in\CC^*$.
In this case we choose $S_t$ to be the algebraic curve given by the
equation $1+ c_1x+c_2y + e^{-t}c_3xy=0$.

In the other case
 $\Gamma'\setminus\delta$ (resp.  $\Gamma''\setminus\delta$)
has an equation of the form $1+ c_1y
+c_2xy=0$ (resp. $ 1+c_2xy +c_3x=0$) with $c_1,c_2,c_3\in\CC^*$, and
 we choose $S_t$ to be the algebraic curve given by the
equation $1+ c_1y+c_2xy + e^{-t}c_3x=0$.

\medskip
In all cases, the phase-tropical morphism $(f,\phi)$ is the
phase-tropical limit of the curves $S_t\subset\ctor$.
\end{exa}

Also it is worth noting that the tropical limit we define here is defined for convergence in compact sets
in $\ctorn$. Thus the limit may have strictly smaller degree than the terms in the sequence.
It is possible to compactify  $\ctorn$ to a toric variety and define convergence there so that the total
degree will be preserved under the limit (some components of the limit will be contained in the toric divisors).
Nevertheless we have the following compactness property once we allow the limit to have smaller degree.

\begin{prop}\label{trop-compact}
Let $f_{t_j}:C_t\to\ctorn$ be a (generalized) sequence of holomorphic curves with $t_j\to +\infty$, $j\to+\infty$.
Then there exists a tropical morphism $f:C\to\RR^n$ and a subsequence of $t_j$ such that
$f$
is the coarse phase-tropical limit of the subsequence.
\end{prop}

\begin{proof}
The proposition is essentially contained in {\cite[Proposition 8.7]{Mik1}}. Indeed, by  {\cite[Proposition 3.9
and Corollary 8.6]{Mik1}} we can extract
from the family $f_t$ a sequence $f_{\tj}$ such that $\Log_{t_j}(f_{\tj}(C_{\tj}))$ converges in the Hausdorff metric
on compact sets in $\R^n$ to a set $A$ that is an image of a tropical curve.

However we need convergence not only to a set in $\RR^n$ but to a tropical morphism $f:C\to\R^n$ whose image is $A$.
Thus we need to 
construct such a morphism.
For that we need to know the vertices and the edges of $C$. 

Let $p=(p_1,\dots,p_n)\in A\subset\R^n$ be a point and $U\ni p$ be 
a convex open set such that the closure $\bar U$ is compact, does not
contain vertices of $A$ 
other than $p$ and its intersection with $A$ is connected.
\newcommand{\ppj}{{p}^{(t_j)}}
Let $\ppj\in \R_{>0}^n$ be a sequence whose tropical limit is $p$, i.e.
such that $\lim\limits_{j\to\infty}\Log_{\tj}{\ppj}=p$.

For each $j$ we have a positive multiplicative translation $$\tau_{\tj}:(z_1,\dots,z_n)\to (\frac{z_1}{\ppj_1},\dots,\frac{z_n}{\ppj_n})$$
that is a biholomorphic  automorphism of the torus $\ctorn$. The projective compactifications of the curves
$$\tau_{\tj}\circ f_{\tj}:C_{\tj}\to\cp^n$$
must contain a converging subsequence. Its limit may be reducible and some of the components
may be contained in the boundary divisors $\cp^n\setminus\ctorn$. 
The restriction of the limit to $\ctorn$ produces a holomorphic map
$\Gamma_{\{\ppj\}}\to\ctorn$
which we can use for $\tilde\Phi_{\Gamma^\circ}$ for every component $\Gamma^\circ$ of $\Gamma_{\{\ppj\}}^\circ$.
The boundary circles of $\Phi_{\Gamma^\circ}$ must be geodesics in $\sonen$ whose homology classes
are proportional to the slope vectors of the
edges of $A$ adjacent to $p$ with positive proportionality coefficients as any non-proportional
class would contradict to the maximum principle.
Indeed, otherwise we choose a codimension 1 subtorus of $\sonen$ parallel to the boundary circle,
but transverse to the corresponding edge adjacent to $p$ in $\R^n$.
Then some multiplicative translate of the corresponding codimension 1 complex subgroup of $\ctorn$ would
have to have a negative intersection point with our curve.

Thus we just need to find sequences ${\{\ppj\}}$ that produce meaningful (in particular, non-empty)
curves $\Log_{\tj}^{-1}(U)$
Note that for all but finitely many $p$ the Riemann surface $\Gamma_{\{\ppj\}}$ is a collection of annuli
if $U$ is sufficiently small, no matter what is the approximation ${\{\ppj\}}$ of $p$.

Indeed, if $\Gamma_{\{\ppj\}}$ has a non-annulus component then by the maximum
principle its Euler characteristic is negative and its image under a character $\chi_{a_1,\dots,a_n}:\ctorn\to\CC^*$,
$(z_1,\dots,z_n)\mapsto z_1^{a_1}\dots z_n^{a_n}$, must have a critical point whenever $(a_1,\dots,a_n)$ is transverse
to the edges of $A$. But the algebraic map $\chi_{a_1,\dots,a_n}\circ f_{\tj}$ may only have finitely many
critical points, so after passing to a (diagonal) subsequence of $\tj$ so that we get a convergence to
infinitely many non-annulus surfaces $\Gamma_{\{\ppj\}}$ we get a contradiction. Since non-annulus irreducible
components of surfaces in $\ctorn$ have negative Euler characteristic we call them
{\em hyperbolic components}.

By passing to even smaller open neighborhoods $U\ni p$ if needed we may ensure that each component of
$\Log_{\tj}^{-1}(U)$ is an annulus except for finitely many $p$ for sufficiently large $\tj$.
Furthermore, once $\tj$ is sufficiently large and $U\ni p$ is sufficiently small we have a
natural one-to-one correspondence between the annuli components of $\Log_{\tj}^{-1}(U)$
and the annuli components of $\Gamma_{\{\ppj\}}$.
In the same time a hyperbolic component of $\Log_{\tj}^{-1}(U)$ may correspond to more than
one hyperbolic component of $\Gamma_{\{\ppj\}}$ as it may converge to a reducible curve.

To reconstruct the limiting curve $C$ and the tropical map $f:C\to A\subset\RR^n$
we take a vertex $v$ for each hyperbolic component 
of $\Log_{\tj}^{-1}(U)$, and we define $f(v)=p$ and $g_v$ as the genus
of this hyperbolic component.
The edges of $C$ are obtained by gluing the corresponding annulus components of $\Log_{\tj}^{-1}(U)$
along paths in $A$. The tropical length on the edges of $C$ comes from the length of the corresponding
edges of $A$ divided by the proportionality coefficient between the homology class of the boundary
circle and the slope vector of the edge of $A$ (a positive integer number).

This procedure gives the limiting tropical morphism $f:C\to\R^n$ for a subsequence $f_{\tj}:C_{\tj}\to\ctorn$.
The limits of hyperbolic components of  $\Log_{\tj}^{-1}(U)$ in the Deligne-Mumford 
compactifications of the corresponding moduli spaces give us (possibly nodal and reducible) curves $\Gamma_v$.
Consider a component $\Gamma^\circ$ of $\Gamma^\circ_v$. Choosing a point $q\in\Gamma^\circ$
and a family of approximating points $q^{(\tj)}$ in the corresponding hyperbolic component of $\Log_{\tj}^{-1}(U)$
we ensure that the curve $\Gamma^\circ_{\{  f_{\tj}(q^{(\tj)}) \}}$ contains $\Gamma^\circ$ as a component.
This defines the map $\tilde\Phi_{\Gamma^\circ}$ for the limiting phase-tropical curve. 
Note that punctures of $\Gamma^\circ$ with
 boundary circle of slope 0 are precisely the removable singularities
 of the map  $\tilde\Phi_{\Gamma^\circ}$. In other words
the boundary circles of a node of $\Gamma$ has
 slope 0, and condition $(4)$ of Definition \ref{phase-curve} is satisfied.
\end{proof}

\subsection{Rational tropical morphisms as tropical
limits}\label{sec:inv trop limit}
Proposition \ref{trop-compact} can be reversed 
in the rational case
-- any phase-tropical rational 
morphism
can be presented as a tropical
limit of a family of holomorphic curves parameterized by a real positive parameter.
Furthermore, we can do this procedure consistently for all phase-tropical curve in
a neighborhood of $(f,\phi)$.
For the purpose of this paper it suffices to consider tropical curves supported on graphs
with no vertices of valence higher than 3. With a slight abuse of terminology
(ignoring the ever-present 1-valent leaves) we call
such curves 3-valent. For the sake of shortness in definitions
we restrict to this case.

Let $f:C\to\R^n$ be a 3-valent rational tropical curve and $\phi$ be its phase.
Choose a reference vertex $v\in C$.
Recall that a {neighborhood} of $f$ is obtained by varying the image $f(v)\in\R^n$
as well as the lengths
of the bounded edges of $C$ while keeping the slope vectors of all edges of $f(C)$ unchanged.
Since $C$ is a 3-valent tree, the image $f(v)$ and the length of all edges define $f$ and so
the space of deformation of $f$ is locally $\R^{b+n}$, where $b$ is the number of the bounded
edges.

To define a {\em neighborhood of $(f,\phi)$ in the space of phase-tropical curves} we take
a neighborhood of $f$ in the space of tropical curves and add to it all phases
obtained by sufficiently small multiplicative translations of $\Phi_v:\Gamma_v\to \sonen$
by $(a_1,\dots,a_n)\in\ctorn$, $|a_1|=\dots=|a_n|=1$ as well as small perturbations of the isometry \eqref{rho-e}
for all bounded edges $e$.
Clearly we get another $b+n$ real parameters.

Let us note that the length of each bounded edge $e$ connecting vertices $v$ and $v'$
has a preferred direction for deformation, say increasing
of its length. Similarly, the orientation-reversing isometry $\rho_\epsilon:b_\epsilon^v\approx b_\epsilon^{v'}$
given by the phase structure $\phi$ also has a preferred direction for deformation. 
Namely, we may compose $\rho_\epsilon$ with a small translation of the circle $b_\epsilon^{v'}$
in the direction coherent with the orientation of $b_\epsilon^{v'}$ (induced from the complex orientation of $\Gamma_{v'}$).
Note that this direction of deformation of $\rho_\epsilon$ stays the same if we exchange the roles of $v$ and $v'$.
We call it {\em the positive twist}.

We can couple translation of $f(v)$ 
in $\R^n$ with $(\arg a_1,\dots,\arg a_n)$ and varying the lengths of $e$ with varying the isometry \eqref{rho-e}
so that the positive twist corresponds to increase of the length. Thus a neighborhood of $(f,\phi)$
can be locally identified with $\R^{b+n}\oplus i\R^{b+n}=\CC^{b+n}$. Recall that $b$ is the number of
bounded edges. For a 3-valent tree $C$ it is equal to the number of ends minus two.

Let us revisit the notion of the degree for a curve and generalize
Definition \ref{degree-dim2} to arbitrary dimension $n$. 
Recall (cf. \cite{Mik-whatis}, \cite{FM}, \cite{Mik08}) that the degree of $f$ can also be defined by the following formula
\begin{equation}\label{d-infty}
d=\sum\limits_{e\in\Ve^\infty(C)}w_{f,e}\ {\max}_{j=1}^n\{0,s_j(e)\},
\end{equation}
where 
$s(e)=(s_j(e))_{s=1}^n$ is the slope vector of the end $e$ in the
outgoing direction.

As the degree is invariant with respect to permutations of the basic directions $-E_1,\dots,-E_n, \sum\limits_{j=1}^nE_j$,
where $(E_j)_{j=1}^n$ is the standard basis for $\R^n$ 
 that are used for the compactification $\tp^n$ we have an alternative formula for computing the degree
\begin{equation}\label{d-k}
d=\sum\limits_{e\in\Ve^\infty(C)}w_{f,e}\ {\max}_{j=1}^n\{s_j(e)-s_k(e),-s_k(e)\}
\end{equation}
that holds for any $k=1,\dots,n$.
It is easy to see that \eqref{d-infty} and \eqref{d-k} are consistent for any curve in $\R^n$ that satisfies 
the balancing condition and that for the case $n=2$ these formulas give the same number as Definition \ref{degree-dim2}.

Note that each end of $C$ has to contribute to at least one of these $n+1$ formulas
for degree. Therefore the maximal number of ends for a minimal curve of degree $d$
is $(n+1)d$ and that in the case when we have $(n+1)d$ ends the slope vectors of the ends are exactly
$-E_1,\dots,-E_n, \sum\limits_{j=1}^nE_j$. Such curves are called {\em generic at $\infty$} in $\R^n$.

Let $U$ be a neighborhood of $(f:C\to\R^n,\phi)$ in the space of phase-tropical curves.
 Denote with $\MM_C$ the space of all rational curves in  $\ctorn$ whose collection of boundary circles realizes the same classes
 in $H_1(\sonen)$ as the
ends of the curve $f$. This means that the homology class of each boundary circle agrees
with the slope vector and the weight of the corresponding end, cf. the first property imposed by Definition \ref{phase-curve}.
\begin{thm}\label{realiz}
For all sufficiently large values $t>>1$ there exists an open embedding $\Lambda_t:U\to \MM_C$
such that for any sequence $f_{\tj}\in\MM_C$, $t_j\to+\infty$, the following conditions are equivalent.
\begin{enumerate}
\item The sequence $f_{\tj}$ converges to a phase-tropical curve $(f',\phi')\in U$ in the sense
of Definition \ref{wtroplimit}.
\item For all sufficiently large ${\tj}$ we have $f_{\tj}\in\Lambda_{\tj}(U)$ and 
$$\lim\limits_{\tj\to+\infty}\Lambda_{\tj}^{-1}(f_{\tj})=(f',\phi')\in U.$$
(Recall that $U$ may be topologically viewed as an open set in $\C^N$ for some $N$.)
\end{enumerate}
\end{thm}

In particular this proposition allows to present any holomorphic curve sufficiently close to $(f,\phi)$
in the $\tj$-framework 
as $\Lambda_{\tj}(f',\phi')$ for a phase-tropical curve $(f',\phi')$ close to $(f,\phi)$.
Proposition \ref{realiz} is true also for curves in higher genus mapping to realizable tropical varieties
under the condition of regularity (which means that the dimension of the deformation space 
of the curve is of expected dimension)
and is proved in a more general case in \cite{Mik08} with some intermediate generalizations proved in \cite{Mik1},
\cite{NishinouSiebert}, \cite{Nishinou}, \cite{Tyomkin}.
However the assumptions we make in this paper imply that the curve is rational and the target variety is a projective space
(even more specifically in dimension 2 but it does not make much difference). This is an especially easy case and
to prove it under these assumptions it suffices to consider the lines in $\PP^n$.

We say that $\pi:\cto^{s-1}\to\ctorn$ is a 
\emph{multiplicatively affine map} if it is obtained by composition
of a multiplicatively linear
map $\cto^{s-1}\to\ctorn$ (given by a $(s-1)\times n$ matrix with integer entries) and an arbitrary multiplicative
translation in $\ctorn$. Note that $\pi$ induces a map
$\pi_{\R}:\R^{s-1}\to\R^n$ such that $\Log\circ\pi=\pi_{\R}\circ\Log$,
and a map $\pi_{\Arg}:(S^1)^{s-1}\to\sonen$ such that 
$\Arg\circ\pi=\pi_{\Arg}\circ\Arg$.
Given a phase-tropical curve $(\tilde f:C\to\R^{s-1},\tilde\phi)$ we may compose it with $\pi$ to obtain
a phase-tropical curve $(f:C\to\R^n,\phi)$ with $f=\pi_{\R}\circ\tilde f$ by setting the phase maps $\Phi_v$
at each vertex $v$ to be the composition maps $\pi_{\Arg}\circ\tilde\Phi_v$.
Note that we may assume that $n<s$, otherwise all the ends of $f(C)$ must be parallel to a 
 $(s-1)$-dimensional 
affine space in $\R^n$, so the whole curve $f(C)$ is contained in such space and we may replace the target with this
smaller-dimensional space 
 $\R^{s-1}$.

We say that a tropical curve $f:C\to\R^n$ is a \emph{line} if it has degree 1 in the natural compactification $\tp^n\supset\R^n$.
A \emph{phase-tropical line} is a phase-tropical curve of degree 1 (i.e.
a tropical line enhanced with any phase structure).

\begin{lemma}\label{lifttoline}
Let $(f:C\to\R^n,\phi)$ be a phase-tropical curve where $C$ is a rational tropical curve with $s$ leaves.
Then there exists a phase-tropical line $(\tilde f:C\to\R^{s-1},\tilde\phi)$ generic at $\infty$ and a multiplicatively affine map
$\pi:\cto^{s-1}\to\ctorn$ such that $\pi\circ(\tilde f, \tilde \phi)=(f,\phi)$.

If $\tilde f_{\tj}:C_{\tj}\to\cto^{s-1}$ is a family of holomorphic curves coarsely converging to $(\tilde f,\tilde\phi)$
then $f_{\tj}=\pi\circ\tilde f_{\tj}$ coarsely converges to $(f,\phi)$.

Furthermore, if $U$ is a small neighborhood of $(f,\phi)$ in the space of phase-tropical curves to $\R^n$
then we can find a locus $\tilde U\ni (\tilde f,\tilde \phi)$ inside the space of deformations of $(\tilde f,\tilde\phi)$
and an isomorphism $U\approx\tilde U$ so that for any $(f',\phi')\in U$ the corresponding point of $\tilde U$
is its lift in the above sense.
\end{lemma}

\begin{proof}
We start by lifting the tropical curve $f:C\to\R^n$ to $\R^{s-1}$. To
do this we 
choose a vertex $v$ in $\Ve^0(C)$ and we
  arbitrarily
associate the outgoing unit tangent vectors
to the $s$ leaves
of $C$ with the $s$ preferred vectors in $\R^{s-1}$: $E_1=(-1,0,\dots,0),\dots,E_{s-1}=(0,\dots,0,-1),E_s=(1,\dots,1)$.
This identification defines slope vectors for the remaining (bounded) edges.
Namely, the tangent vector to a point inside an edge of $C$ can be associated to the sum of the vectors
associated to the leaves of $C$ in the direction of the edge (recall that $C$ is a tree).

Let us choose a multiplicatively affine map $\pi:\cto^{s-1}\to\ctorn$
such that $\pi_{\R}$ sends the vector $E_j$ to the corresponding 
outgoing unit tangent vector
multiplied by its weight.
Then we choose $\tilde f(v)$ to be an arbitrary point of $\pi^{-1}(f(v))$.
The slope vectors of the edges along with the tropical structure on $C$ define $\tilde f:C\to\R^{s-1}$.
To lift the phase we choose an arbitrary pluriharmonic
map $\tilde\Phi_v:\Gamma_v\to (S^1)^{s-1}$ that lifts $\Phi_v:\Gamma_v\to\sonen$.
The isometry \eqref{rho-e} defines the lifts of the phases on the vertices connected to $v$ with a single edge. Inductively
we get the lift of the phase.

Note that the only ambiguities in the choice of $(\tilde f,\tilde\phi)$ is the choice of a point in $\pi^{-1}(f(v))\approx\R^{s-1-n}$
and the translation of the phase $\tilde\Phi_v$ by the corresponding $(s-1-n)$-dimensional subgroup of the torus
$(S^1)^{s-1}$. Taken together these ambiguities form $\cto^{s-1-n}$.
\end{proof}

\begin{lemma}\label{proofforline}
 Theorem~\ref{realiz} holds if $U$ is a small neighborhood of $(f:C\to\R^n,\phi)$,
where $f:C\to\R^n$ is a 3-valent line generic at $\infty$.
\end{lemma}
\begin{proof}
We prove the Lemma 
 by induction on $n$. Note that for $n\le 1$ the lemma holds trivially.
Consider the multiplicative affine map $\pi:\cto^n\to\cto^{n-1}$
$(z_1,\dots,z_n)\mapsto(z_1,\dots,z_{n-1})$. By the induction hypothesis $\pi(f,\phi)$ is realizable.
The map $\pi_{\R}|_{f(C)}$ contracts an end $E\subset f(C)$ adjacent to a leaf (1-valent vertex) $u\in C$.
By the 3-valency assumption the point $\pi(f(u))\in\pi(f(C))\subset\R^{n-1}$ is inside an edge of $\pi(f(C))$.
Thus it can be obtained as a transverse intersection point of $\pi(f(C))$ and a hyperplane $\{x_j=c\}\subset\R^{n-1}$ for some
$j=1,\dots,n-1$. Since $f(C)\subset\RR^n$ has degree 1 so 
 has 
 $\pi(f(C))\in \RR^{n-1}$.
Therefore $$\{\pi(f(u))\}=\pi(f(C))\cap \{x_j=c\}$$
and the intersection in the right-hand side has tropical intersection number 1.

Let $f^\pi_t:C^\pi_t\to \cto^{n-1}$ be a sequence of holomorphic curves that converges to $\pi(f,\phi)$.
Let $f^{\pi,j}_t:C^\pi_t\to\CC^*$ be the $j$th coordinate of $f^\pi_t$.
We are looking for the family $f_t$ converging to $(f,\phi)$ in the form
$$(f^\pi_t, \alpha t^a(f^{\pi,j}_t- \beta t^c)):C^\pi_t\setminus\{f^{\pi,j}_t=\beta t^c\}\to\ctorn,$$
$\alpha,\beta\in\C^*$, $|\alpha|=|\beta|=1$, $a\in\R$. 

We set $a$ to be the maximum of the $n$th coordinate of the contracted edge $E\subset f(C)$.
This maximum is attached at the vertex $v\in C$.
This ensures the convergence of $\Log_t\circ f_t(C^\pi_t\setminus\{f^{\pi,j}_t=\beta t^c\})$ to 
 $f$ 
as we have $$f_t(C^\pi_t\setminus\{f^{\pi,j}_t=\beta t^c\})\subset P_t=\{z_n=\alpha t^a(z_j-\beta t^c)\}$$
and clearly $\Log_t(P_t)$ converges to the tropical hyperplane $P$ given by the tropical polynomial
$``z_n+a(z_j+c)"$ and we have $P\supset f(C)$.

It remains to choose the arguments $\alpha$ and $\beta$ to guarantee convergence at the phase level.
These unit complex numbers are determined by the phase structure $\phi$, namely by its value $\Phi_v:\Gamma_v\to\sonen$.
As $v$ is 3-valent the Riemann surface $\Gamma_v$ is a standard pair-of-pants.
The argument $\beta$ is determined by the boundary circle of $\Phi_v$
of slope vector $-e_n$
while the argument $\alpha$ is determined by the boundary circle of
$\Phi_v$ of slope vector $-e_j$.
(Recall that we already have $\pi\circ\Phi_v$
compatible with the limit of $f_t^\pi$.)
\end{proof}

Theorem \ref{realiz} now follows by combining Lemma \ref{lifttoline} and
Lemma \ref{proofforline} together.

\section{Tangency conditions in the phase-tropical
  world}\label{sec:proof corres}
\subsection{Proof of Theorems \ref{Corres} and \ref{Corres2}}
Note that the {\em classical} number $N_{d,0}(k;d_1,\dots,d_{3d-1-k})$ does not depend on
the choice of configuration of the $k$ points and $3d-1-k$ curves (as long as these constraints
are generic and all the curves we count are regular so that the corresponding enumerative problem is well-defined).
Thus we can take a family of constraints $\P_t,\L_t$ depending on the real parameter $t>1$
to compute the (independent of $t$) number $N_{d,0}(k;d_1,\dots,d_{3d-1-k})$ .

Recall that in the hypothesis of Theorem \ref{Corres} the tropical constraints $\P$, $\L$ are already
fixed and they are in tropical general position.

We choose any family $\P_t$ so that its tropical limit
in the sense of Definition \ref{troplim-points} is $\P$. It is easy to see that we can always make such choice.
 Indeed a point $p_t\in\ctor$ is determined by $\Log_t(p_t)\in\R^n$ and $\Arg(p_t)\in S^1\times S^1$.
Once we choose an arbitrary phase $\phi(p)$ the point $p_t$ is determined by $\Log_t(p_t)=p$
and $\Arg(p_t)=\phi(p)$. We do this for every point $p\in\P$.

Consider a tropical curve $L$ from $\L$. 
 Even though this curve is not
 rational, since it is immersed to $\R^2$, it can still be presented as 
the tropical limit of a 1-parametric real family of complex curves in $\cp^2$,
see \cite{Mik1}. Namely, there exists a phase for $L$  (viewed as an
embedding $L\to\R^2$) in the sense of Definition \ref{phase-curve}
and a family $L_t\subset \ctor$ such that $L$ is the tropical limit of $L_t$
in the sense of Definition \ref{wtroplimit}. 

For construction of $L_t$ in the case of general immersion
we refer to {\cite[Proposition 8.12]{Mik1}}. Note that in the case  when $L\to\R^2$ is an
embedding it is especially easy to construct this approximating
family and this can be done directly by {\em patchworking} \cite{V9}
(see also  {\cite[Chapters 7 and 11]{GKZ}}).

Namely, any embedded tropical curve $L\subset\R^2$ is given by
a tropical polynomial $$F(x,y)=\operatorname{max}\limits_{j,k\in\Z}\{a_{jk}+jx+ky\}
= `` \ \sum\limits_{j,k\in\Z} a_{jk}x^jy^k\ "$$
in two real variables $x,y$,
where $a_{jk}\in [-\infty,+\infty)$ and $a_{jk}=-\infty$ except for some
finitely many values of $(j,k)$ with $j\ge 0$ and $k\ge 0$, see
{\cite[Proposition 2.4]{Mik12}}.
The quotation marks here signify tropical arithmetic operations, and the formula
above can be viewed as the definition of these operations (addition and multiplication).
We set 
$$F_t(z,w)=\sum\limits_{j,k\in\Z}\alpha_{jk}t^{a_{jk}}z^jw^k,$$
$t>1$,
where we choose $\alpha_{jk}\in\C$ with $|\alpha_{jk}|=1$ arbitrarily.
The function $F_t$ is a polynomial in two complex variables $z,w$.
We define the curve $L_t\subset\ctor$ as the zero set of $F_t$.

Recall (see {\cite[Chapter 7]{GKZ}}) that the polynomial $F$ defines a subdivision of the Newton
polygon $$\Delta_F=\operatorname{Convex\ Hull}\{(j,k)\in\Z^2\ |\ a_{jk}\neq 0\}\subset\R^2.$$
This subdivision is defined by projections of the faces of the extended Newton polygon $\tilde\Delta_F$ of
which is the undergraph of $(j,k)\mapsto a_{jk}$, i.e. the set
$$\tilde\Delta_F=\operatorname{Convex\ Hull}\bigcup\limits_{j,k\ |\ a_{jk}\neq0}\{(j,k,t)\ |\ t\le a_{jk}\}.$$
Recall that the tropical polynomial $F$ is smooth if the projection of each finite face of $\tilde\Delta_F$
is a triangle of area $\frac12$ (it is easy to see that this area is the minimal possible for a lattice polygon).

Let $(x_0,y_0)\in\R^2$ be a point. Each open set $U\subset\R^2$ defines a real 1-parametric family 
of open sets in $\ctor$
$$U_t^{(x_0,y_0)}=(\Log_t\circ\tau_t^{(x_0,y_0)})^{-1}(U),$$
where $\tau_t^{(x_0,y_0)}:\ctor\to\ctor$ is a coordinatewise multiplication (i.e. multiplicative translation)
by $(t^{-x_0},t^{-y_0})$. We set $L_t^{(x_0,y_0)}=\tau_t^{(x_0,y_0)}(L_t)$.

By definition of  tropical hypersurfaces we have $(x_0,y_0)\in L$ if and only if there are at least two tropical 
monomials in $F$ that assume the same value at $(x_0,y_0)$
 and have value greater than the other monomials of $F$. 
Note that $L_t^{(x_0,y_0)}$ is defined by a polynomial
$F_t^{(x_0,y_0)}$ which is obtained by multiplying the coefficients of the $x^jy^k$ monomials of $F_t$
by $t^{x_0j+y_0k}$.
Thus if $(x_0,y_0)\notin L$ then 
for sufficiently large $t$ the absolute value of one monomial in $F_t$ is larger than the sum of
the absolute values of all the other monomials
and $U_t^{x_0,y_0}\cap L_t=\emptyset$ for any bounded open set $U\subset\R^2$.

Similarly, if $U$ is bounded, $t>>1$ is large and $(x_0,y_0)\in L$ we have
more than one dominating monomial for $F_t^{(x_0,y_0)}$
(after division by an appropriate power of $t$).
Furthermore, if $L$ is smooth all indices $(j,k)\in\Delta_F$ of the dominating
monomials for $F_t^{(x,y)}$ must be contained in a lattice triangle of area $\frac12$.
Thus the convex hull  
 $\Delta_{(x_0,y_0)}$ of such indices is either such a triangle
itself or one of its sides which is an interval of integer length 1.

Consider the {\em truncation} 
$$F_{\Delta_{(x_0,y_0)},t}^{(x_0,y_0)}(z,w) =\sum\limits_{(j,k)\in\Delta_{(x_0,y_0)}} t^{jx_0+ky_0}\alpha_{jk}t^{a_{jk}}z^jw^k,$$
cf. \cite{V9}.
Clearly in $U_t^{(x_0,y_0)}$ the polynomial $F_t^{(x_0,y_0)}$ is a small perturbation of
$F_{\Delta_{(x_0,y_0)},t}^{(x_0,y_0)}$ for large $t$.
However if $L$ is smooth so is the hypersurface $L_{\Delta_{(x_0,y_0)},t}^{(x_0,y_0)}$ defined by $F_{\Delta_{(x_0,y_0)},t}^{(x_0,y_0)}$.
Thus $L_t^{(x_0,y_0)}\cap U_t^{(x_0,y_0)}$ is a small perturbation of  $L_{\Delta_{(x_0,y_0)},t}^{(x_0,y_0)}\cap U_t^{(x_0,y_0)}$.

\begin{proof}[Proof of Theorem \ref{Corres}]
Consider the constraints $(\P_t,\L_t)$. As in the very beginning of
this paper, we set $\mathcal S_t=\S_t(d,0,\P_t,\L_t)$ to be
the set of all degree $d$ genus 0 curves that are passing through
$\P_t$ and tangent to $\L_t$ 
(see Section \ref{intro}).
By Proposition \ref{trop-compact} there exists a tropical limit for a (generalized) subsequence of $\S_t$.
We note that it must be a curve from $\S^{\T}(d,\P,\L)$. To see this we assume 
$f:C\to\R^2$ is such a limit. Since $\P_t\subset f_t(C_t)$ for some $f_t:C_t\to\ctor$ from $\S_t$
we see that $\P\subset f(C)$. Let us assume that there exists a line $L\in\L$ such that $f:C\to\R^2$ is not
pretangent to it. In such case every intersection of $L$ and $f(C)$ is disjoint from the vertices and
thus contained inside the edges of both $L$ and $C$. Therefore each intersection points $p\in f(C)\cap L$ must be transverse
intersections of the edges of $C$ and $L$. But this means that for sufficiently large $t$ the intersections of
$L_t$ and $f_t:C_t\to\ctor$ in $\Log_t^{-1}(U)$ are also transverse for any bounded $U\ni p$. But therefore
all intersection points of $L_t$ and $f_t(C_t)$ are transverse and we get a contradiction.

Thus any accumulation point of $\S_t$ when $t\to+\infty$ must be contained in $\S^{\T}(d,\P,\L)$ which is a finite
set of 3-valent curves by Proposition \ref{generic curves}. In turn, Proposition \ref{realiz} describes
all curves that have a chance to converge to an element $f:C\to\R^2$ from $\S^{\T}(d,\P,\L)$ through
its neighborhood $U$ in the space of all deformations of $f$ in the
class of phase-tropical morphisms.
Thus it suffices to describe those curves from $\Lambda_t(U)$ that pass through $\P_t$ and are tangent to $\L_t$.
We do it below.

Consider a small neighborhood $U$ of $f$ in the space of tropical curves. Such a neighborhood
itself consists of 3-valent curves $f'\in U$, $f':C'\to\R^2$. As we have already seen,
the curve $f'\in U$ is parameterized
by $\R^2\times\R^b$, where $b$ is the number of bounded edges of $C$, once we fix a reference vertex $v\in C$.
In these coordinates we define the curve $f':C'\to\R^2$ to be such curve that $C'$ is isomorphic to the curve $C$ as a graph,
the first $\R^2$ coordinate corresponds to $f'(v)-f(v)\in\R^2$ and the remaining coordinates correspond to the
difference in lengths of the corresponding edges of $C'$ and $C$. Note that in these coordinates $f$ corresponds
to the origin of $\R^2\times\R^b$.

These coordinates are naturally coupled with the coordinates responsible for the phase
structure. For this we choose an arbitrary reference phase-structure $\phi$ for $f$.
Namely, the first $\R^2$-coordinates are coupled with the $(S^1\times S^1)$-coordinates
of the arguments of the same coordinates
in $\ctor$. Each bounded edge defines a parameter for identifying the corresponding
boundary circles. This coordinate couples with the $\R$-coordinate in $\R^2\times\R^b$ corresponding 
to the same edge so that the increment in length corresponds to the positive direction of the twist.
Together the coupled coordinates form $\C^*$. Thus the space of all phase-tropical
curves corresponding to tropical curves from $U$ is $$\U=U\times (S^1\times S^1)\times (S^1)^b$$
that can be considered as a subspace of $\ctor\times(\C^*)^b$ obtained as $\Log_t^{-1}(U)$
for arbitrary $t>1$.

By Lemma \ref{lem1} each constraint $q\in\P\cup\L$ imposes a condition defining
a hyperplane (in the classical sense)  $\Lambda_p$ 
with a rational slope in $U\subset\R^2\times\R^b$.
Similarly, the space of all phase-tropical curves that can be approximated by classical 
curves passing through $p_t$ (in the case when $q=p\in\P$) or tangent to $L_t$ (in
the case when $q=L\in\L$) is defined by a subtorus $N_q\subset (S^1\times S^1)\times (S^1)^b$
of the same slope as shown below.

Consider a point $p\in\P$. The condition that $f'$ passes through $p$ is a linear condition on $f'$ in $\R^2\times\R^b$.
It defines a hyperplane (in the classical sense) with a rational
slope. 
The phase $\phi(p)$ must be contained in 
$\phi(e)=\Phi^\epsilon_v(b^v_\epsilon)\subset S^1\times S^1$, where $e$ is an edge of $C$ passing through $p$.
This imposes a linear condition in $(S^1\times S^1)\times (S^1)^b$.
Note that there might be several edges of $C$ passing through $p$ but Proposition \ref{generic curves}
ensures that all such edges have the same slope vector (as they are connected by a chain in the same line).
As 
 $\phi(e)$ 
 is an annulus covered $w_{f,e}$ times by $b^v_\epsilon$, the point $\phi(p)$ is covered
$w_{f,e}$ times.

Let us choose a straight line $G_p$ in $U$ that is transversal to $\Lambda_p$ in the sense that
the primitive integer vector parallel to $G_p$ forms an (integer) basis of $\Z^2\times \Z^b$ when taken with
some integer vectors parallel to $\Lambda_p$. Such a line can be equipped with a phase that is
a geodesic of the same slope vector in  $(S^1)^2\times (S^1)^b$. 
Together with the phase the line $G_p$ defines a proper annulus $\G_p\subset\U$.
The image $\Lambda_t(\G_p)$ contains some holomorphic curves passing through $p_t$.
Their number is determined by the edges of $C$ passing through $p\in\R^2$.

We claim that this number is equal to the sum of the weights of all edges of $C$ containing $p$.
Indeed, this number does not depend on the choice of $\G_p$ by topological reasons. Thus
we may choose for $\G_p$ the phased straight line obtained by multiplicative translation of $(f,\phi)$ by
a given $\C^*$-subgroup of $\ctor$. The points of $\Lambda_t(\G_p)$ passing through $p_t$
will correspond to the points of $b^v_\epsilon$ covering $\phi(p)$.
 
Consider $L\in\L$. First we suppose that the pretangency set consists of a single point $v\in\R^2$.
If $v$ is a vertex of $L$ we look at 
the \emph{coamoeba} of $L$ at $v$, i.e.
the closure $\phi(v)\subset S^1\times S^1$ of
$\Arg(L^v_{\Delta_v})$. 
As the curve $L$ is smooth its coamoeba is an image
of the coamoeba of a line in $\ctor$ under a linear automorphism
of the torus $S^1\times S^1$. This coamoeba consists of two triangles of equal area cut
by three geodesics in the torus whose slope vectors coincide with the
slope vectors  of the edges of $L$
adjacent to $v$. The two triangles share their three vertices, see Figure~\ref{coamoeba}.
We call these three vertices the {\em coamoeba vertices}.

\begin{figure}[h]
\centering
\begin{tabular}{ccc}
\includegraphics[width=4cm,angle=0]{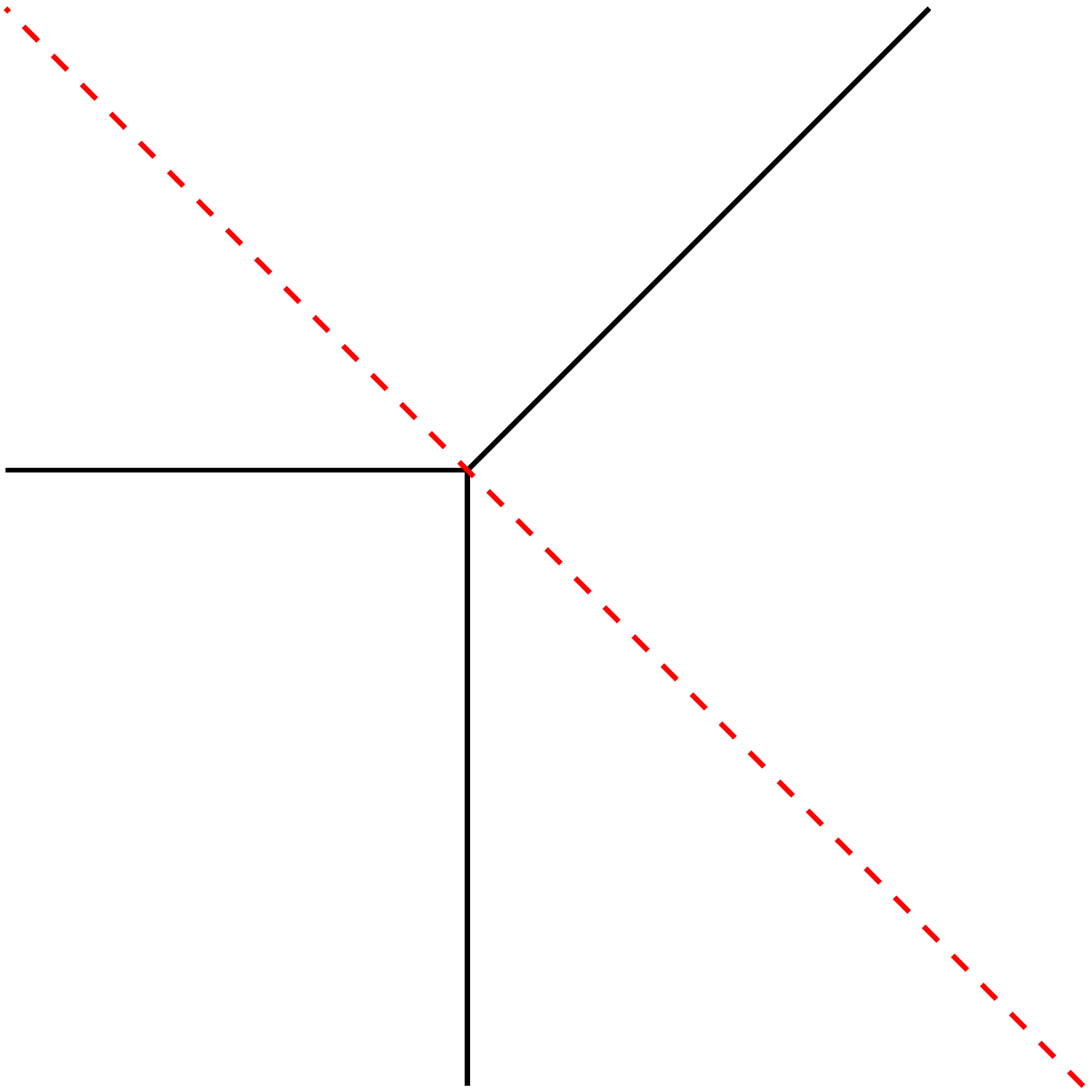}
& \hspace{8ex} &
\includegraphics[width=4cm, angle=0]{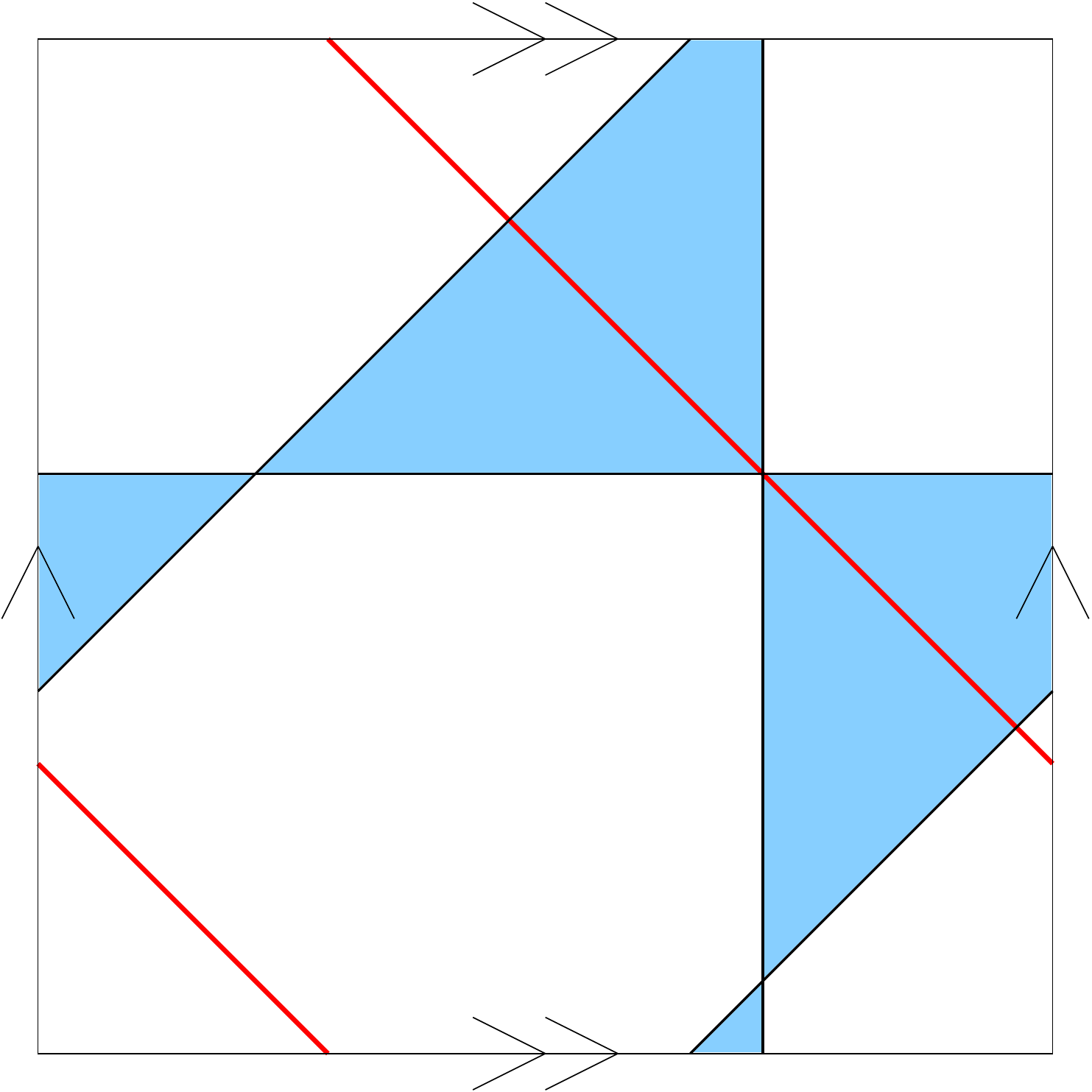}
\\ a) && b)
\end{tabular}

\caption{Two tangent tropical curves and their phases}
\label{coamoeba}
\end{figure}

The interior of the triangles of the coamoeba correspond to the the interior
of amoeba of 
$L^v_{\Delta_v}$, 
see e.g. the Theorem of \cite{Passare}. In turn the logarithmic
Gauss map (i.e. the map taking each point of $L$ to the slope of its tangent space after applying a branch
of the holomorphic logarithm map) 
on this interior takes imaginary values, see  {\cite[Lemma 3]{Mik-00}}, while the real values are
assumed on the boundary of the amoeba of the line. But the three arcs of the boundary of the
amoeba are contracted to the coamoeba vertices while three points where the logarithmic
Gauss map takes values equal to the slope vector of the edges adjacent to $v$ are blown up
to the three geodesics on the coamoeba.

But if the pretangency set is a point this means that any edge $e$ of $C$ that contains $v$
cannot be parallel to one of the edges adjacent to $v$. Therefore $\phi(e)$ and $\phi(v)$ must
intersect in a coamoeba vertex $u$, otherwise $C_t$ and $L_t$ cannot be tangent for large $t>>1$
as there are no nearby points with the same value of the logarithmic Gauss map.
Furthermore, locally near $u\in S^1\times S^1$, $\phi(e)$ must be
contained in the coamoeba
$\Arg(L^v_{\Delta_v})$ 
as only this arc contains the real value of the logarithmic
Gauss map corresponding to the slope vector of $e$. This completely determines $\phi(e)$ as well
as makes a linear condition on the space of phases for $\Lambda_q\subset U$.

As before we choose a straight line $G_L$ in $U$ transversal to $\Lambda_L$
and a phase for $G_L$ that forms $\G_L\subset\U$.
The image $\Lambda_t(\G_L)$ contains some holomorphic curves tangent to $L_t$.
By the same argument as before their number is the sum of the weights of all edges of $C$ passing through $v$.

The situation is similar if $v$ is an image of a vertex $\tilde v$ of $C$
and $f$ is an immersion near $\tilde v$. In this case the coamoeba of
the corresponding edge of $L$ must pass through one of the three coamoeba vertices of the phase of $\tilde v$,
the one determined by the slope vector of the edge $e$ of $L$ containing $v$. As we assume that $L$ is smooth
we have the weight of $e$ as well as the weight of $\Lambda_L$ equal to 1.

This reasoning can be easily modified to include the case when $f$ is not necessarily an immersion near $\tilde v$.
Consider the linear projection $\R^2\to\R$ such that its kernel is parallel to $e$. 
The exponentiation of this map gives us a multiplicatively-linear map
$\pi:\ctor\to\C^*$.
Let $\Psi_{\tilde{v}}:\Gamma^\circ_{\tilde v}\to\ctor$ be a
holomorphic map such that $\Phi_{\tilde{v}}=\Arg\circ\Psi_{\tilde{v}}$
(see Definition~\ref{phase-curve}). 
By the Riemann-Hurwitz formula,
the holomorphic map $\pi\circ
\Psi_{\tilde{v}}:\Gamma^\circ_{\tilde v}\to\C^*$
has a unique ramification point $r\in\Gamma^\circ_{\tilde v}$ (recall that as the vertex $\tilde{v}$ is 3-valent
the surface $\Gamma^\circ_{\tilde v}$ is a pair of pants). Note that $\Phi_{\tilde{v}}(r)$ must be contained
in the phase-boundary circle corresponding to $e$ in order for the corresponding approximation curves to be tangent.

Let $e\subset C$ be any edge adjacent to $\tilde v$.
Varying the phase structure of $C$ by slightly changing the orientation-reversing isometry $\rho_e$ (from Definition~\ref{phase-curve})
and applying $\Lambda_t$ (from Proposition~\ref{realiz}) for large but finite values of $t$ we have a unique curve tangent to $L_t$,
so that the weight of $\Lambda_L$ is again 1.

Applying Proposition~\ref{practical comp} inductively we see that the determinant $\det(\Lambda_{p_1},\ldots,
\Lambda_{p_k},\Lambda_{L_1},\ldots, \Lambda_{L_{3d-1-k}})$ computes the number of 
different phase structures satisfying to the phase tangency conditions at $v$. However the same
phase structure is counted several times, once for every automorphism of our tropical curve $f:C\to\R^2$
as inducing a phase structure by an automorphism of $f$ gives an isomorphic phase structure.
Thus we have to divide the result by $|Aut(f)|$.

Finally we consider the general case when the pretangency components are not necessarily points.
Let $E_L$ be a pretangency component. This component is a tree
(as it is contained in the tree $C$). Furthermore, the coamoeba of this component (by which we mean the
union of the argument of phases of all vertices and edges contained in $E_L$) is a circle 
which we denote $\phi(E_L)\subset S^1\times S^1$. Indeed all edges adjacent to a vertex of $E_L$ 
are parallel to the same line therefore the only possible change in coamoeba of an edge in $E_L$
is a number of times it runs through the same circle $\phi(E_L)$ while the number of times coincides
in  turn with the weight of the edge. The same can be said about the coamoeba of the part of $L$
corresponding to the same pretangency set. To obtain tangencies for $L_t$ and $C_t$ these two 
coamoebas must coincide.

To compute the weight of $\Lambda_L$ we consider again the properly embedded annulus $\G_L\subset\U$
obtained as a phased straight line in an integer direction transversal to $\Lambda_L$ 
as well as its image under $\Lambda_t$ for large $t>>1$.

To compute 
the number of curves in $\Lambda_t(\G_L)$ tangent to $\L_t$
we prepare an auxiliary phase tropical curve from $(f,\phi)$ and $E_L$. Namely we take the
subtree $C_L\subset C$ formed by the closed edges of $C$ intersecting $E_L$ and continue
the resulting new 1-valent edges to infinity.
This defines $f_L:C_L\to\R^2$.
The phase $\phi$ induces a phase $\phi_L$ for $f_L$.

Proposition~\ref{realiz} produces a real 1-parametric family of complex curves $C_t^{(E_L)}$ whose
tropical limit is $f_L:C_L\to\R^2$. As $\phi_L$ is induced by $\phi$
we can compute the weight of $\Lambda_L$ with the help of $C_t^{(E_L)}$ and a 1-parametric family
obtained by multiplicative translation in a direction transversal to $\Lambda_L$.

Furthermore, in this computation we can replace the family $L_t$ with the family $L_t^{(E_L)}$
obtained in a similar way as $C_t^{(E_L)}$. Namely we take the subtree $E_L\subset L$
formed by the closed edges of $L$ intersecting $E_L$ and continue the new 1-valent edges to infinity.
Denote the resulting (rational immersed) tropical curve by $L_L$.
Then we take the approximating family $L_t^{(E_L)}$ provided by Proposition~\ref{realiz}
for the phase structure induced by that of $L$.

Then the number of tangencies of $C_t^{(E_L)}$ and
$\tau_\lambda(L_t^{(E_L)})$, 
the multiplicative translation of $L_t^{(E_L)}$ by $\lambda\in\G_L$, 
can be computed by Euler characteristic calculus as follows. 
Projection of $C_t^{(E_L)}$ and $L_t^{(E_L)}$ along $G_L$ to $\C^*$ allows one
to define the fiber product of  $C_t^{(E_L)}$ and $L_t^{(E_L)}$ which we denote by $A_t$.
As there exists infinitely many directions $\Z$-transversal to $E_L$,
we may assume that the direction of $G_L$ is not parallel to any edge of $L$ or $C$.

Note that minus the Euler characteristic of $A_t$ equals to the number of tangencies between
$C_t^{(E_L)}$ and $\tau_\lambda(L_t^{(E_L)})$, $\lambda\in\G_L$ plus a correction term $\delta$ at infinity
by the Fubini theorem
for the calculus based on the Euler characteristic, see  
{\cite[Theorem 3.A]{Viro-Euler}}.
The correction at infinity is computed for each end of $L$ contained in the (classical) line $D\subset\R^2$ extending the pretangency set $E_L$.
Whenever $L$ has an end contained in $D$ we add to $\delta$ the number of the ends of $C_L$ contained in $D$
and going in the same direction.
Since there are two possible infinite directions in $D$, $\delta$ is the sum of two possible corrections. 

Indeed,
unless the point of the target $\C^*$ is an image of a tangency point for some $\lambda\in\G_L$,
the number of inverse images in the projection $A_t\to\C^*$ equals to the product of the degrees
of the projections $C_t^{(E_L)}\to \C^*$ and $L_t^{(E_L)}\to\C^*$. Thus the Euler characteristic
of $A_t$ equals to a multiple of $\chi(\C^*)=0$ minus the number of the tangency points in the family
parameterized by $\lambda$.

The Euler characteristic of the fiber product $A_t$ for large $t>>1$ can be computed from $C$ and $L$.
Indeed, each vertex $v$ of $L$ or $C$ (resp.) that belongs to $D$ gives a contribution to this Euler characteristic of
the fiber product $A_t$. The contribution is equal to the degree of projection of $C_t^{(E_L)}\to \C^*$ or
$L_t^{(E_L)}\to\C^*$ (resp.). In turn this degree can be computed from $f:C\to\R^2$ and $L$ (resp.) and
it is equal to the number of intersections of the edges $E$ of $C_L$ and $L_L$ (resp.) and the line $G_L^v$ parallel to $G_L$
and passing through $v$. Each point of $E\cap G_L^v$ contributes the corresponding tropical intersection number,
i.e. the absolute value of the determinant of the matrix formed by a primitive integer vector parallel to $E$ and a primitive
integer vector parallel to $G_L$ multiplied by the weight of $E$.

Note that if $E$ is not contained in $D$ the corresponding contribution can be excluded from the Euler characteristic
by passing to a smaller surface $C_t^{(E_L)}$ or $L_t^{(E_L)}$ (resp.) by taking an intersection with
$\Log_t^{-1}(V)$ for a small neighborhood $V\supset D$. Indeed this contribution corresponds to
tangencies with $\tau_\lambda(L_t^{(E_L)})$ for large $\lambda$ for large $t>>1$ and $\lambda$ cannot
tend to zero when $t\to+\infty$. Summarizing all contributions together we get $w_L$ from Section 3 as
the sum of the weights of the vertices $v$ of $L$ such that $v\in C$, where each weight equals to the sum of the weights
of $C_L$ passing through it plus the number of vertices of $C$ that are mapped on $L$ (as $L$ is embedded
and all its weights are 1) minus the contribution $\delta$ at infinity.

To conclude the proof we note that the intersection number of proper submanifolds
$$\Lambda_q\times N_q\subset \ctor\times (\C^*)^b=(\R^2\times \R^b)\times ((S^1\times S^1)\times (S^1)^b)$$
for all $q\in\P\cup\L$ is determined by their homology classes with closed support and thus
coincides with the corresponding tropical intersection number as both coincide with
the same intersection numbers in $H_*((S^1\times S^1)\times (S^1)^b)$. Thus each $f\in\S^{\T}(d,\P,\L)$ contributes
$\mu_{(\P,\L)}(f)$ to the Zeuthen number.
\end{proof}

\begin{proof}[Proof of Theorem~\ref{Corres2}]
To prove Theorem \ref{Corres2} we need to construct a suitable configuration of
immersed complex curves $L_j$ of genus $g_j$ starting from our configuration $L_j^{\mathbb T}$.
Such construction can be provided by Proposition \ref{realiz} (with large value of $t$) 
once we equip each tropical curve $L_j^{\mathbb T}$ with the phase structure in
the case when the tropical curves $L_j^{T}$ are rational, i.e. all $g_j=0$.
In the general case a family of immersed complex curves $L_j^t$ for large $t$ 
converging to $L_j^{\T}$ in the sense of Definition \ref{wtroplimit}
is provided by the proof of Theorem 1 of \cite{Mik1}, more precisely by Proposition 8.23.

The rest of the proof is similar to the proof of Theorem \ref{Corres}.
The only difference is that 
the constraints $L_j^{\T}$ no longer have to be smooth near their 3-valent vertices.
But, since the curve is immersed there is always a linear map 
${\mathbb R}^2\to{\mathbb R}^2$
such that near that vertex $L_j^{\T}$ is the image of a tropical line in ${\mathbb R}^2$.
Here the determinant $m$ of this linear map is the {\em multiplicity} of our vertex,
cf.  {\cite[ Definition 2.16]{Mik1}}. The exponent of this map 
is a multiplicatively-linear map $M:({\mathbb C}^*)^2\to ({\mathbb C}^*)^2$ of degree $m$.
Note that the genus of the part of $L_j^t$ approximating
$L_j^{\T}$ is zero by {\cite[Proposition 8.14]{Mik1}}.
Taking the pull-back by $M$ we reduce the problem to the case of smooth
curves which is already considered in the proof of Theorem \ref{Corres}.
\end{proof}

\subsection{Enumeration of real curves}\label{sec:real enumerative}
As already mentioned, the proof of Theorem \ref{Corres}
 establishes a
correspondence between phase-tropical curves  and complex curves 
close to the tropical limit. In
particular, if we choose real phases for all constraints in
$(\P,\L)$, it is possible to recover all real algebraic curves passing
through a configuration of real points and tangent to a configuration
of real lines when these points and lines are close to the tropical
limit.

\begin{exa}
Let us revisit Example  \ref{Exa:conic} from a real point of view.
For example, if all the three points in Example \ref{Exa:conic} have
phase $(1,1)$, then the tropical curve in Figure
\ref{Fig:conic-with-constraints} ensures that there exists a
configuration of three points and two lines in $\RR P^2$ such that
all four conics  passing through these points and tangent to these lines
are real.
 On the opposite, if the middle point has phase $(-1,1)$ and the two
 other points have phase $(1,1)$, then there exists a corresponding configuration of
 three points and two lines in $\RR P^2$ such that none of the four
 conics passing through these points and tangent to these lines are
 real.
\end{exa}

\begin{exa}\label{ex real conic}
One can interpret in tropical terms the method used in \cite{RoToVu}
to construct a configuration of 5 real conics such that all 
3264  conics tangent to these 5  conics are real. The main step in
this construction is to find 5 real lines $L_1,\ldots,L_5$ in $\RR
P^2$
and 5 points
$p_1\in L_1,\ldots, p_5\in L_5$ such that for any set
$I\subset\{1,\ldots,5\}$, all the conics passing through the points
$p_i$, $i\in I$, and tangent to the lines $L_j$, $j\in
\{1,\ldots,5\}\setminus I$ are real. 
As in \cite{RoToVu}, let us
start with the configuration depicted in Figure
\ref{3264}a, whose tropical analog is depicted in Figure \ref{3264}b
(without phase) and \ref{3264}c (equipped with the appropriate real
phases). 
Next, we perturb the double lines $L_i^2$ as depicted in
Figure \ref{3264}d 
(without phase, the cycle defined by the image is a twice a line)
 and \ref{3264}e (equipped with the appropriate real
phases). Then
 there exist 5 families of
 real conics converging 
 to our 5 phase conics and producing 3264 real conics as in
 \cite{RoToVu}.

It would be interesting to explore the possible numbers of real conics
tangent to 5 real conics, in connection to \cite{Wel5} and \cite{Ber3}. In
particular, does there exist a configuration of 5 real conics, any one of which lying
outside the others, such that exactly 32 real conics are tangent to
them?
\end{exa}
\begin{figure}[h]
\centering
\begin{tabular}{ccc}
\includegraphics[height=5cm, angle=0]{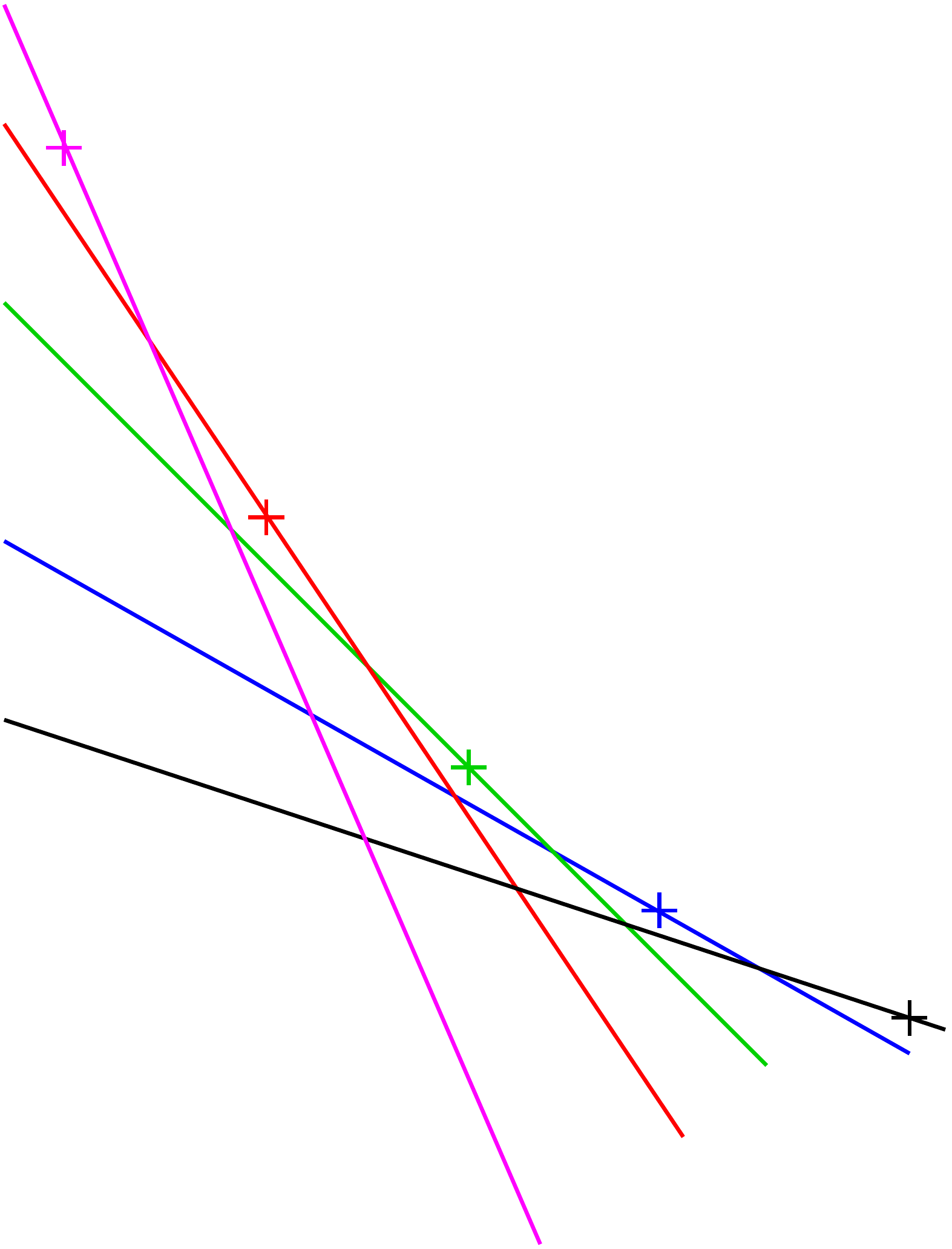}& \hspace{3ex}&
\includegraphics[height=4cm, angle=0]{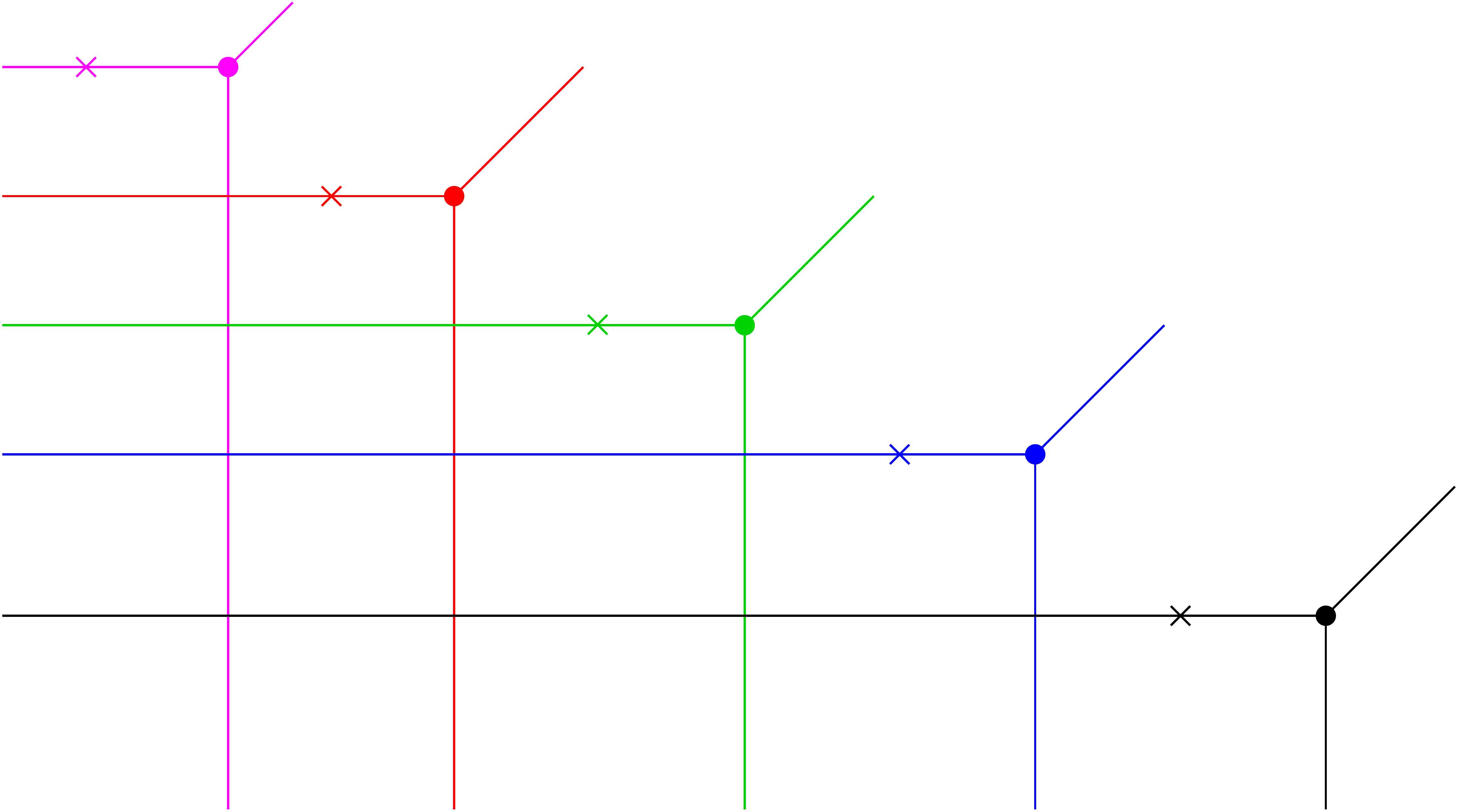}
\\
\\ a) &&b)
\\
\\\includegraphics[height=4cm, angle=0]{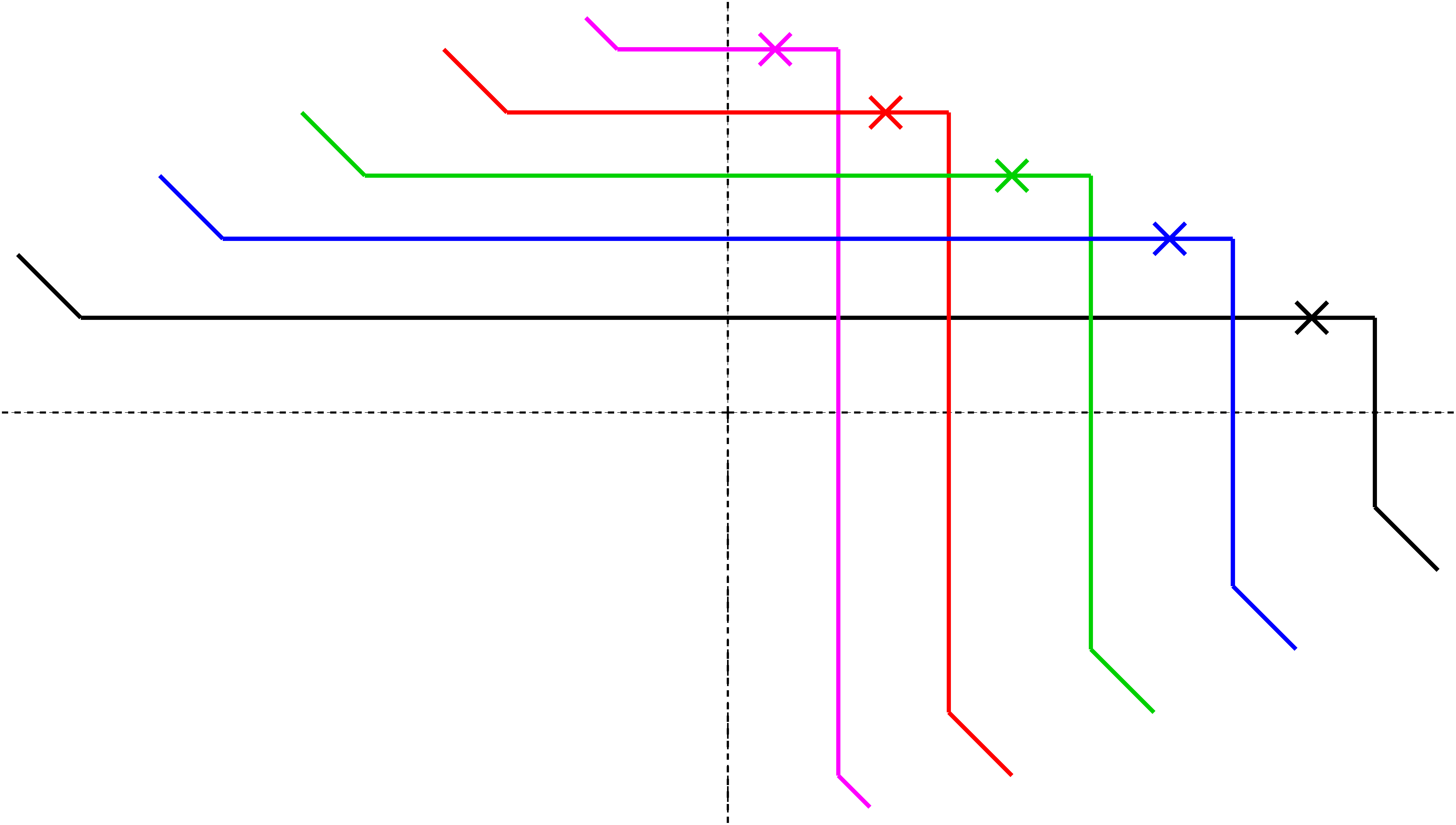}& &
\includegraphics[height=4cm, angle=0]{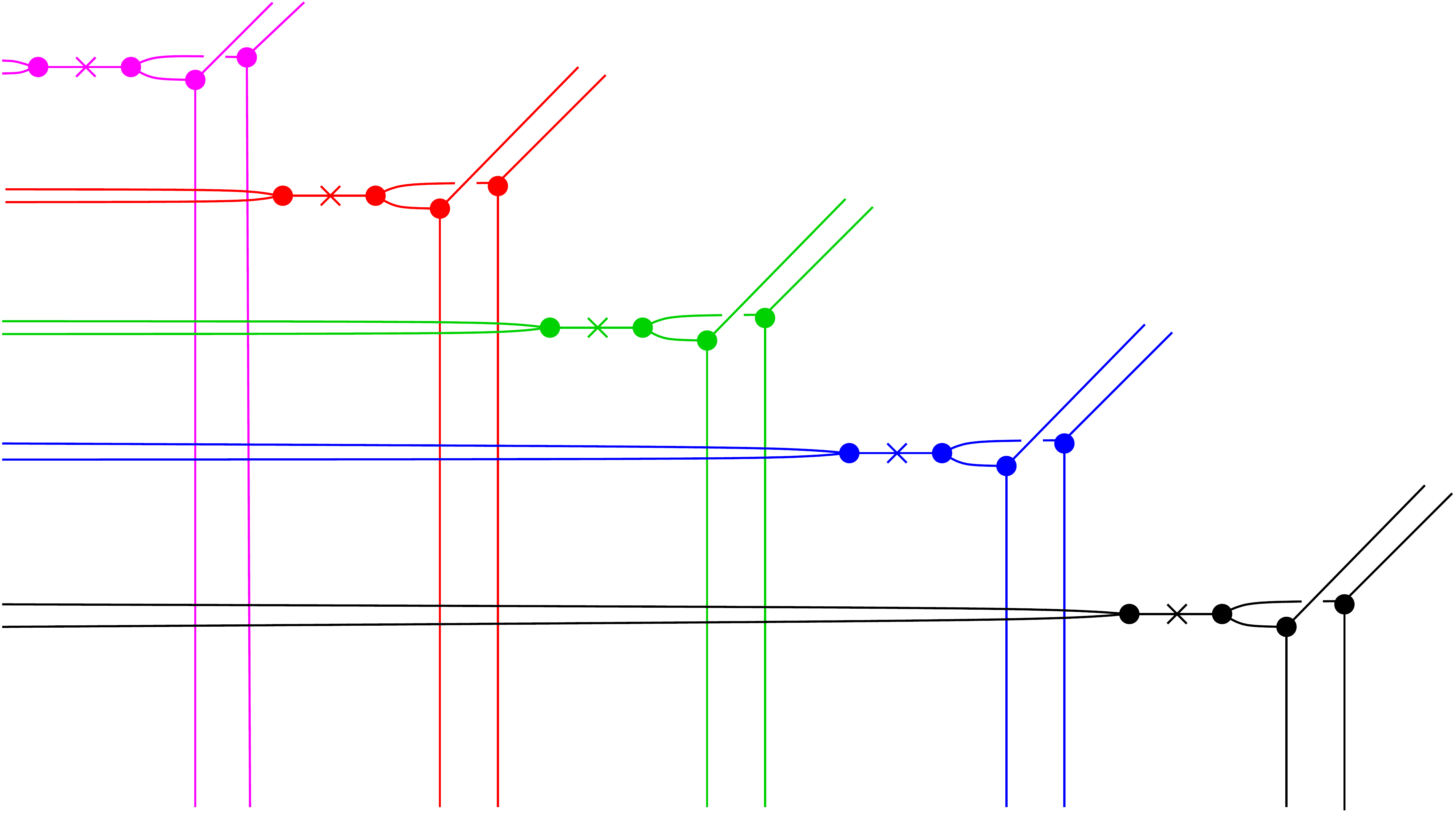}
\\
\\c) &&d)
\end{tabular}
\begin{tabular}{c}
\includegraphics[height=7cm, angle=0]{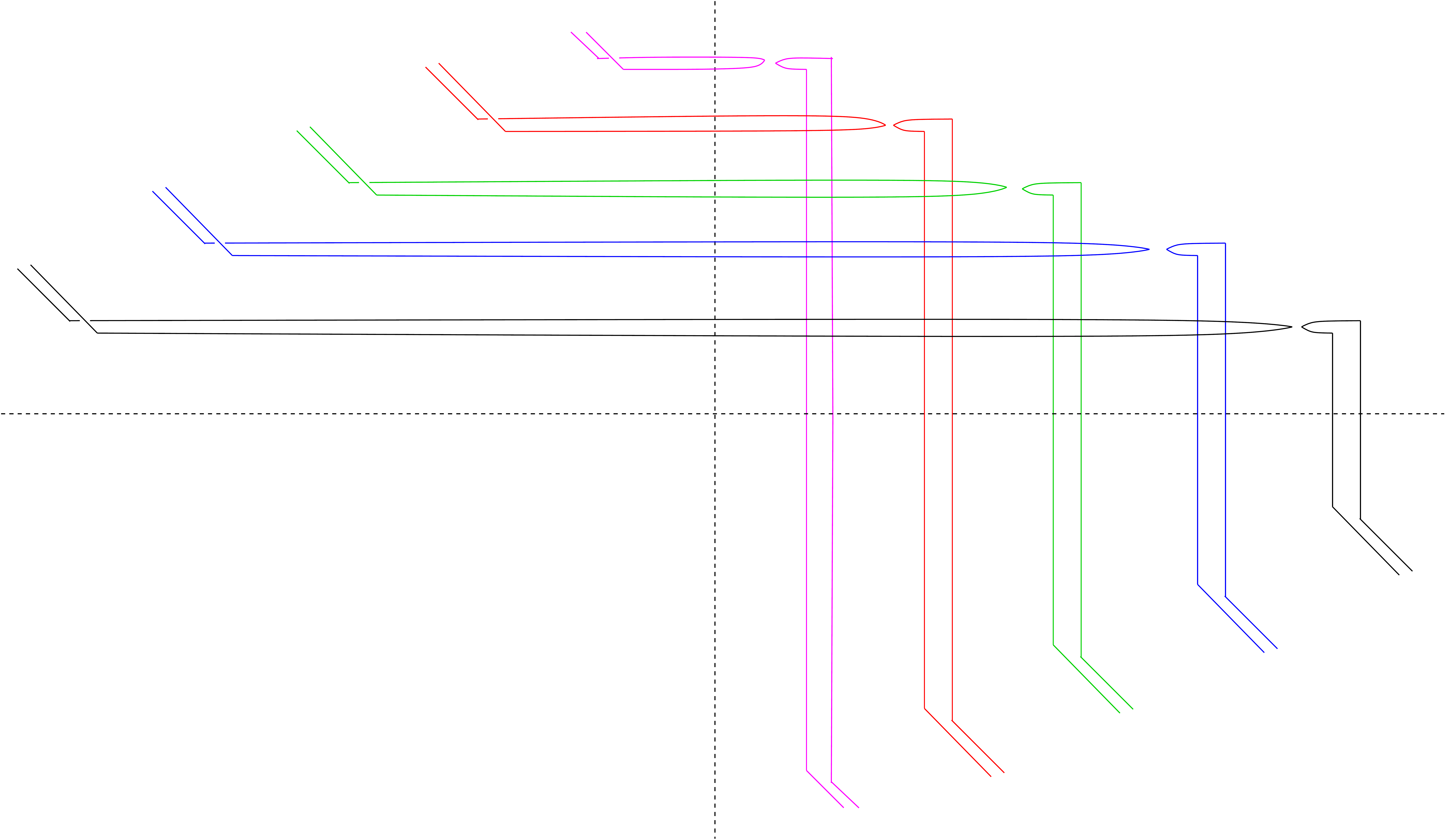}
\\
\\e)
\end{tabular}
\caption{3264 real conics tangent to 5 real conics}
\label{3264}
\end{figure}

Note that once the lines and points $L_i$ and $p_i$ are chosen as above,
 arguments used in Example \ref{ex real conic} also prove next proposition.
\begin{prop}
For any $0\le k\le 5$, any $d_1\ldots,d_{5-k}\ge 1$, and any
$g_1\ldots,g_{5-k}\ge 0$, 
there exists a generic configuration of $k$ points $p_1,\ldots,p_k$ 
in $\RR P^2$ and
$5-k$ immersed real algebraic curves  $C_1,\ldots,C_{5-k}$ with $C_i$
of degree $d_i$ and genus $g_i$ such that all conics passing through
$p_1,\ldots,p_k$
and tangent to $C_1,\ldots,C_{5-k}$ are real.
\end{prop}
For other  examples of totally real enumerative problems concerning conics
in $\RR P^n$,
see \cite{Br14}.

\section{Floor decompositions}\label{floor dec}

\renewcommand{\C}{\mathcal C}

\subsection{Motivation}\label{motivation}

In this section we give a purely combinatorial solution to the
computation of characteristic numbers.  To obtain totally
combinatorial objects we stretch our configuration of constraints in
the vertical direction,
i.e. we only consider configurations $(\P,\L)$
for which the difference of the
$y$-coordinates of any two elements of the set $\P\cup_{L\in\L}\Ve(L)$
is very big compared to the difference of their
$x$-coordinates.
For a sufficiently stretched configuration
$(\P,\L)$, tropical morphisms $f:C\to\RR^2$ in $\S^\TT(d,\P,\L)$ will
have a very simple decomposition into \textit{floors} linked together
by \textit{shafts}.  \textit{Marked floor diagrams} and their
multiplicities will encode the combinatoric of these decompositions
together with the distribution of $f^{-1}(\P)$ and the tangency
components of $f$ with elements of $\L$.  In the case where no
tangency condition is imposed, these new floor diagrams get simplified
to an equivalent of those introduced in \cite{Br7} and \cite{Br6b}.
The floor diagrams we define here allow us to compute characteristic
numbers of the plane in terms of a slight generalization of Hurwitz
number. Indeed the multiplicity of a marked floor diagram will be expressed in terms 
of \textit{open Hurwitz numbers} which  appear in two distinct ways in the count of $N_{d,0}(k;d_1,\ldots ,d_{3d-1-k})$. These numbers were introduced in
\cite{Br13}. We give in Appendix \ref{open Hurwitz} the definitions and
result from \cite{Br13} we need in this paper.

\begin{defi}
An \emph{elevator} of a  tropical morphism $f:C\to\RR^2$ is an edge $e$
of
$C$ with $u_{f,e}=\pm (0,1)$. A \emph{shaft} of $f$ is 
a connected
component of the topological closure of the union of all elevators of $f$.
The set of  shafts of $f:C\to\RR^2$ is
denoted by $\Sh(f)$.
 
A \emph{floor} of a tropical morphism $f:C\to\RR^2$ is a connected
component of the topological closure of $C\setminus \Sh (f)$. 
The \emph{degree} of a floor $\F$ of $f$, denoted by $\deg(\F)$, is the tropical
intersection number of $f(\F)$ with a generic vertical line of $\RR^2$.  
\end{defi}

Let us illustrate our approach on a simple case. Let us consider  $\L$
the set composed of the
five tropical lines depicted in Figure \ref{exa approach}. The set 
$\S^\TT(2,\emptyset,\L)$ is then reduced to the tropical morphism $f:C\to\RR^2$
depicted in Figure \ref{exa approach}b, which is of multiplicity 1.
 This  morphism has one
floor of degree 2, and one shaft made of three elevators. Let us
represent the morphism $f$ by the graph depicted in Figure 
\ref{exa approach}c, where the black vertex represents the shaft of
$f$, the white vertex represents the floor of $f$, and the edge
represents the weight 2 elevator of $f$ which join the shaft and the
floor of $f$. By remembering on this graph how are distributed the
tangency components of $f$ with the lines $L_i$, we obtain the labeled
graph depicted in Figure 
\ref{exa approach}d.

\begin{figure}[h]
\centering
\begin{tabular}{ccccc}
\includegraphics[width=3cm, angle=0]{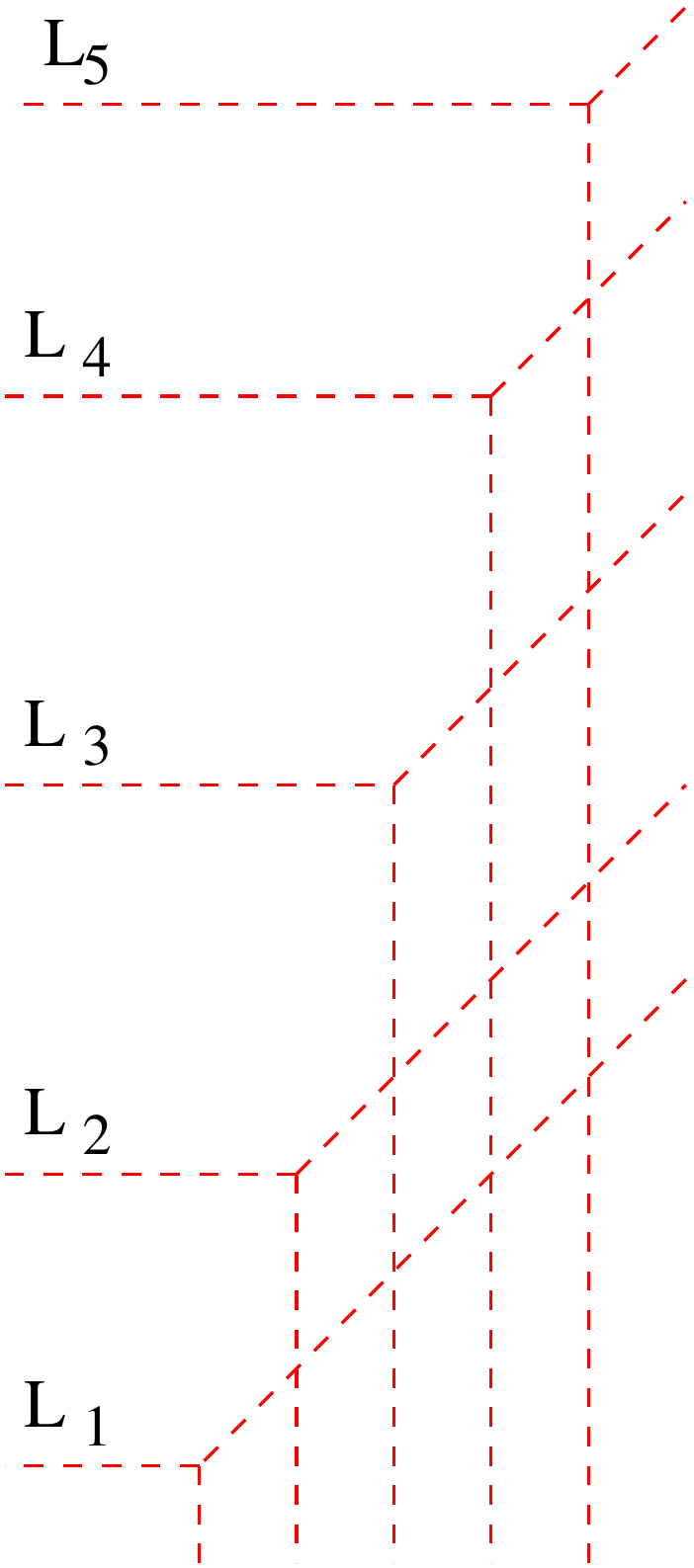}& \hspace{10ex} &
\includegraphics[width=3cm, angle=0]{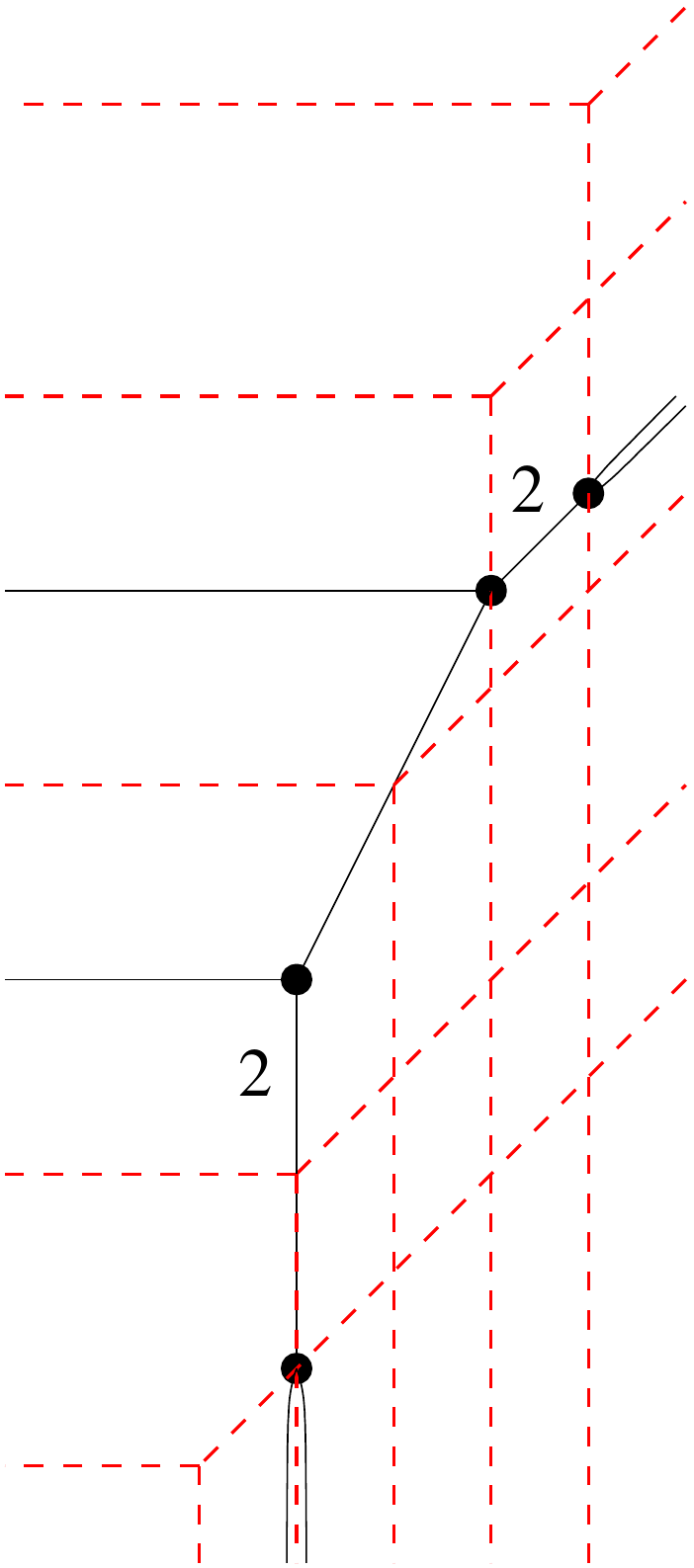}&
\includegraphics[width=3cm, angle=0]{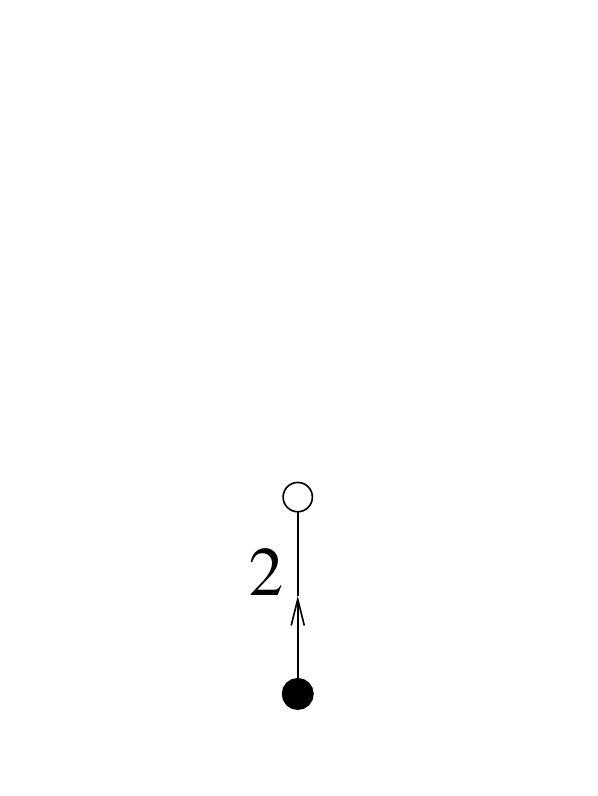}&
\includegraphics[width=3cm, angle=0]{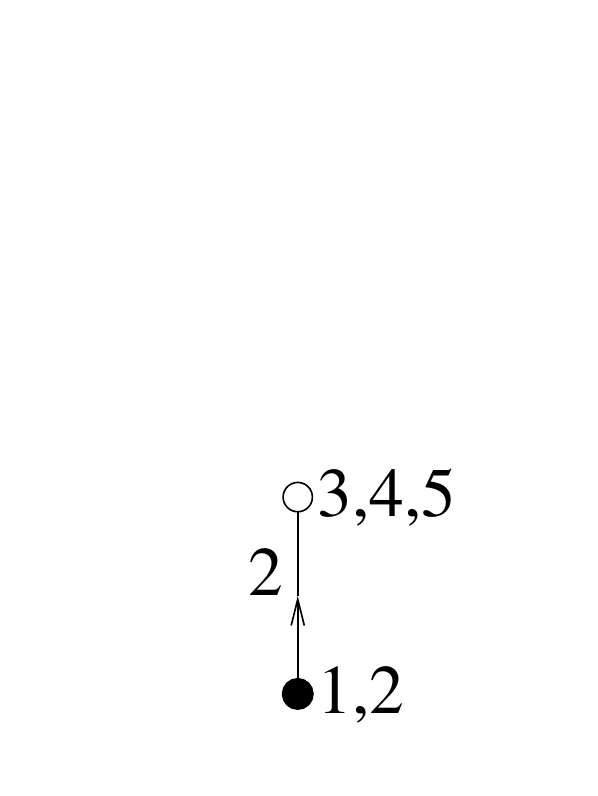} 
\\ a)& &b) &c) &d)
\end{tabular}

\caption{A tropical conic tangent to five lines, and its associated
 marked  floor diagram}
\label{exa approach}
\end{figure}

Let us define the two projections $\pi_x$ and $\pi_y$ as follows
$$\begin{array}{ccccccccc}\pi_x :&\RR^2&\to& \RR & \quad \text{and}\quad &
\pi_y :&\RR^2&\to& \RR
\\ & (x,y)&\mapsto&x
&& & (x,y)&\mapsto&y
 \end{array}.$$

Our main observation is the following:
\begin{itemize}
\item the map $\pi_x\circ f$ restricted to the
floor of $f$ is a tropical 
 ramified covering of $\RR$ of degree 2; its
critical values correspond to the vertical edge of the
lines $L_4$ and $L_5$ (see Figure \ref{exa approach 2}); 
\item  the map $\pi_y\circ f$ restricted to the 
shaft of $f$ is a tropical morphism with source a tropical curve with one 
boundary component; its critical value
correspond approximatively to the horizontal edge of $L_1$; the image
of its boundary component corresponds approximatively to the
horizontal edge of $L_3$. 
\end{itemize}

Vice versa, the
morphism $f:C\to\RR^2$ can be reconstructed out of the labeled diagram
of Figure \ref{exa approach}d in the following way: we first find the
tropical solutions of two tropical open Hurwitz problems, one for the
floor of $f$ and one for its shaft; next we glue them according to the
elevator joining this floor and this shaft, and the lines $L_2$ and
$L_3$.

\begin{figure}[h]
\centering
\begin{tabular}{c}
\includegraphics[width=8cm, angle=0]{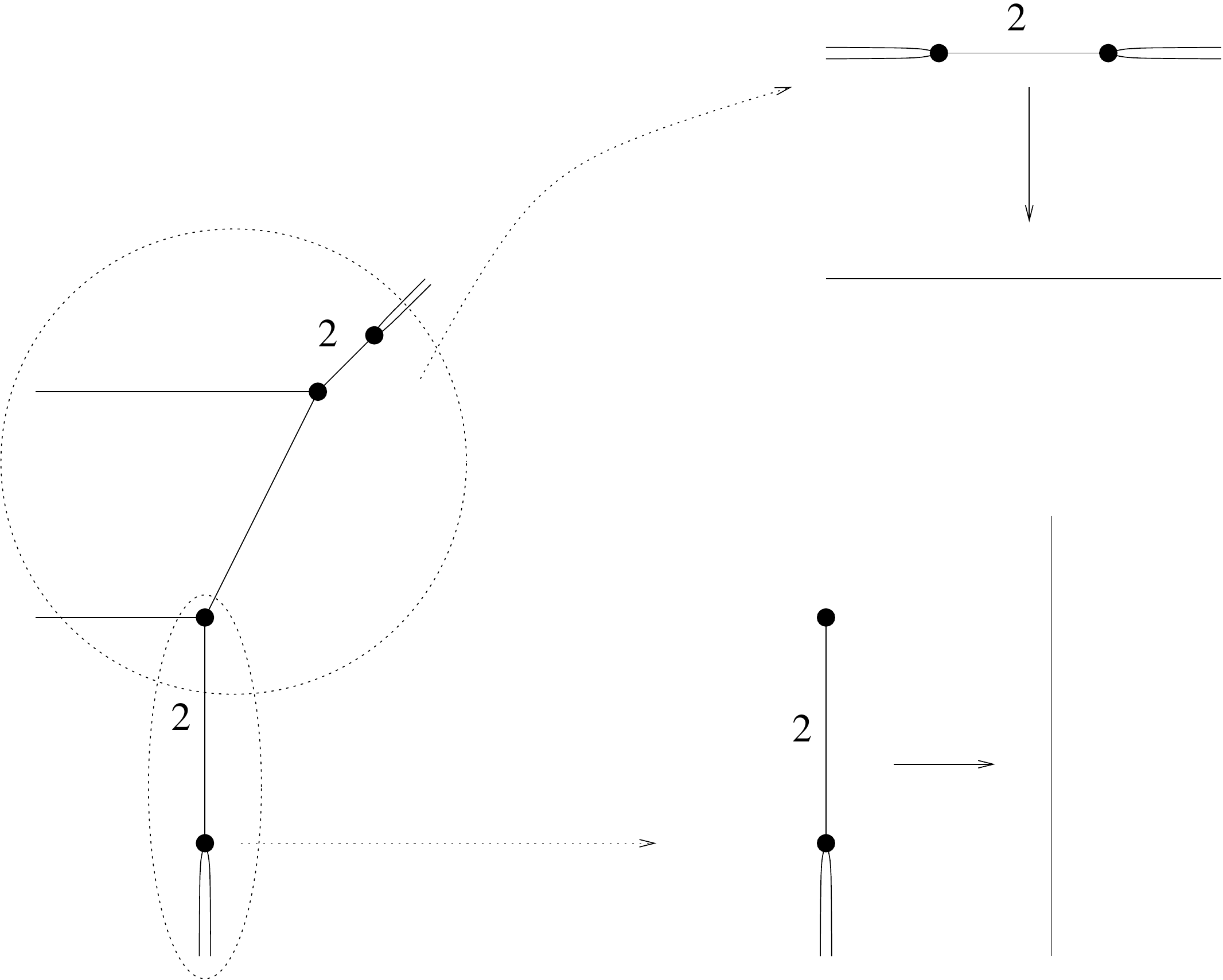}
\end{tabular}

\caption{From characteristic numbers to open Hurwitz numbers}
\label{exa approach 2}
\end{figure}

The floor of $f$ leads to the Hurwitz number $H(2)=\frac{1}{2}$; we
have two possibilities to attach the weight 2 elevator;
 making the floor tangent to $L_3$ gives  a factor 1; the Hurwitz
number we have to compute to reconstruct the shaft of $f$ is
$H(\delta,n)=\frac{1}{2}$ where $\delta(0)=2, \delta(1)=n(1)=0$, and
$n(0)=1$ 
(see Appendix \ref{open Hurwitz} for the definition of
Hurwitz numbers);
 making the shaft tangent to $L_2$ gives us a factor 1;
gluing the floor and the shaft along the weight 2 elevator gives
 an extra factor 2. Hence, the total multiplicity of $f$ is
$$\frac{1}{2}\times 2\times 1\times\frac{1}{2}\times 1\times 2=1 $$
as expected.

The next section is devoted to the generalization of the previous
computation to arbitrary degree and set of constrains.

\subsection{Floor diagrams}
Here we define floor diagrams and their markings. 
Our definitions are similar
in the spirit of those given in
\cite{Br7}, \cite{Br6b}, and \cite{FM}, but are somewhat different since our
enumerative problems also involve tangency conditions.
For simplicity, we only explain in detail how to turn the problem of
computing the numbers $N_{d,0}(k;1^{3d-1-k})$ into the enumeration of
marked floor diagrams. The general computation of the numbers 
$N_{d,0}(k;d_1,\ldots,d_{3d-1-k})$ in terms of floor diagrams require no
more substantial efforts, but makes the exposition heavier. Hence we
 restrict ourselves to the case of tangency with lines, which by
Equation
 (\ref{break constraint})
 is enough to recover all genus 0 characteristic numbers
of $\CC P^2$.

The floor diagrams we  deal with in this paper underlie bipartite trees, whose
vertices are divided between \textit{white} and \textit{black} vertices.
As usual, the divergence of a vertex $v$ of $\D$, denoted by
$\div(v)$,
 is the sum of the weight
of all its incoming adjacent 
edges minus the sum of the weight of all its outgoing adjacent edges.
\begin{defi}
A \emph{floor diagram} (of genus 0) is an oriented bipartite tree $\D$
equipped with a weight function  $w:\Ed(\D)\to \ZZ_{>0}$ such that
white vertices have positive divergence, and black vertices have
non-positive divergence.
\end{defi}

The
sum of the divergence of all white vertices is called the
\textit{degree} of $\D$. We denote by $\Ve^\circ(\D)$ the set of white
vertices of $\D$, and by 
$\Ve^\bullet(\D)$ the set of its black vertices.

As explained in Section \ref{motivation},
 a white vertex represents a \textit{floor} of
a  tropical
morphism, whereas a black vertex represents one of its  shafts.

\begin{exa}
All floor diagrams of degree 2 are depicted in
Figure \ref{fd conic 0}. We precise the weight of an edge of $\D$ only
if this latter is not 1.
\end{exa}
\begin{figure}[h]
\centering
\begin{tabular}{ccccc}
\includegraphics[width=3cm, angle=0]{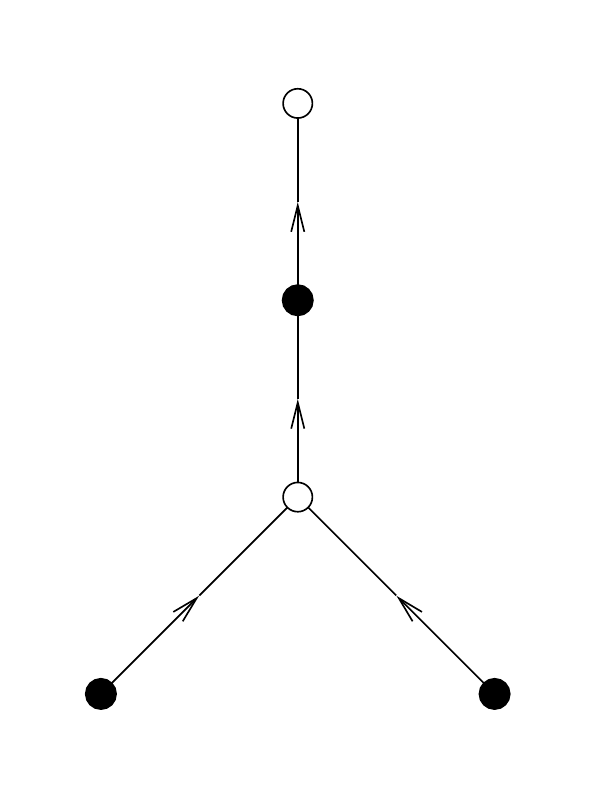}&
\includegraphics[width=3cm, angle=0]{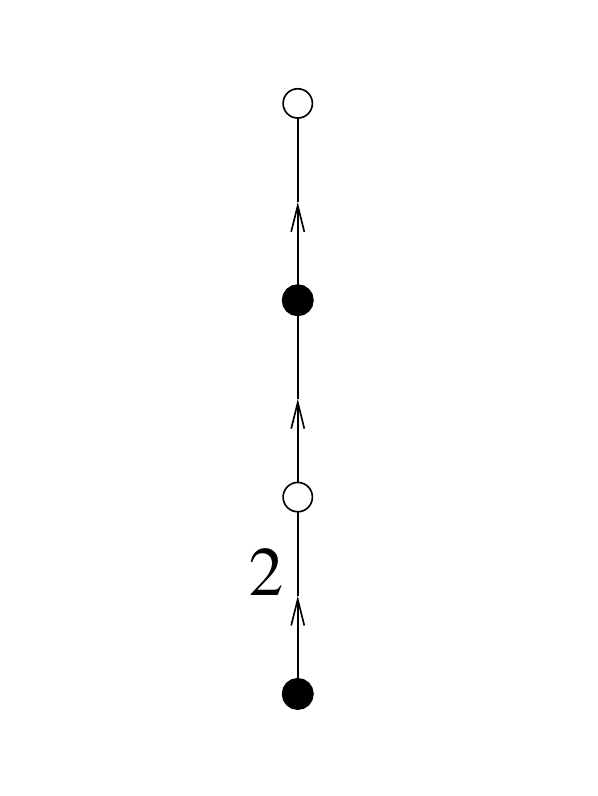}&
\includegraphics[width=3cm, angle=0]{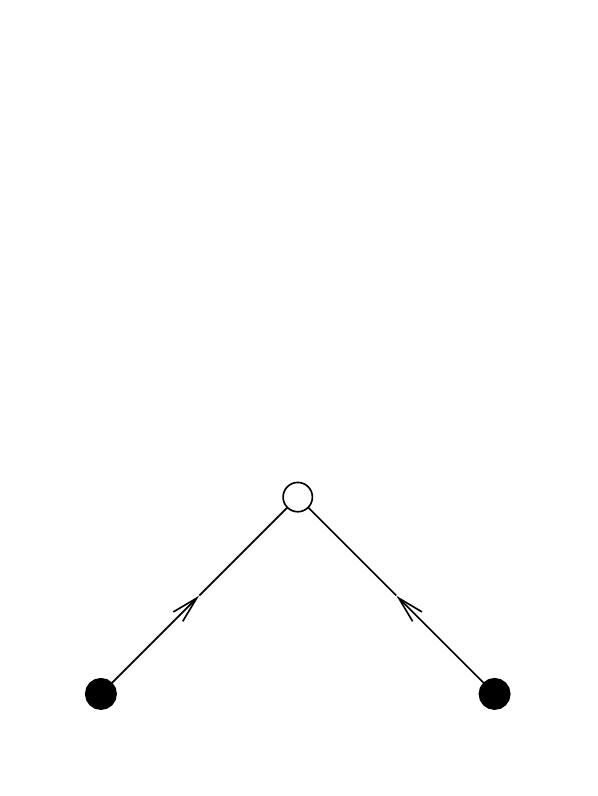}&
\includegraphics[width=3cm, angle=0]{Figures/FD2d.pdf}&
\includegraphics[width=3cm, angle=0]{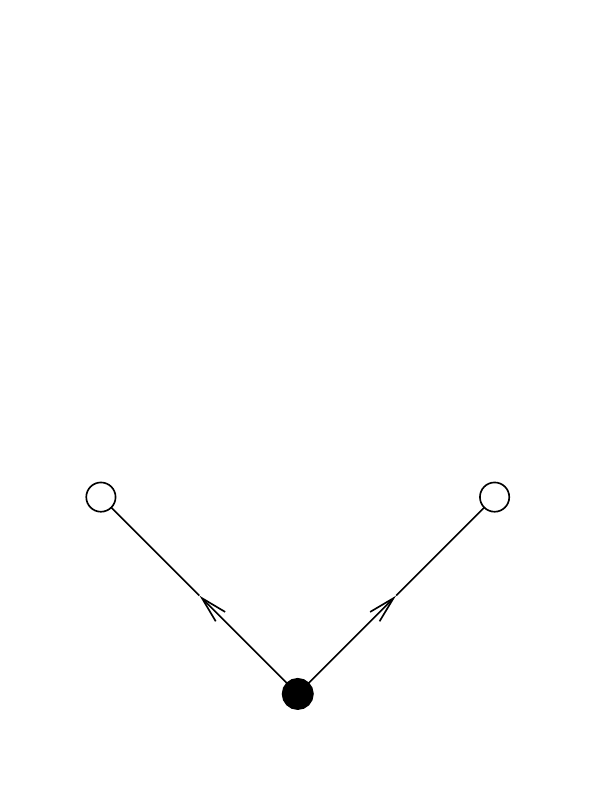}
\end{tabular}

\caption{Floor diagrams of degree 2}
\label{fd conic 0}
\end{figure}

Given a vertex $v$ of $\D$, we denote by $\Ve(v)$ the set of vertices
of $\D$ adjacent to $v$.

\begin{defi}
Let $\L^{comb}\sqcup \P^{comb}$ be a partition of the set
$\{1,\ldots,3d-1\}$.
A \emph{$\L^{comb}$-marking} of a floor diagram $\D$ of degree $d$ is a surjective map 
$m:\{1,\ldots,3d-1\}\to \Ve(\D)$ such that
\begin{itemize}
\item for any $v\in\Ve(\D)$, the set $m^{-1}(v)$ contains at most one
  point in $\P^{comb}$; moreover, if $v\in\Ve^\circ(\D)$ and 
 $m^{-1}(v)\cap \P^{comb}=\{i\}$, then  $i=\min(m^{-1}(v))$;

\item for any $v\in\Ve^\circ(\D)$, $|m^{-1}(v)|=2div(v)-1$;

\item for any $v\in\Ve^\bullet(\D)$, $|m^{-1}(v)|=val(v)-div(v)-1$;
  moreover there exists at most one element $i$ in $m^{-1}(v)$ such
  that $i> \max_{v'\in\Ve(v)} \min(m^{-1}(v'))$, and if such an
  element exists then we
  have $m^{-1}(v)\cap \P^{comb}= \emptyset$.

\end{itemize}
\end{defi}
Two $\L^{comb}$-markings $m:\{1,\ldots,3d-1\}\to \Ve(\D)$ and 
$m':\{1,\ldots,3d-1\}\to \Ve(\D')$ are \emph{isomorphic} if there exists an
isomorphism of bipartite graphs $\phi:\D\to \D'$ such that
$m=m'\circ\phi$. In this text, $\L^{comb}$-marked floor diagrams are
considered up to isomorphism.

The set $\{1,\ldots,3d-1\}$  represents the configuration
of constraints in the increasing height order (see Section
\ref{motivation}),
 the set $\L^{comb}$
represents the lines in the configuration, and the set $\P^{comb}$
represents the points.
Note that unlike in \cite{Br7}, \cite{Br6b}, and \cite{FM}, 
we do not consider the partial order on $\D$ defined by its
orientation. In particular, it makes no sense here to require the marking $m$
to be an increasing map.

In order to define the multiplicity of an $\L^{comb}$-marked floor
diagram, we first define the multiplicity of a vertex of $\D$. Recall
that the definitions of 
 open Hurwitz numbers $H(\delta,n)$ and
Hurwitz numbers $H(d)=H(d,0)$ 
 are given
in Definition  \ref{def open hurwitz}.

\begin{defi}\label{def mult white}
The \emph{multiplicity of a vertex $v$ in $\Ve^\circ(\D)$} is defined as
\begin{itemize}
\item  if
  $\min(m^{-1}(v))\in\P^{comb}$, then
$$\mu_{\L^{comb}}(v)=\div(v)^{\val(v)+1}H(\div(v))$$

\item  otherwise
$$\mu_{\L^{comb}}(v)=(\div(v)-2 + \val(v))\div(v)^{\val(v)}H(\div(v)).$$
\end{itemize}
\end{defi}

\begin{exa}
We give in Figure \ref{ex mult white} some examples of  multiplicities
of white vertices of a marked floor diagram. The corresponding Hurwitz
numbers are given in Proposition \ref{exa hurwitz}. We write the
elements of $m^{-1}(v)$ close to the vertex $v$.
\end{exa}
\begin{figure}[h]
\centering
\begin{tabular}{cccc}
\includegraphics[width=3cm, angle=0]{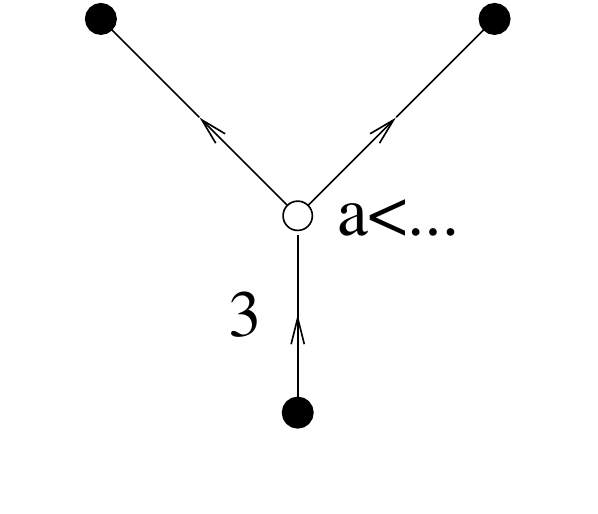}&
\includegraphics[width=3cm, angle=0]{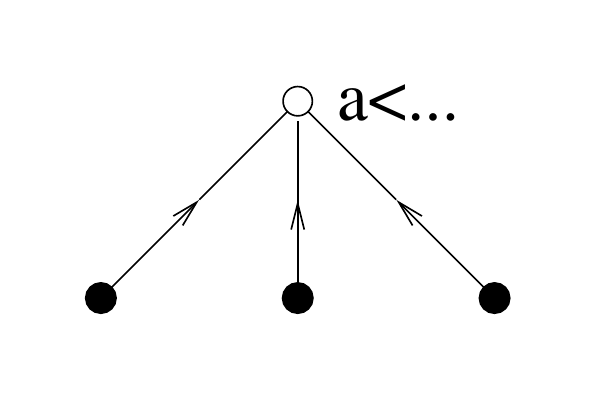}&
\includegraphics[width=3cm, angle=0]{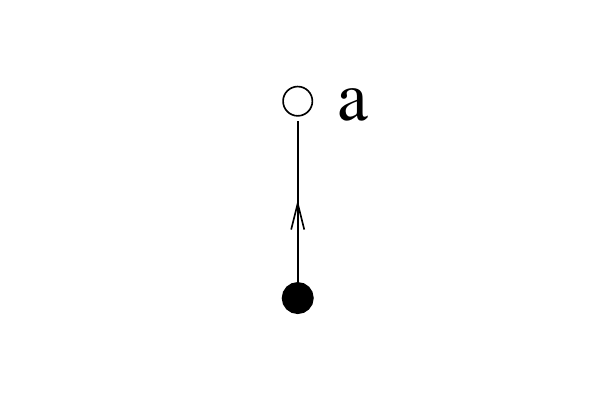}&
\includegraphics[width=3cm, angle=0]{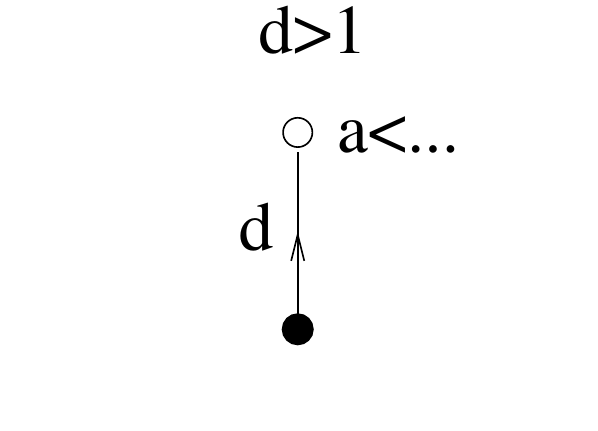}
\\ \begin {tabular}{l} 
 $a\in\P^{comb}$, $\mu=1$
  \\ 
   $a\in\L^{comb}$, $\mu=2$
\end{tabular}  & 
\begin {tabular}{l} $a\in\P^{comb}$, $\mu=324$
\\  $a\in\L^{comb}$, $\mu=432$ \end{tabular}    & 
\begin {tabular}{l} $a\in\P^{comb}$, $\mu=1$
\\  $a\in\L^{comb}$, $\mu=0$ \end{tabular}    & 
\begin {tabular}{l} $a\in\P^{comb}$, $\mu=\frac{d^{d-1}(2d-2)!}{d!}$
\\  $a\in\L^{comb}$, $\mu=\frac{d^{d-3}(2d-2)!}{(d-2)!}$ \end{tabular}   
\end{tabular}

\caption{Example of multiplicities of white vertices of $\D$}
\label{ex mult white}
\end{figure}

The definition of the multiplicity of a black
vertex $v$ of $\D$ requires a preliminary construction.
 The order  on $\{1,\ldots,3d-1\}$
induces an order on $\Ve(v)$ via the map $v'\mapsto\min(m^{-1}(v'))$. Note that
this order doesn't have to be compatible with the orientation of $\D$.
Let us
denote by $v'_1<\ldots <v'_s$ the elements of $\Ve(v)$ according to this
order. We denote by $e_i$ the edge of $\D$ joining the vertices $v$
and $v'_i$, and define $\varepsilon_i=1$ if $e_i$ is oriented toward $v$, and
$\varepsilon_i=-1$ otherwise.
Given $j\in m^{-1}(v)$ we define the integer $i_j$ by

$i_j=0$ if $j<\min(m^{-1}(v'_1))$;

$i_j=i$ if $\min(m^{-1}(v'_{i}))<j<\min(m^{-1}(v'_{i+1}))$;

$i_j=s$ if $j>\min(m^{-1}(v'_s))$.

\vspace{1ex}
\noindent We define two functions $\delta, \tilde n:\{0,\ldots,s\}\to \ZZ$ 
by

\begin{itemize}
\item $\delta(0)=-div(v)$,

\noindent   $\delta(i+1)=\delta(i)+
 \varepsilon_{i+1}w(e_{i+1})$;
\vspace{1ex}

\item $\tilde n(i)=|\{j\in m^{-1}(v)\ |\ i_j=i\}|$.
\end{itemize}

\noindent Given $i_0\in  \{0,\ldots,s\} $, we define the function $n_{i_0}:
 \{0,\ldots,s\}\to \ZZ_{\ge 0}$ by
$n_{i_0}(i_0)=\tilde n(i_0)-1$ and $n_{i_0}(i)=\tilde n(i)$ if $i\ne
 i_0$. Finally, we define 
 $\tilde N(i)=\sum_{l=0}^{i}\tilde n(l)$ 
and $\tilde N(-1)=0$.

\begin{defi}\label{def mult black}
The \emph{multiplicity of a vertex $v$ in $\Ve^\bullet(\D)$} is defined by
the following rules
\begin{itemize}
\item if $m^{-1}(v)\cap\P^{comb}=\{j\}$,
  then 
$$\mu_{\L^{comb}}(v)=\delta(i_j)H(\delta,n_{i_j})$$

\item if $m^{-1}(v)\cap\P^{comb}=\emptyset$ and 
$m^{-1}(v)$ contains an element $j$ such
  that $j> \max_{v'\in\Ve(v)} \min(m^{-1}(v'))$, 
  then 
$$\mu_{\L^{comb}}(v)=(2\val(v) -2)H(\delta,n_s)$$

\item otherwise,
$$\mu_{\L^{comb}}(v)=\frac{1}{2}\sum_{i=0}^s \left(\tilde n(i)
\left(2\delta(i)+ 2i
+\tilde N(i)+\tilde N(i-1)-1 +2\div(v) \right)H(\delta,n_{i})  \right).$$

\end{itemize}
\end{defi}
\begin{exa}
We give in Figure \ref{ex mult black} some examples of  multiplicities
of black vertices of a marked floor diagram. The corresponding open Hurwitz
numbers are given in Proposition \ref{exa hurwitz} and Example~\ref{ex comp open H}.
\end{exa}
\begin{figure}[h]
\centering
\begin{tabular}{cc}
\includegraphics[width=4cm, angle=0]{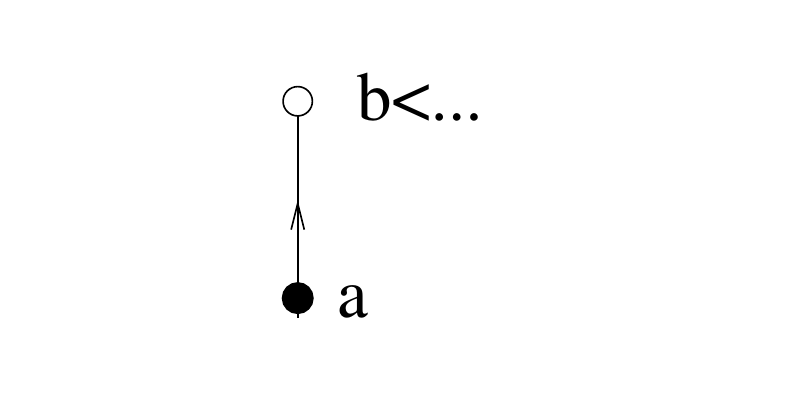}&
\includegraphics[width=4cm, angle=0]{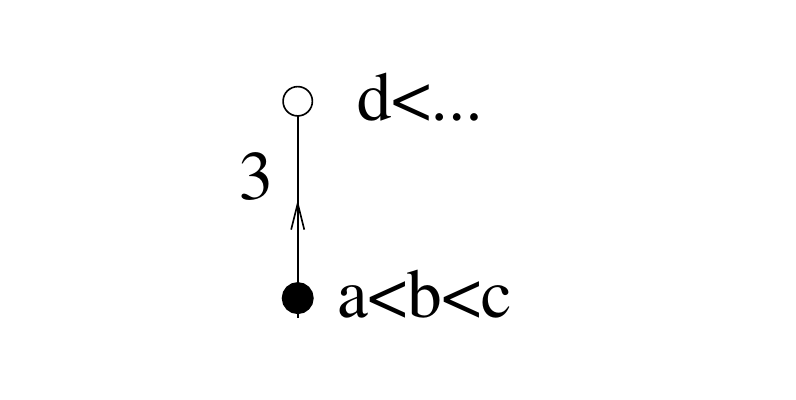}
\\ \begin {tabular}{l} $a>b$, $\mu=0$
  \\ $a<b$, $a\in\P^{comb}$, $\mu=1$
\\ $a<b$, $a\in\L^{comb}$, $\mu=0$  \end{tabular}  & 
\begin {tabular}{l} $c>d$, $\mu=0$
  \\ $c<d$, $\{a,b,c\}\cap\P^{comb}\ne\emptyset$, $\mu=3$
\\ $c<d$, $\{a,b,c\}\subset\L^{comb}$,
$\mu=3$  \end{tabular}   
\\ \\

\includegraphics[width=4cm, angle=0]{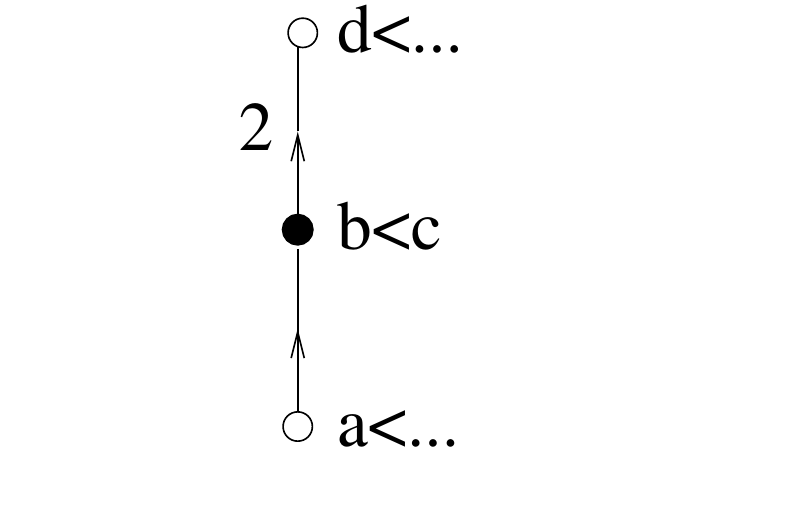}&
\includegraphics[width=4cm, angle=0]{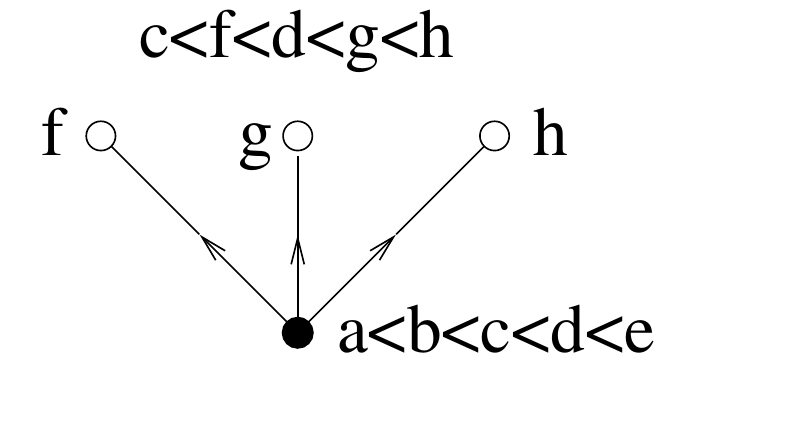}
\\ 
\begin {tabular}{l} $b<a$, $b\in\P^{comb}$, $\mu=1$ 
\\$b<a$, $b\in\L^{comb}$, $\mu=0$
\\$a<b<c<d$, $\{b,c\}\cap\P^{comb}\ne\emptyset$, $\mu=2$
  \\ $a<b<c<d$, $\{b,c\}\subset\L^{comb}$, $\mu=5$
  \\ $a<b<d<c$, $\{b,c\}\subset\L^{comb}$, $\mu=2$
\\ \hspace{1ex}
  \end{tabular}  &
\begin {tabular}{l} 
$e>h$, $e\in\L^{comb}$, $\mu=32$
  \\ $e<g$, $\{d,e\}\cap\P^{comb}\ne\emptyset$, $\mu=16$
  \\ $e<g$, $c\in\P^{comb}$, $\mu=6$
\\ $e<g$, $\{a,b,c,d,e\}\in\L^{comb}$, $\mu=62$  \end{tabular}  

\end{tabular}

\caption{Example of multiplicities of black vertices of $\D$}
\label{ex mult black}
\end{figure}

\begin{defi}
The \emph{multiplicity of an $\L^{comb}$-marked floor diagram} is defined as
$$\mu_{\L^{comb}}(\D,m)=\prod_{e\in\Ed(\D)}w(e)\prod_{v\in\Ve(\D)}\mu_{\L^{comb}}(v).$$
\end{defi}

Note that $\mu_{\L^{comb}}(\D,m)$ can be equal to 0.

\begin{exa}
We give in Figure \ref{ex mult fd} a few examples of  multiplicities
of $\L^{comb}$-marked floor diagrams.
\end{exa}
\begin{figure}[h]
\centering
\begin{tabular}{ccccc}
\includegraphics[width=3cm, angle=0]{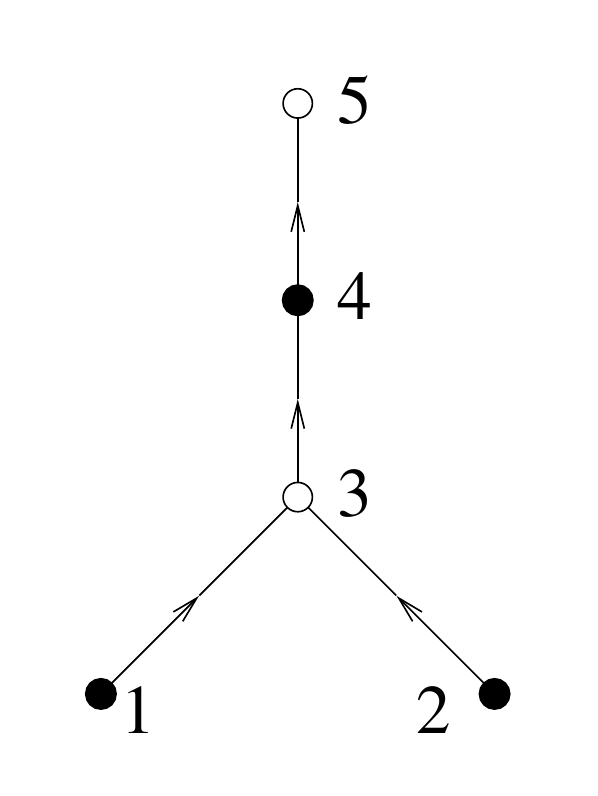}&
\includegraphics[width=3cm, angle=0]{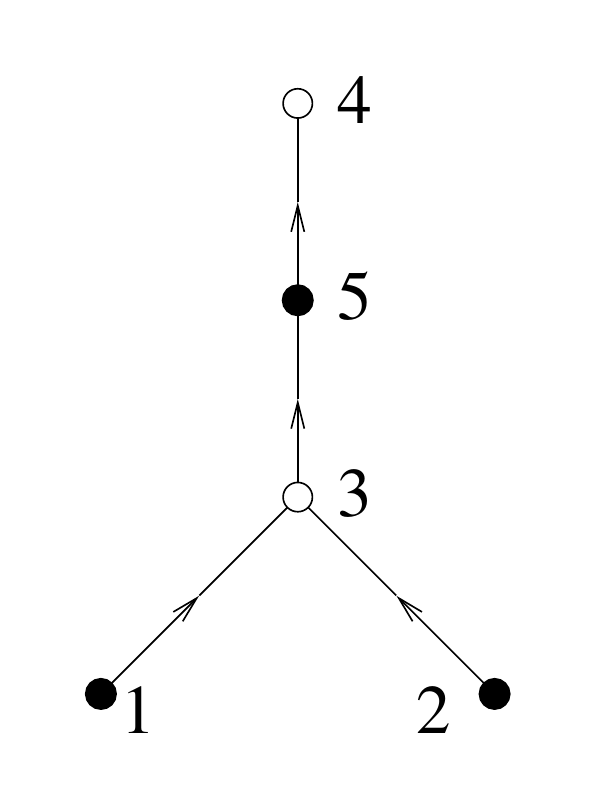}&
\includegraphics[width=3cm, angle=0]{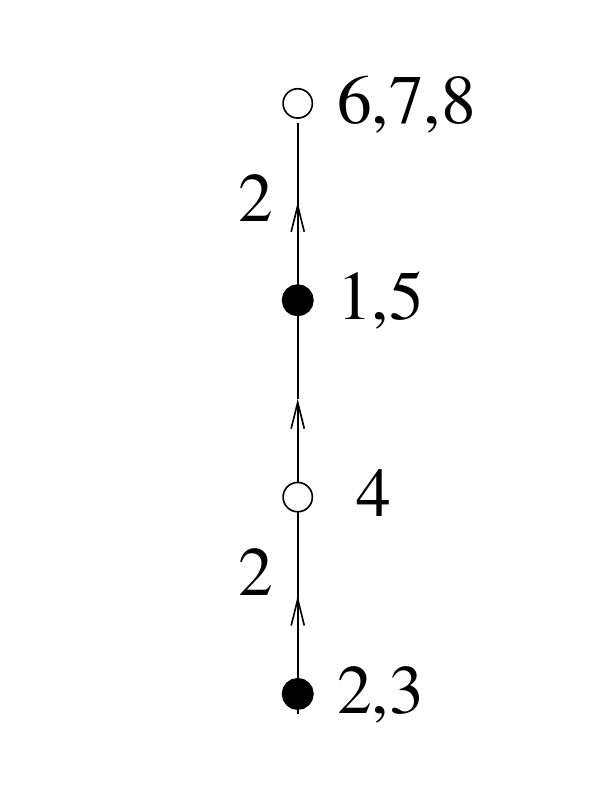}&
\includegraphics[width=3cm, angle=0]{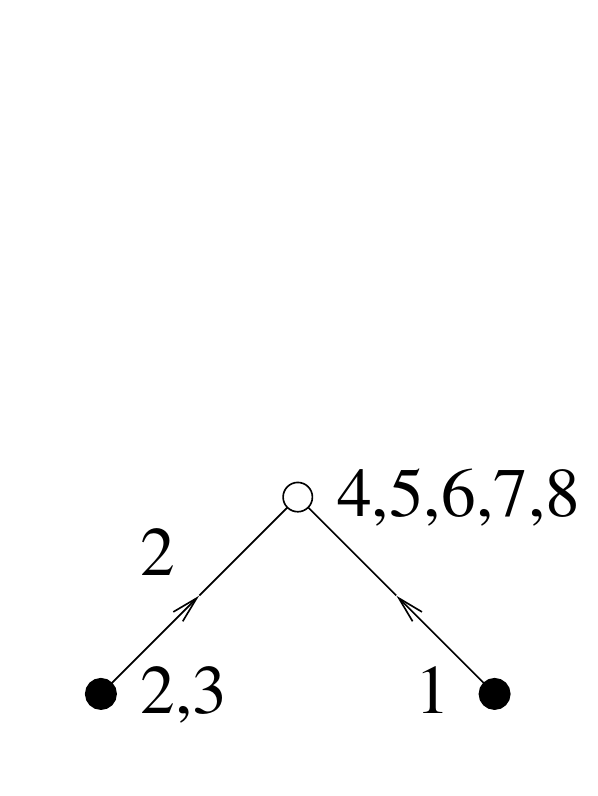}&
\includegraphics[width=3cm, angle=0]{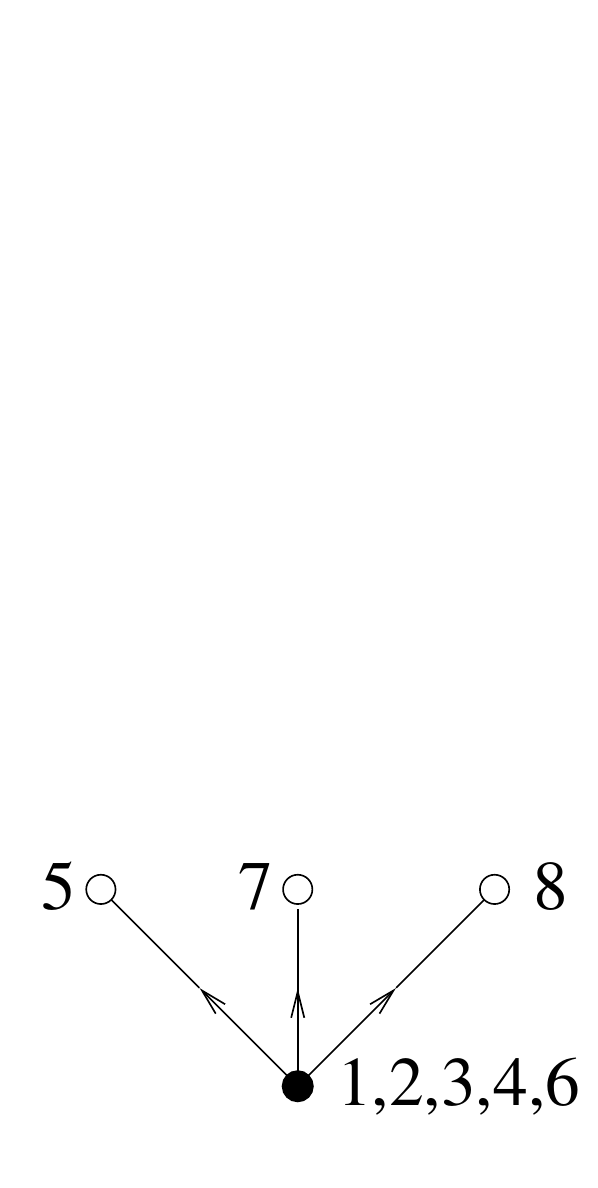}
\\ $\mu_{\emptyset}=1$ &$\mu_{\{5\}}=2$  &$\mu_{\{2,\ldots,8\}}=2$
&$\mu_{\{2,\ldots,8\}}=108$  &$\mu_{\{1,2,3,4,6\}}=80$  
\end{tabular}

\caption{Example of multiplicities of $\L^{comb}$-marked floor diagrams}
\label{ex mult fd}
\end{figure}

Next theorem is a direct consequence of Proposition \ref{fiber} and
Theorem \ref{Corres}.
\begin{thm}\label{fd thm}
For any $d\ge 1$,  $k\ge 0$, and $\L^{comb}\subset\{1,\ldots,3d-1\}$
of cardinal $3d-1-k$, we have
$$N_{d,0}(k;1^{3d-1-k})=\sum \mu_{\L^{comb}}(\D,m) $$
where the sum ranges over all $\L^{comb}$-marked floor diagrams of degree $d$.
\end{thm}

Note that in the case $k=3d-1$, Theorem \ref{fd thm} agrees with
\cite[Theorem 3.6]{Br6b} and \cite[Theorem 1]{Br7}. Indeed, 
 a $\emptyset$-marked floor diagram 
has non-null multiplicity if and only if the marking is increasing
with respect to the partial order on $\D$ defined by its orientation;  in
this case, according to Example \ref{elem1},
 the different definitions of multiplicity of a marked floor
diagram coincide.

\begin{exa}
We  first compute the numbers $N_{2,0}(k;1^{5-k})$, with
$\L^{comb}=\{k+1,\ldots,5\}$. In each case, there is exactly one marked
floor diagram of positive multiplicity, depicted in Figure 
\ref{fd conic 1}.
\end{exa}
\begin{figure}[h]
\centering
\begin{tabular}{c|c|c|c|c}
&\includegraphics[width=3cm, angle=0]{Figures/FD2Ma.pdf}&
\includegraphics[width=3cm, angle=0]{Figures/FD2Mb.pdf}&
\includegraphics[width=3cm, angle=0]{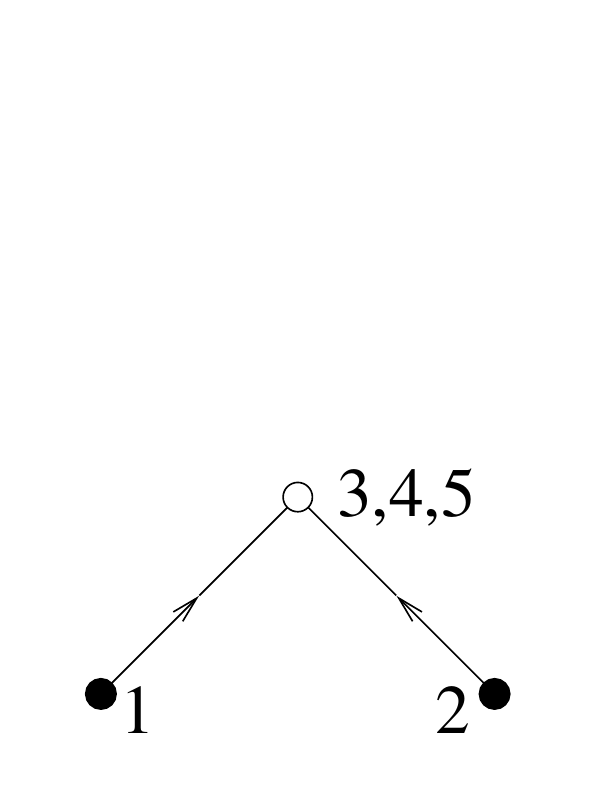}&
\includegraphics[width=3cm, angle=0]{Figures/FD2Md.pdf}
\\\hline $\mu_{\emptyset}$   &  1  & 0 & 0 &0
\\\hline $\mu_{\{5\}}$   & 0& 2  & 0 & 0

\\ \hline $\mu_{\{4,5\}}$   & 0&0    & 4 & 0
\\ \hline     $\mu_{\{3,4,5\}}$    & 0&0    & 4 & 0 
\\\hline $\mu_{\{2,3,4,5\}}$   & 0   &  0& 0& 2
\\\hline $\mu_{\{1,2,3,4,5\}}$   & 0   &  0& 0 &1

\end{tabular}

\caption{Computation of  $N_{2,0}(k;1^{5-k})$ with $\L^{comb}=\{k+1,\ldots,5\}$}
\label{fd conic 1}
\end{figure}

\begin{exa}
In order to  get used with  floor diagrams,
 we compute again  the numbers $N_{2,0}(k;1^{5-k})$, using now
$\L^{comb}=\{1,\ldots,5-k\}$. In this case we have one (resp. 3, and
 2) marked floor 
diagrams of positive multiplicity for $k=5,4,1,$ and 0 (resp. 3, and
2). These marked floor diagrams are depicted in Figure 
\ref{fd conic 2}.
\end{exa}
\begin{figure}[h]
\centering
\begin{tabular}{c|c|c|c}
&\includegraphics[width=3cm, angle=0]{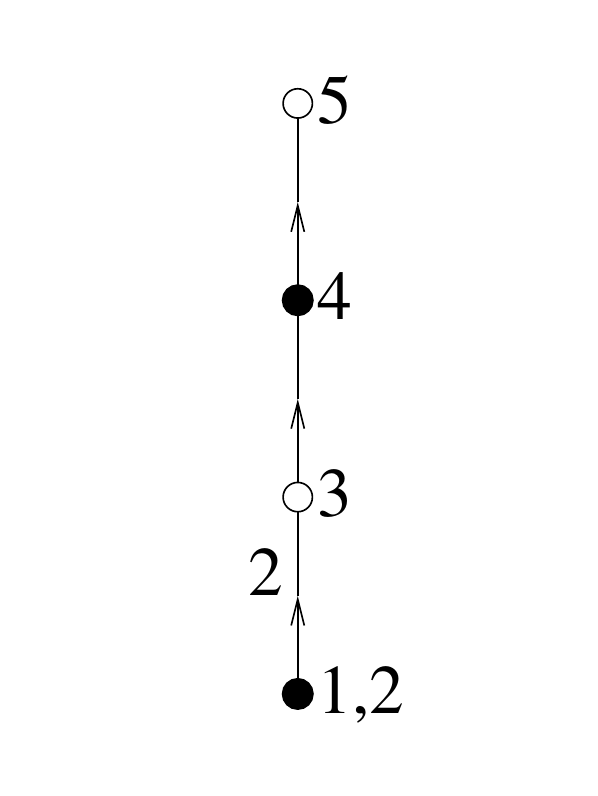}&
\includegraphics[width=3cm, angle=0]{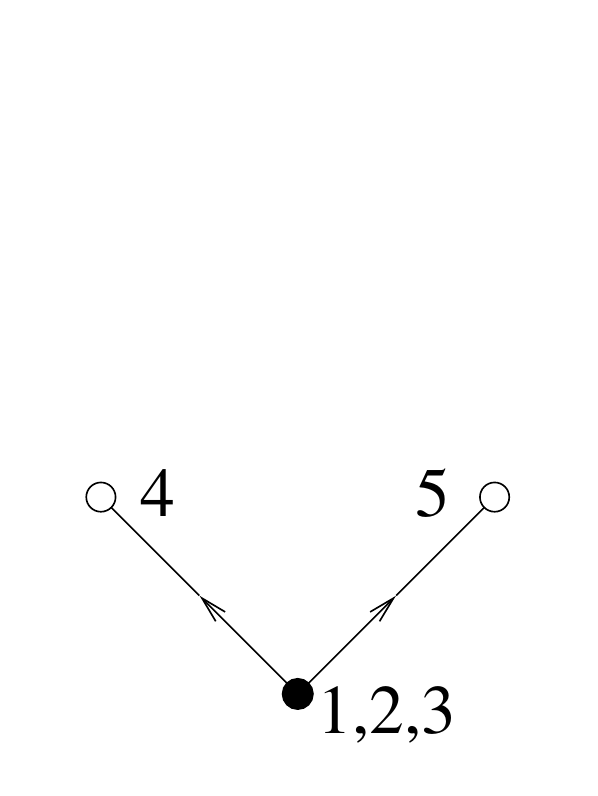}&
\includegraphics[width=3cm, angle=0]{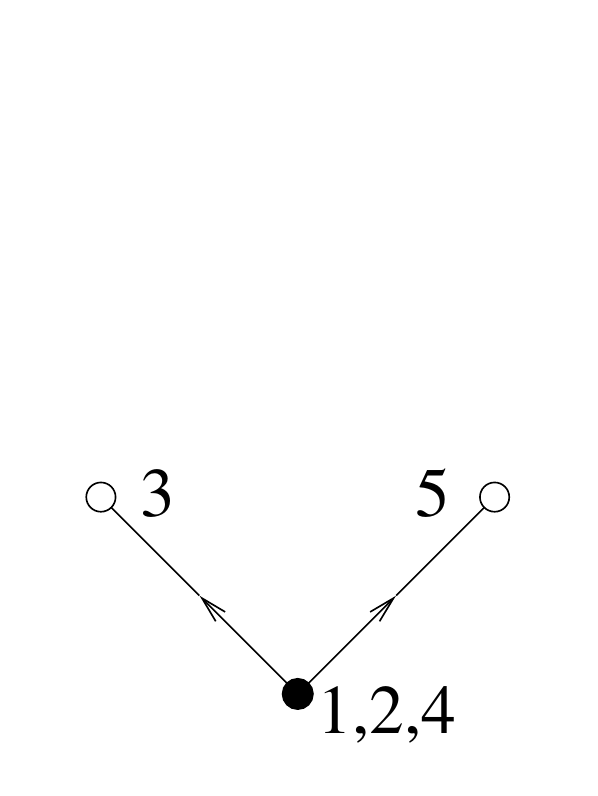}
\\\hline $\mu_{\{1\}}$   &  2  & 0 & 0

\\ \hline $\mu_{\{1,2\}}$   & 1   & 2 & 1
\\ \hline     $\mu_{\{1,2,3\}}$     & 1   & 3 & 0
\\\hline $\mu_{\{1,2,3,4\}}$   & 2   &  0& 0
\end{tabular}

\caption{Computation of  $N_{2,0}(k;1^{5-k})$ with $\L^{comb}=\{1,\ldots,5-k\}$}
\label{fd conic 2}
\end{figure}

\begin{exa}
 Figure \ref{fd cubic} represents all marked
floor diagrams of degree 3 with positive multiplicity when
$\L^{comb}=\{2,\ldots,8\}$.  Hence there are exactly 600 rational
cubics passing through 1 point 
and tangent to 7 lines.
\end{exa}
\begin{figure}[h]
\centering
\begin{tabular}{ccc}
\includegraphics[width=3cm, angle=0]{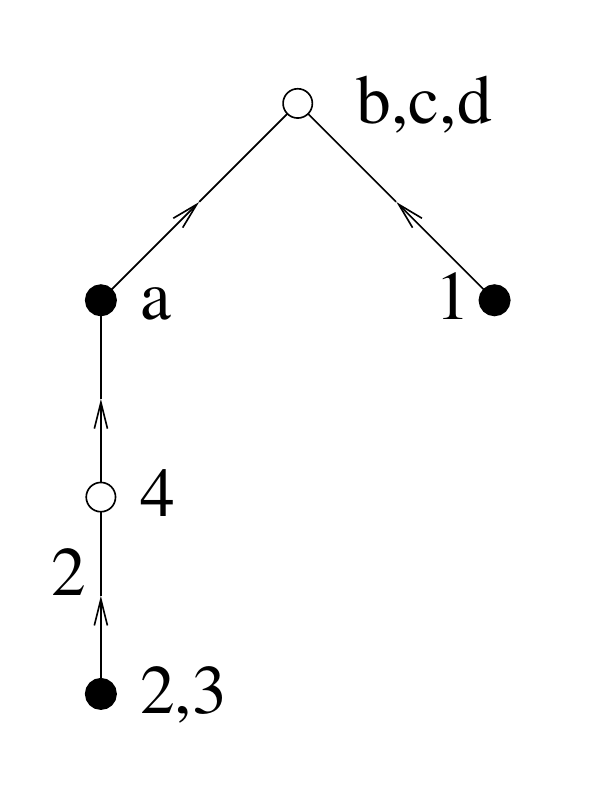}&
\includegraphics[width=3cm, angle=0]{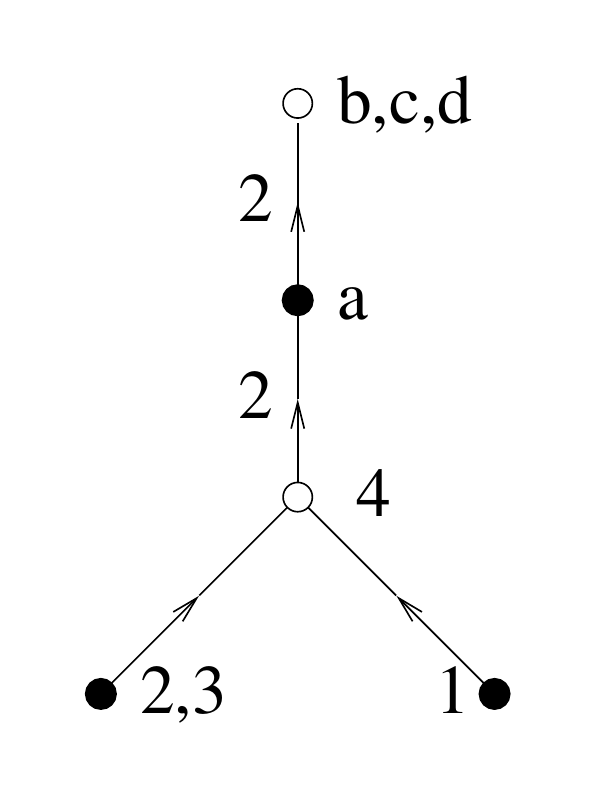}&
\includegraphics[width=3cm, angle=0]{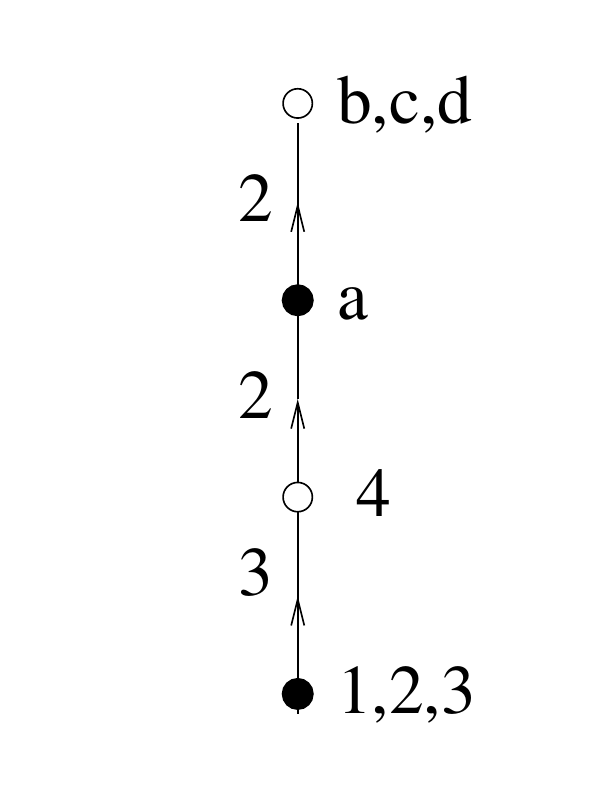}
\\\begin {tabular}{l} $a=5,6,7$, or 8 and $\mu=8$ \\ \hspace{1ex}  \end{tabular}  &
\begin {tabular}{l}$a=5$ and $\mu=12$\\  $a=6,7$, or 8 and
  $\mu=8$ \end{tabular} &
\begin {tabular}{l}$a=5$ and $\mu=54$\\  $a=6,7$, or 8 and
  $\mu=36$ \end{tabular} 
\end{tabular}

\begin{tabular}{cccc}

\\\includegraphics[width=3cm, angle=0]{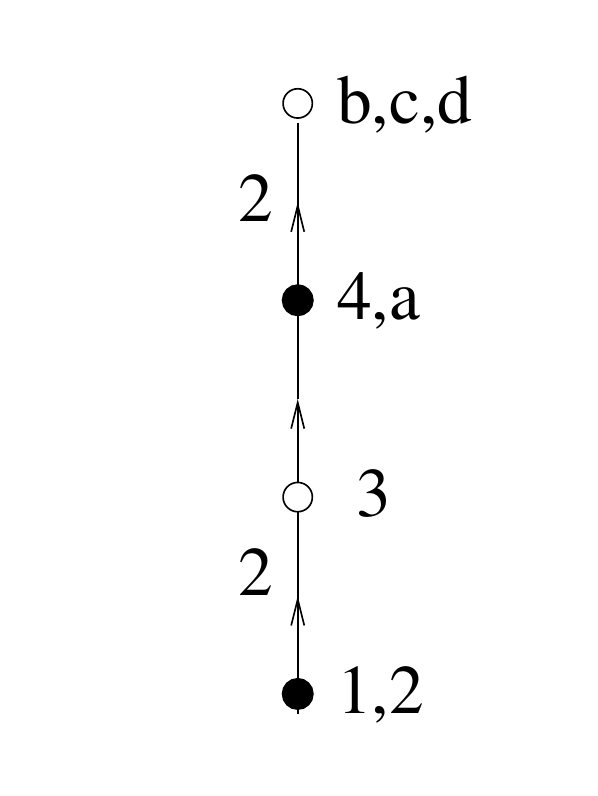}&
\includegraphics[width=3cm, angle=0]{Figures/FD3Me.pdf}&
\includegraphics[width=3cm, angle=0]{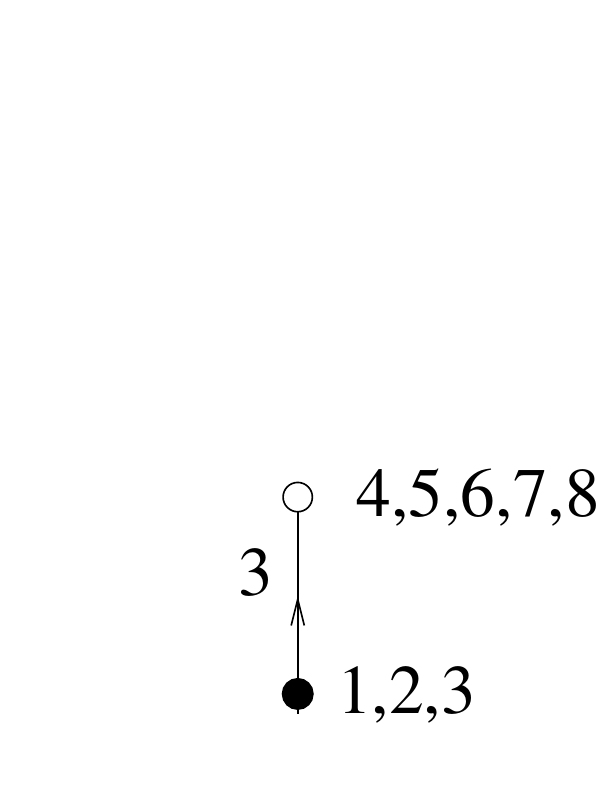}&
\includegraphics[width=3cm, angle=0]{Figures/FD3Mg.pdf}
\\\begin {tabular}{l} $a=5$ and $\mu=20$ \\ $a=6,7$, or 8 and
  $\mu=8$  \end{tabular}  &
\begin {tabular}{l} $\mu=2$\\ \hspace{2ex} \end{tabular} &
\begin {tabular}{l}$\mu=216$\\ \hspace{2ex} \end{tabular} &
\begin {tabular}{l}$\mu=108$\\ \hspace{2ex} \end{tabular}

\end{tabular}
\caption{Computation of  $N_{3,0}(1;1^{7})=600$ with $\L^{comb}=\{2,\ldots,8\}$}
\label{fd cubic}
\end{figure}

\begin{exa}
We give 
in  Figure \ref{fd cubic2} the sum of multiplicities of floor diagrams
of degree 3 with $\L^{comb}=\{1,2,3,4,6\}$, when this sum is positive.
In particular, we find that there are exactly 712 rational
cubics passing through 3 points 
and tangent to 5 lines.
\end{exa}
\begin{figure}[h]
\centering
\begin{tabular}{cccc}
\includegraphics[width=3cm, angle=0]{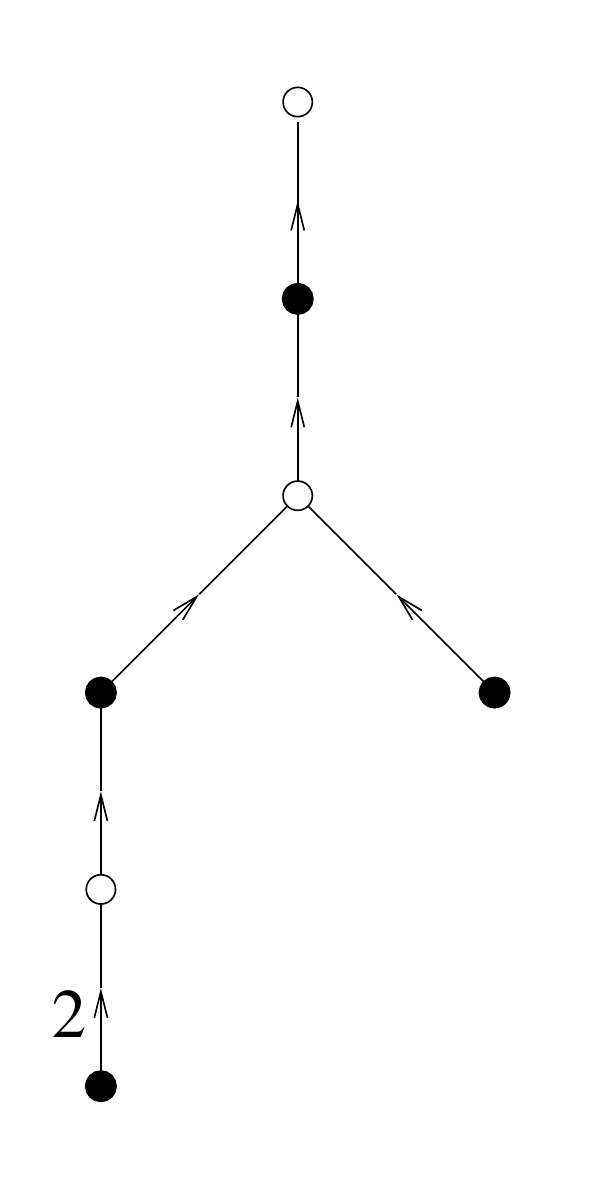}&
\includegraphics[width=3cm, angle=0]{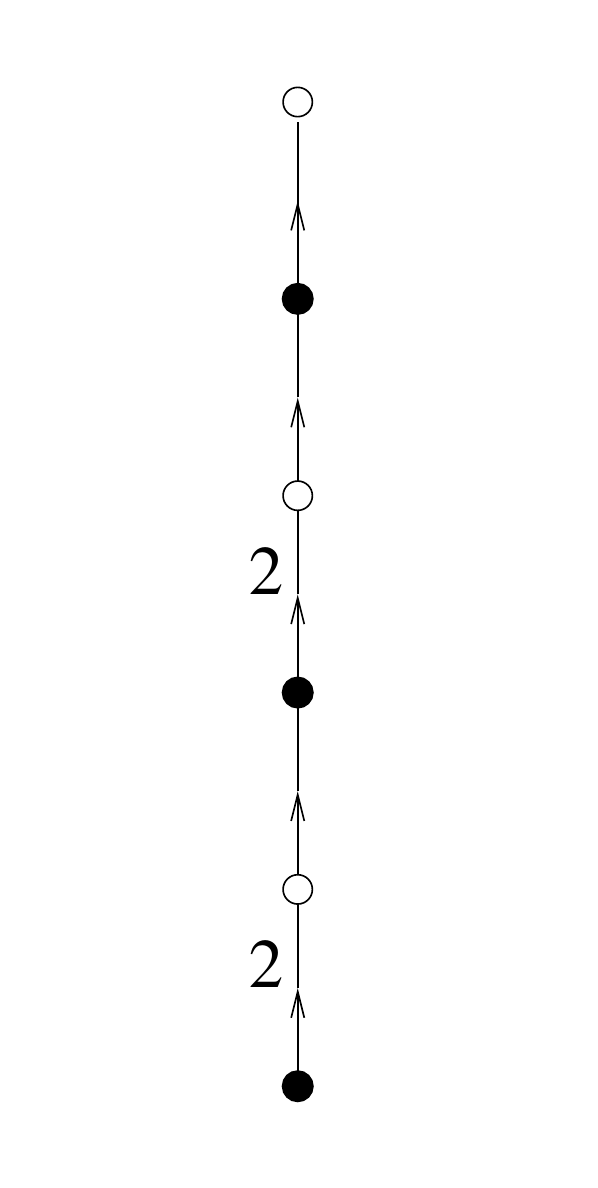}&
\includegraphics[width=3cm, angle=0]{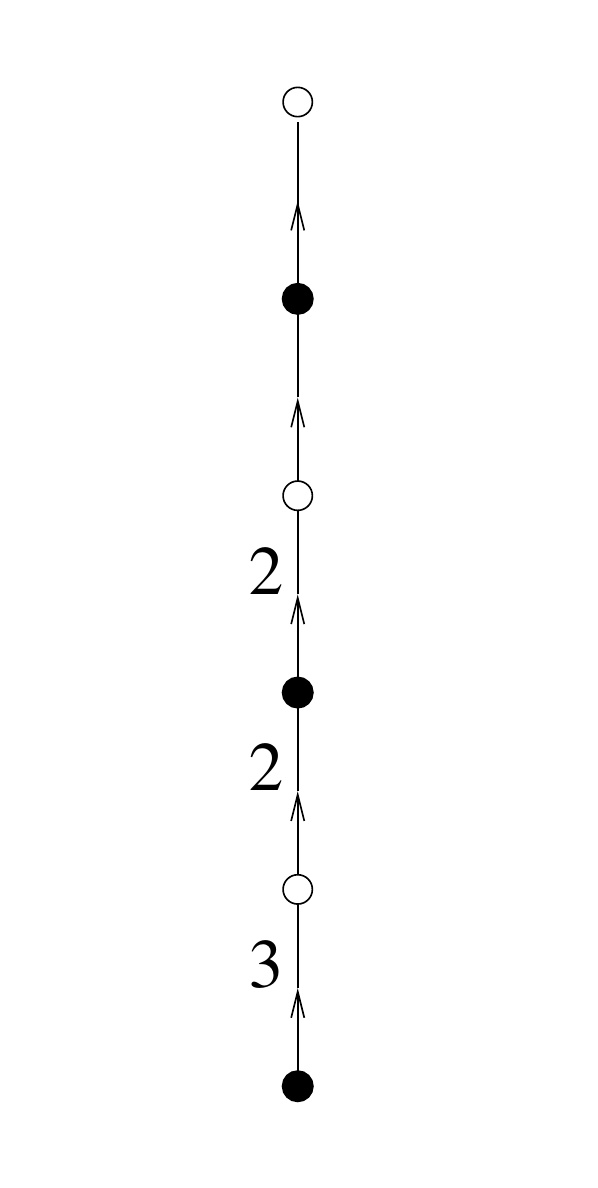}&
\includegraphics[width=3cm, angle=0]{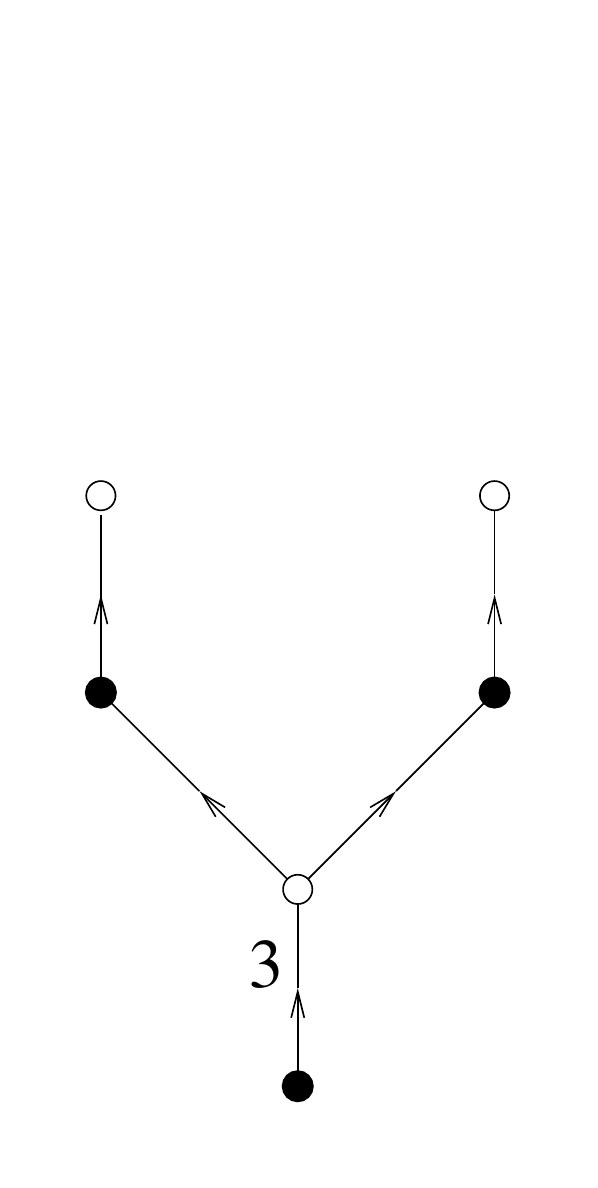}
\\ $\sum\mu=4 $& $\sum\mu=8 $ & $\sum\mu=72 $ & $\sum\mu=108 $
\end{tabular}

\begin{tabular}{cccc}

\\\includegraphics[width=3cm, angle=0]{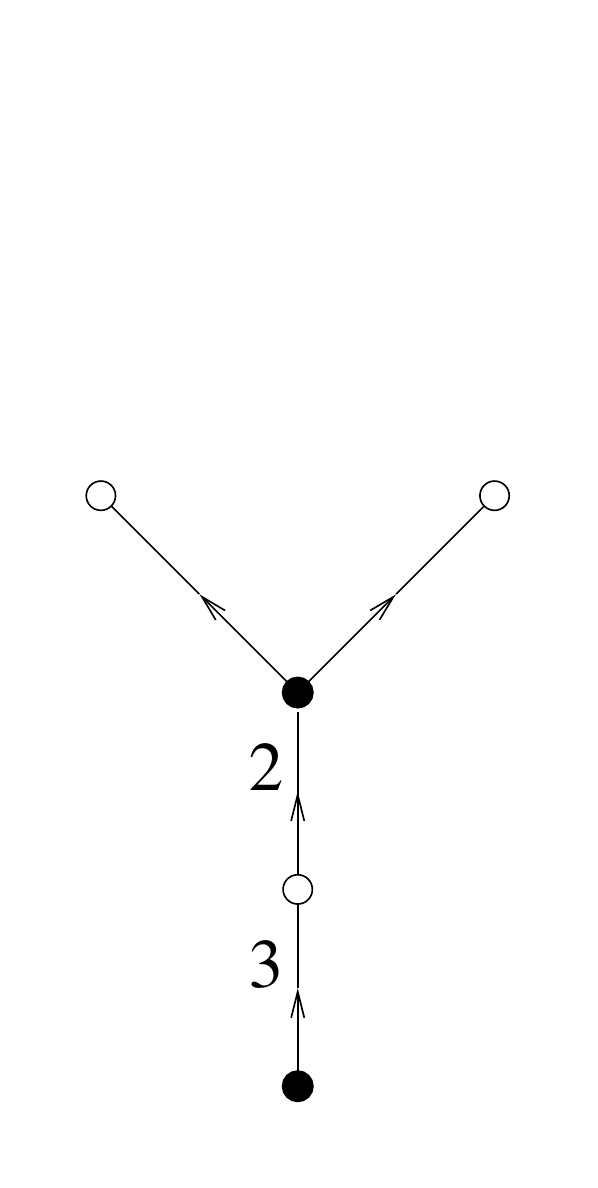}&
\includegraphics[width=3cm, angle=0]{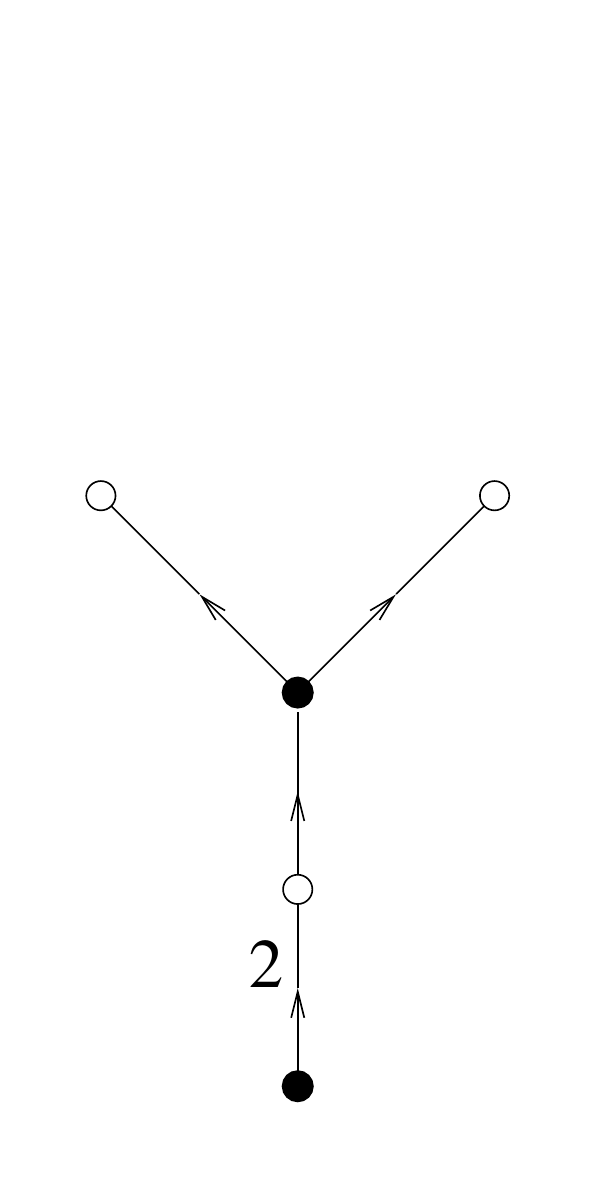}&
\includegraphics[width=3cm, angle=0]{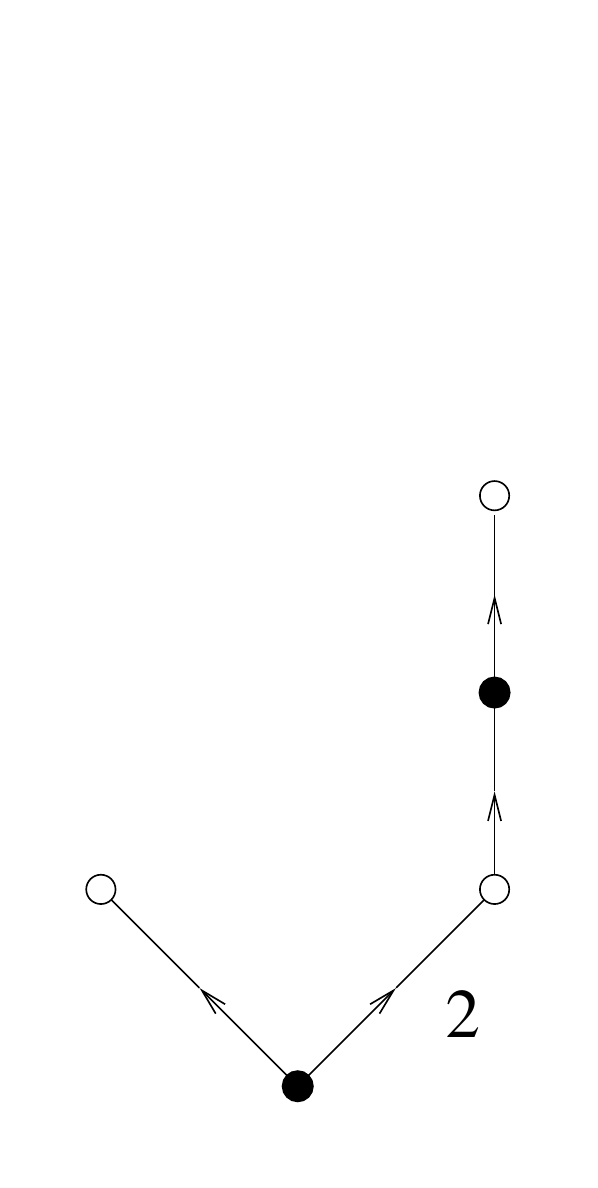}&
\includegraphics[width=3cm, angle=0]{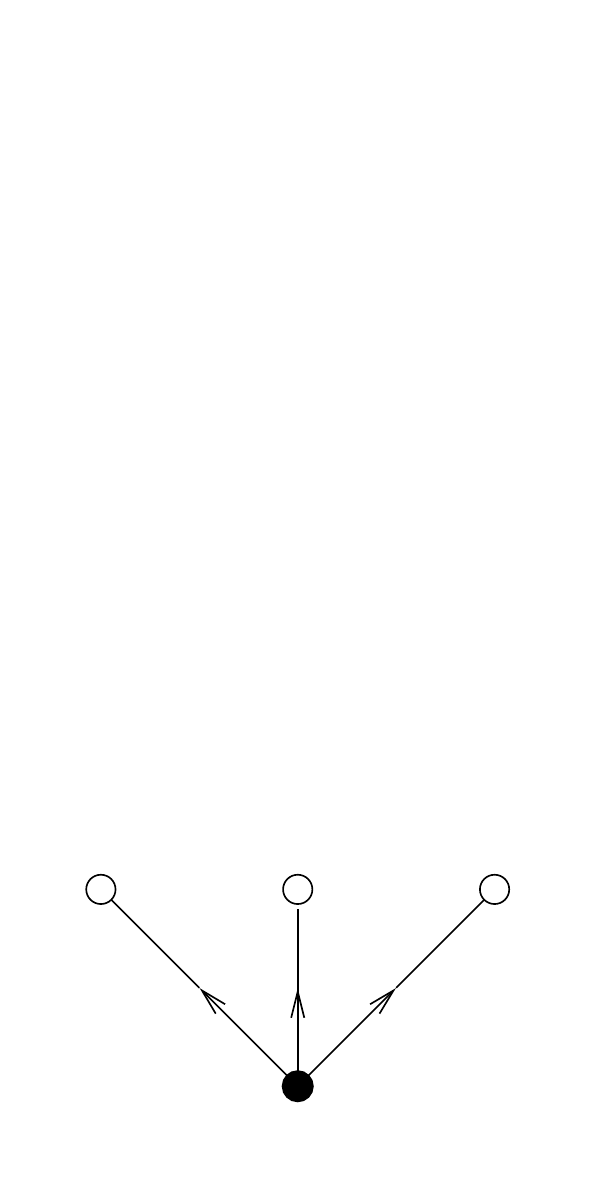}
\\ $\sum\mu=36 $& $\sum\mu=4 $ & $\sum\mu=400 $ & $\sum\mu=80 $

\end{tabular}
\caption{Computation of  $N_{3,0}(3;1^{5})=712$ with $\L^{comb}=\{1,2,3,4,6\}$}
\label{fd cubic2}
\end{figure}

\subsection{Proof of Theorem \ref{fd thm}}\label{proof fd}
The strategy of the proof of Theorem \ref{fd thm} is the same as in
the proof of \cite[Theorem 3.6]{Br6b}: we first prove that if 
points in $\P$, and vertices of elements of $\L$ lie
 in some strip $I\times\RR$, then all the vertices of a curve in
 $\S^\TT(d,\P,\L)$ tangent to $\L$ also lie in the strip
 $I\times\RR$. As a consequence, if the configuration
 $(\P,\L)$ is 
sufficiently stretched in the vertical direction, then all floors of a 
curve in
 $\S^\TT(d,\P,\L)$ tangent to $\L$  contain exactly 1 horizontal
constraint. It will then remain  to study how the constraints can
be distributed among the floors and the shafts of the tropical
curves we are counting.

The rest of this section is devoted to make precise the latter
explanations.

\vspace{2ex}
From now on, we fix $I=[a;b]\subset\RR$ a bounded interval, and a
generic configuration $(\P,\L)$ where $\P=\{p_1,\ldots,p_k\}$ and 
$\L=\{L_1,\ldots,L_{3d-1-k}\}$ is a set of $3d-1+k$ tropical lines in
$\RR^2$. 
 We  denote by $\eta_i$ the vertex of the line $L_i$. Recall that an
 element of $\S^\TT(d,\P,\L)$ with positive $(\P,\L)$-multiplicity is
 called tangent to $\L$.

\begin{lemma}\label{key lemma statement}
Suppose that $\P\bigcup\{\eta_1, \ldots,\eta_{3d-1-k}\}\subset I\times\RR$.
 Then for any tropical morphism  $f:C\to\RR^2$ in $\S^\TT(d,\P,\L)$
 tangent to $\L$, 
and for any vertex $v\in\Ve^0(C)$, we have $f(v)\in I\times\RR$.
\end{lemma}
\begin{proof}
The proof follows the same lines as the proof of \cite[Proposition
  5.3]{Br6b}. Suppose that there exists a tropical morphism
$f:C\to\RR^2$
 in $\S^\TT(d,\P,\L)$ and a vertex $v$ in $\Ve^0(C)$ such that
$f(v)=(x_v,y_v)$ with $x_v<a$. We can choose  $v$ such that no vertex of $C$
is mapped by $f$ to the half-plane $\{(x,y)\ | \ x<x_v\}$.
Since the configuration $(\P,\L)$ is generic, the set $\S^\TT(d,\P,\L)$
is finite and $v$ is a trivalent
vertex of $C$. Hence, the proof of 
\cite[Proposition 5.3]{Br6b} shows that $v$ is adjacent to an end $e$ of
direction $(-1,0)$, and two other edges of direction $(0,\pm 1)$ and
$(1,\alpha)$, and that one line $L_i$ has its horizontal edge passing
through $f(v)$ (see Figure \ref{key lemma}). Thus, $e\cup v$ is the
pretangency component of $f$ with $L_i$, and $\mu_{L_i}(f)=0$.

The case where there exist an element $f:C\to\RR^2$ in $\S^\TT(d,\P,\L)$
and a vertex $v$ in $\Ve^0(C)$ such that
$f(v)=(x_v,y_v)$ with $x_v>b$ works analogously.
\end{proof}
\begin{figure}[h]
\begin{center}
\begin{tabular}{c}
\includegraphics[height=3cm, angle=0]{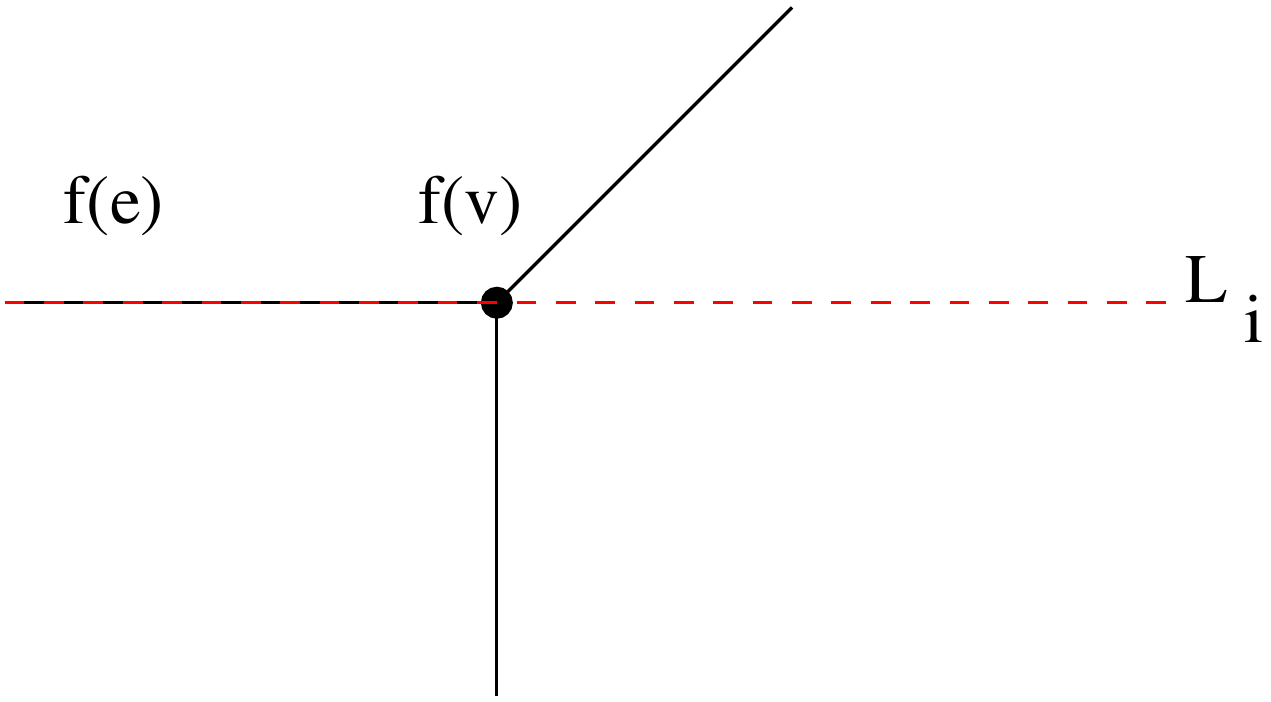}

\end{tabular}
\end{center}
\caption{No vertex $v\in\Ve^0(C)$ with $f(v)<a$ if  $f$ is tangent to $\L$}
\label{key lemma}
\end{figure}

Recall that pretangency sets are defined in Section \ref{sec:trop pretang}.
A line $L\in\L$ is called a \textit{vertical constraint}
(resp. a  \textit{horizontal constraint}) of a
morphism $f\in\S^\TT(d,\P,\L)$ if the pretangency set of $f$ and $L$
 is (resp. is not) contained in the
vertical edge of $L$. A point in $\P$ is at the same time a horizontal
and a vertical constraint of any morphism $f\in\S^\TT(d,\P,\L)$.
We say that a floor $\F$ (resp. a shaft $S$) of $f$ \emph{matches} the
constraint $q\in\P\cup\L$ if $\F$ (resp. $S$) contains $f^{-1}(q)$ or
contains the 
pretangency component of $f$ with $q$.

\begin{cor}\label{at most 1 horizontal}
If the points in 
 $\P\bigcup\{\eta_1, \ldots,\eta_{3d-1-k}\}$ are in $ I\times\RR$ and far
enough one from the others, 
then any floor of any
 tropical morphism  $f:C\to\RR^2$ in $\S^\TT(d,\P,\L)$ tangent to $\L$
 matches at most one horizontal constraint. 
\end{cor}
\begin{proof}
If $e$ is an edge of $C$ with $u_{f,e}\ne (0,\pm 1)$, then the slope of $f(e)$ 
is uniformly bounded in terms of the degree of $f$. Hence, the
result
 is an immediate consequence of Lemma \ref{key lemma statement}.
\end{proof}
\begin{defi}
We say that a generic configuration $(\P,\L)$ is \emph{vertically
  stretched} if it satisfies the
hypothesis of Corollary \ref{at most 1 horizontal}.
\end{defi}
To any morphism $f:C\to\RR^2$ in $\S^\TT(d,\P,\L)$, we can naturally associate a
floor diagram  $\D_f$ of degree
$d$ in the following 
way: white vertices of $\D_f$ correspond to floors of $C$, black
vertices correspond to shafts of $C$, and a white vertex and a black
vertex of $\D_f$ are joined by an edge of weight $w$ 
if and only if the
corresponding floor and shaft are joined by an elevator of weight $w$;  an
elevator of $C$ has an orientation inherited from the standard
orientation of the line $\{x=0\}$ in $\RR^2$, and
edges of $\D_f$
inherits  this  orientation as well. 
Given a shaft $S$ of $f$, we define $l(S)$ to be the sum of the
number of floors of $C$ adjacent to $S$ and the number of
ends of
$C$ contained in $S$. In other words, if $S$ corresponds to the black
vertex $v_S$ of $\D_f$, then $l(S)=\val(v_S)-\div(v_S)$.

\begin{cor}\label{exact incidence}
Let $(\P,\L)$ be a vertically
  stretched configuration of constraints and
let $f$ be an element of $\S^\TT(d,\P,\L)$ tangent to $\L$. Then
\begin{itemize}
\item any floor $\F$ of $f$ matches exactly $2\deg(\F)-1$ constraints,
 and exactly one of them is horizontal;

\item any shaft $S$ matches exactly $l(S)-1$ constraints,
 and exactly one of them is vertical.

\end{itemize}

\end{cor}
\begin{proof}
A floor $\F$ and a shaft $S$ of $f$ have respectively exactly 
$2\deg(\F)-2$ and $l(S)-2$ vertices,
since $\pi_x\circ f_{|\F}:\F\to\RR$ and $\pi_y\circ f_{|S}:S\to\RR$
    are tropical (open) ramified coverings with respectively
    $2\deg(\F)$ and $l(S)$ ends. 
According to Corollary \ref{at most 1
horizontal} a floor $\F$ matches
 at most one horizontal constraint so $\F$ matches
at most $2\deg(\F)-1$ constraints. Since the configuration
$(\P,\L)$ is generic, a shaft $S$ matches
at most one vertical constraint, and hence at most $l(S)-1$
constraints. All together, we get the inequality
$$\sum_\F (2\deg(\F)-1) + \sum_S (l(S)-1)\le 3d-1.$$
We have
$$\sum_\F (2\deg(\F)-1) + \sum_S (l(S)-1)=2d -|\Ve^\bullet(\D_f)| -
|\Ve^\circ(\D_f)| +\sum_{v\in\Ve^\bullet(\D_f)}\val(v) +d $$
and the fact that $\D_f$ is a bipartite tree gives us 
$$|\Ve^\bullet(\D_f)| +
|\Ve^\circ(\D_f)| -\sum_{v\in\Ve^\bullet(\D_f)}\val(v) =|\Ve^\bullet(\D_f)| +
|\Ve^\circ(\D_f)| -|\Ed(\D_f)|=1$$
 so we obtain
$$\sum_\F (2\deg(\F)-1) + \sum_S (l(S)-1)=3d-1.$$
Hence all inequalities above are in fact equalities.
\end{proof}

The set $\P\cup\{\eta_1,\ldots, \eta_{3d-k-1}\}$ inherits a total order from
the map $\pi_y$, and we relabel elements of  this set by 
$\overline q_1 <\ldots  < \overline q_{3d-1}$.
 We define
the set $\L^{comb}\subset\{1,\ldots, 3d-1\}$ by
$$i\in \L^{comb}\Longleftrightarrow \overline q_i\in\L.$$
Given a morphism $f:C\to\RR^2$ in $\S^\TT(d,\P,\L)$ we define a map 
$m_f: \{1,\ldots, 3d-1\}\to\D_f$ by
$$m_f(i)=v\Longleftrightarrow \text{the floor or shaft of $C$
  corresponding to $v$ matches }\overline q_i.$$

\begin{lemma}
Let $(\P,\L)$ be a vertically
  stretched configuration of constraints.
If $f$ is tangent to $\L$, then $(\D_f,m_f)$ is an $\L^{comb}$-marked
floor diagram of degree $d$.
\end{lemma}
\begin{proof}
This is an immediate consequence of Corollary \ref{exact incidence}
and the genericity of the configuration $(\P,\L)$.
\end{proof}

Hence we have constructed a map $\phi_{(\P,\L)}$ which maps an element 
of $f\in \S^\TT(d,\P,\L)$ tangent to $\L$ to  the $\L^{comb}$-marked
floor diagram $(\D_f,m_f)$.
Theorem \ref{fd thm} is a direct consequence of Theorem \ref{Corres} and of 
the following proposition.
\begin{prop}\label{fiber}
Let $(\P,\L)$ be a vertically
  stretched configuration of constraints.
For any $\L^{comb}$ marked floor diagram $(\D,m)$  of degree $d$, we have
$$\sum_{f\in\phi_{(\P,\L)}^{-1}(\D,m)}\mu_{(\P,\L)}(f)=\mu_{\L^{comb}}(\D,m).$$
\end{prop}
\begin{proof}
Let $(\D,m)$ be a $\L^{comb}$-marked floor diagram. We start with two
easy
 observations. Given a tropical morphism
 $f\in\phi_{(\P,\L)}^{-1}(\D,m)$ and
 $v\in\Ve^\circ(\D)$ corresponding to the floor $\F_v$ of $f$, 
then the horizontal constraint matched by $\F_v$ 
is   $\overline q_{\min(m^{-1}(v))}$.
If $v\in\Ve^\bullet(\D)$ corresponds to the shaft $S_v$ of $f$, 
then the vertical constraint matched by $S_v$ is
\begin{enumerate}
\item the point $\overline q_{i_v}$ in $\P$ where
  $i_v$ is  the unique element of $m^{-1}(v)\cap\P^{comb}$ 
  if this  set is non-empty;

\item the line $\overline q_{i_v}$ in $\L$ 
where  
  $i_v=\max(m^{-1}(v))$ if  
$i_v>\max_{v'\in\Ve(v)}\min((m^{-1}(v'))$;

\item a line $\overline q_{i_v}$ where  $i_v\in m^{-1}(v)$ otherwise.
\end{enumerate}
Hence, we see that the horizontal constraint
matched by $S_v$  is not 
determined by $\D$  only in the case (3) above.
We denote by $\Ve^{\bullet,3}(\D)$ the subset of  
$\Ve^{\bullet}(\D)$ composed of vertices in this case.

Given $\varepsilon >0$ and $p\in \P$ (resp. $L\in\L$ with vertex
$\eta$), 
we denote by $I_{\varepsilon,p}$ (resp. $I_{\varepsilon,L}$) the
interval of $\RR$
  centered in $\pi_y(p)$ (resp. $\pi_y(\eta)$) and of length
  $\varepsilon$. Let us denote by $\overline q'_1<\ldots <\overline
  q'_{\val(v)}$ the elements of $\P\cup\L$ which are horizontal
  constraints for a floor of $f$ adjacent to the shaft $S_v$, and by 
$x_i$ is the center of the 
interval $I_{\varepsilon,\overline q'_i}$. We define the set $\Q^\TT$
to be the set of centers of all
intervals $I_{\varepsilon,L}$ with $L$ a non vertical constraint of
$S_v$ matched by $S_v$.
Now we can consider the set
$\H^\TT(\delta, n_{i_v})$ corresponding to the $x_i$'s and the set
$\Q^\TT$
  (see Section \ref{open Hurwitz} 
for the definition of
$\H^\TT$, and Definition \ref{def mult black} 
for the definition of $\delta$ and
$n_{i_v}$).
Since the points of
$\P\cup\{\eta_1,\ldots,\eta_{3d-1-k}\}$ are contained in the strip
$I\times\RR$ and very far one from the others,  Lemma \ref{key lemma
  statement} tells us that there exists
$\varepsilon>0$ depending only on $a$,$b$, and $d$ (in particular independent
on $f$, $\P$, and $\L$), such that 
\begin{itemize}
\item if $v'$ is a vertex of $S_v$ such that
 $f(v')\in L\in\L$ with $L\ne \overline q_{i_v}$, then
 $\pi_y\circ f(v')\in I_{\varepsilon,L}$;

\item if $v'$ is a vertex of $S_v$ which is also a vertex of
 a floor matching
  the horizontal constraint $q$, then $\pi_y\circ f(v')\in I_{\varepsilon,q}$;

\item if $q\ne q'$ then $I_{\varepsilon,q}\cap I_{\varepsilon,q'}=\emptyset$.
\end{itemize}
Hence, there exists a unique element $g$ of $\H^\TT(\delta, n_{i_v})$
which can be obtained from a deformation $(g_t:C'_t \to \RR^2)_{t\in[0;1]}$ 
of the restriction $f_{|S_v}$  in its deformation space  such that
$g_t(\Ve(C'_t))\subset\cup_{q\in\P\cup \L} I_{\varepsilon,q}$ for all $t$.
 Moreover when $f$ ranges over all elements of 
$\phi_{(\P,\L)}^{-1}(\D,m)$ with fixed $i_v$,
there is a natural bijection between all possible
restrictions $f_{|S_v}$ and the set $\H^\TT(\delta, n_{i_v})$.

Once again, since the points of
$\P\cup\{\eta_1,\ldots,\eta_{3d-1-k}\}$ are contained in the strip
$I\times\RR$ and very far one from the others, to construct a
morphism $f$ in $\phi_{(\P,\L)}^{-1}(\D,m)$, 
it is enough to construct independently the restriction of 
$f$ on the floors and on
the shafts of $C$, and to glue all these pieces together along elevators.
It follows from Definition \ref{def mult black}, Proposition
 \ref{practical comp}, and Theorem \ref{Hurwitz count},
 that  the contribution
of a vertex $v\in\Ve^\bullet(\D)\setminus \Ve^{\bullet,3}(\D)$
 is equal to $\mu_{\L^{comb}}(v)$. 
In the case of a vertex $v\in\Ve^{\bullet,3}(\D)$,
 an easy Euler characteristic computation and formula~(\ref{wL}) for the weight associated to a tangency show that the contribution of
$v$ with a fixed $i_v$ is, in the notation of Definition \ref{def
  mult black}, 
$$  r_{i,v}H(\delta,n_{i_v})$$
where $r_{i_v}=(\delta(i_v)+i+ |\{i\in m^{-1}(v)\ | \ i<i_v\}| +\div(v))$.
Hence, according to Lemma \ref{lemma floor},
the sum of the
$(\P,\L)$-multiplicity of all morphisms $f\in
\phi_{(\P,\L)}^{-1}(\D,m)$ with fixed $i_v$ for all $v\in\Ve^{\bullet,3}(\D)$
 is exactly 
$$\prod_{e\in\Ed(\D)}w(e) \prod_{v\in\Ve(\D)\setminus \Ve^{\bullet,3}(\D)}\mu_{\L^{comb}}(v) 
\prod_{v\in\Ve^{\bullet,3}(\D)}
\Big(r_{i_v}
H(\delta,n_{i_v}) \Big). $$
Given $v\in\Ve^{\bullet,3}(\D)$,  we have
$$\begin{array}{lll}
\sum_{i_v} 
\Big(r_{i_v}H(\delta,n_{i_v}) \Big) &=&  \sum_{i=0}^s
\left(\left(\tilde n(i)(\delta(i)+i +\div(v)) + 
\sum_{j=\tilde N(i-1)}^{\tilde N(i)-1}j \right)H(\delta,n_{i}) \right)
\\ &=& \sum_{i=0}^s
\left(\left(\tilde n(i)(\delta(i)+i +\div(v)) + 
\frac{\tilde N(i)^2-\tilde N(i-1)^2 -(\tilde N(i)-\tilde N(i-1)) }{2}
\right)H(\delta,n_{i}) \right)
\\ &=& \mu_{\L^{comb}}(v) 
\end{array}.$$
Hence we have
$$\sum_{f\in\phi_{(\P,\L)}^{-1}(\D,m)}\mu_{(\P,\L)}(f)
=\prod_{e\in\Ed(\D)}w(e) \prod_{v\in\Ve(\D)}\mu_{\L^{comb}}(v) 
= \mu_{\L^{comb}}(\D,m) $$
as desired.
\end{proof}

\begin{lemma}\label{lemma floor}
Let $d_v\ge 1$ and $w_1,\ldots,w_l>0$ be integer numbers.
Choose 
a generic configuration $(\P,\L)$ of constraints such that
$\P=\{p_1,\ldots, p_l\}\subset I\times\RR$, and $\L=\{L_1,\ldots,
  L_{2d_v-2}\}\subset I\times\RR$ is a set of vertical lines.
Choose also a point $p_0\in I\times\RR$ and
a tropical line
 $L_0$ whose vertex $\eta_0$ is in $I\times\RR$ such that the
configuration $(\P\cup\{p_0\},\L\cup\{L_0\})$ is generic. Suppose that 
the points $p_i$ and $\eta_0$ are very far one from the others, and denote by
  $\C_{p_0}(2d_v+l,\P,\L)$ (resp.  $\C_{L_0}(2d_v+l,\P,\L)$)
the set of all minimal tropical morphisms $f:C\to \RR^2$ such that
\begin{enumerate}
\item $C$ has  $l$ ends of direction $(0,\pm 1)$; given  such an end
  $e$ of $C$, $f(e)$ passes through one point $p_i$ of $\P$, 
and $w_{f,e}=w_i$;
\item $C$ has
$d_v$ ends of direction $(-1,0)$ and weight 1,
 and $d_v$ ends of direction $(1,1)$ and weight 1;

\item $f$ is pretangent to all lines in $\L$;

\item $f$  passes through $p_0$ (resp.  is pretangent to
  $L_0$).

\end{enumerate}
Then
$$\sum_{f\in\C_{p_0}(2d_v+l,\P,\L)}\mu_{(\P\cup{p_0},\L)}(f)=
d_v^{l+1}H(d_v)\prod_{e\in\Ed^\infty(C),\ u_{f,e}=(0,\pm 1)}w(e) $$
and 
$$\sum_{f\in\C_{L_0}(2d_v+l,\P,\L)}\mu_{(\P,\L\cup{L_0})}(f)=
(l-2+d_v)d_v^{l}H(d_v)\prod_{e\in\Ed^\infty(C),\ u_{f,e}=(0,\pm
  1)}w(e). $$
Moreover,  any tropical morphism $f\in \C_{p_0}(2d_v+l,\P,\L)$
(resp. $f\in \C_{L_0}(2d_v+l,\P,\L)$) has exactly one floor, and
this floor matches $p_0$ (resp. $L_0$). 
\end{lemma}
\begin{proof}
The fact that $f$ has only one floor is straightforward: $C$ has
$2d_v+l-2$ vertices, exactly $l$ of which are adjacent to a vertical
end of $C$, so $C$ has exactly 1 vertex mapped to each line
$L_i$, $i> 0$; in particular, it has no other vertical edge than its
vertical ends. Since the floor of $f$ has to match a horizontal
constraint, it matches necessarily $p_0$ or $L_0$.

In the following, we use notations of Appendix \ref{open Hurwitz}.
Let us denote by $\Q^\TT$  the set of
intersection points of all tropical lines $L_i$ with $\{y=0\}$ when $i>0$, and
let us consider the set
$\H^\TT(\delta,n)$ with $s=0$, $\delta(0)=d_v$, and $n(0)=2d_v-2$.
Let us
 fix an element $f_0:C_0\to \{y=0\}=\RR$ in $\H^\TT(\delta,n)$.
Consider a sequence of $l$ open tropical modifications $\pi:C_1\to
C_0$, and 
a minimal tropical morphism $f_1:C_1\to \RR^2$ 
 satisfying conditions (1)-(2)-(3) such that $\pi_x\circ
 f_1=f_0\circ\pi$. 
Composing $f_1$ with 
a translation in the $(0,1)$ direction, we construct a finite number
of elements of  $\C_{p_0}(2d_v+l,\P,\L)$ and  $\C_{L_0}(2d_v+l,\P,\L)$.

\begin{itemize}
\item To construct an element of $\C_{p_0}(2d_v+l,\P,\L)$,
 we have to make one of the
edges of $C_1$ 
 pass through $p_0$. Once
this is done, the orientation of the curve $C_1$ defined in
Section \ref{practical} is as follows: the rays emanating from $p_0$
are oriented away from $p_0$, and exactly one edge  $e$ with
$u_{f_1,e}\ne (0,\pm 1)$ is
  oriented toward $v_1$ at a vertex $v_1$
 of $C_1$.
 Hence the multiplicity of such a vertex $v_1$ is 
$\mu(v_1)=|\alpha|$ if $u_{f_1,e}= (\alpha,\beta)$.
The multiplicity can then be computed via Proposition~\ref{practical comp}. 
Notice that for a given choice of morphism $f_0$, and every thing being fixed but for one vertical edge (resp. the edge containing the marked point), summing the corresponding $\mu(v_1)$ (resp. weights of the supporting edges) over all the possible choices, one gets a factor $d_v$. Thus adding the $(\P\cup{p_0},\L)$-multiplicity of all possible
morphisms $f_1$ constructed in this way starting with a fixed $f_0$, we obtain
$d_v^{l+1}\mu_H(f_0)$.
 Considering all possible tropical
    morphisms $f_0\in \H^\TT(\delta,n)$, Theorem \ref{Hurwitz count}
    tells us that we obtain 
$$\sum_{f\in\C_{p_0}(2d_v+l,\P,\L)}\mu_{(\P\cup{p_0},\L)}(f)=
d_v^{l+1}H(d_v)\prod_{e\in\Ed^\infty(C),\ u_{f,e}=(0,\pm 1)}w(e).$$

\item To construct an element of $\C_{L_0}(2d_v+l,\P,\L)$,
we have 2 possibilities: either
$f_1(C_1)$ passes through the vertex $\eta_0$ of $L_0$, 
or a vertex of $C_1$ is mapped to
$L_0$. 
Once
this is done, the orientation of the curve $C_1$ defined in
Section \ref{practical} is as follows ($E_0$ denotes the union of all
 pretangency
components of $f_1$ with $L_0$):
an edge $e$ with $u_{f_1,e}\ne (0,\pm 1)$ intersecting $E_0$ but
not included in $E_0$ is oriented away from $E_0$, and 
exactly one edge  $e$ with
$u_{f_1,e}\ne (0,\pm 1)$ is
  oriented toward $v_1$ at a vertex $v_1$
 of $C_1\setminus E_0$.
Note that if there exists an end of $C_1$ with direction $(-1,0)$ or
$(1,1)$ and intersecting $E_0$ in infinitely many points, 
then $\mu_{(\P,\L\cup{L_0})}(f_1)=0$.
Hence,  adding the $(\P,\L\cup{L_0})$-multiplicity of all possible
morphisms $f_1$ constructed in this way starting with a fixed $f_0$, 
we obtain 
$d_v^{l}K\mu_H(f_0)$, where $K=d_v+ |\Ve^0(C_1)|-2d_v=l-2+d_v$. Considering all possible tropical
    morphisms $f_0\in \H^\TT(\delta,n)$,
 Theorem \ref{Hurwitz count} tells us that we obtain
$$\sum_{f\in\C_{L_0}(2d_v+l,\P,\L)}\mu_{(\P,\L\cup{L_0})}(f)=
(l-2+d_v)d_v^{l}H(d_v)\prod_{e\in\Ed^\infty(C),\ u_{f,e}=(0,\pm
  1)}w(e). $$

\end{itemize}
Hence the lemma is proved.

\end{proof}

\appendix
\section{Open Hurwitz numbers}\label{open Hurwitz}
We recall here the definition of open Hurwitz numbers. These
numbers were introduced in \cite{Br13} and are a slight generalization
of well known Hurwitz numbers. For simplicity, we restrict ourselves 
 to the special
cases we need in this paper. We refer to \cite{Br13} for more details
and examples 
about open Hurwitz numbers and their tropical counterpart (see also \cite{CJM}).

Let $s\ge 0$ be an integer number, and 
$\delta, \ n:\{0,\ldots,s\}\to \ZZ_{\ge 0}$ be two functions.
Choose a collection of $s$ embedded circles $c_1,\ldots, c_s$
in 
the sphere
$S^2$ such that $c_1$  (resp. $c_s$) bounds a disk $D_0$
(resp. $D_{s}$), 
and $c_i$ and
$c_{i+1}$ bound an annulus $D_{i}$ for $1\le i\le s-1$.
Choose also a collection $\Q$ of points in $S^2\setminus \cup_{i=1}^s
c_i$, such that each $D_i$ contains exactly $n(i)$ points of $\Q$.
Let us  consider the set $\H(\delta,n)$ of all
equivalence class of ramified coverings $f:\Sigma\to
S^2$ where

\begin{itemize}
\item $\Sigma$ is a  connected compact oriented surface of genus 0 with $s$ boundary
  components;

\item $f(\partial \Sigma)\subset \cup_{i=1}^sc_i$;

\item $f$ is unramified over $S^2\setminus \Q$;
 
\item  $f_{|f^{-1}(D_i)}$ has degree $\delta(i)$ for each $i$;

\item  each point in $\Q$ is a simple critical value of $f$;

\item for each circle $c_i$, the set $f^{-1}(c_i)$  contains exactly one
  connected component $c$ of $\partial \Sigma$, and
$f_{|c}:c\to c_i$  is an
 unramified covering of degree $|\delta(i)-\delta(i-1)|$.

\end{itemize}
Two continuous maps   $f : \Sigma\to S^2$ and $f':\Sigma'
\to S^2$ are considered  equivalent if  there exists
a homeomorphism $\Phi:\Sigma\to \Sigma'$  such that
 $f'\circ \Phi= f$. 

Note that the cardinal of the set $\Q$ is prescribed by the 
Riemann-Hurwitz formula
$$|\Q|=\delta(0)+\delta(s) +s -2.$$

\begin{defi}\label{def open hurwitz}
The \emph{open Hurwitz number} $H(\delta,n)$ is defined as
$$H(\delta,n)=\sum_{f\in  \H(\delta,n)}\frac{1}{|Aut(f)|}.  $$
\end{defi}

Note that we can naturally extend the definition of the numbers
$H(\delta,n)$ to the case where $\delta:\{0,\ldots,s\}\to\ZZ$ is any
function by
setting
$$H(\delta,n)=0\quad \text{if}\quad \Im\delta\nsubseteq \ZZ_{\ge
0}. $$

\begin{exa}\label{elem1}
If $s=2$, $\delta(0)=\delta(2)=n(0)=n(1)=n(2)=0$, and
$\delta(1)=d$, 
one computes easily
$$H(\delta,n)=\frac{1}{d}.$$
\end{exa}

In the special case where $s=0$,
  we recover usual Hurwitz numbers. In particular $\delta$ is
 just a positive integer number, the degree of the maps  we are
 counting and that  we just denote by $d$. We simply denote this
 Hurwitz number 
by $H(d)$.

\begin{prop}[Hurwitz]\label{exa hurwitz}
For any $d\ge 1$, then
$$H(d)=\frac{d^{d-3}(2d-2)!}{d!}.  $$
\end{prop}

\vspace{2ex}
Let us now define tropical open Hurwitz numbers let us relate them to
the open Hurwitz numbers we just defined. 
Recall that we have chosen $s\ge 0$  an integer number, and 
$\delta, \ n:\{0,\ldots,s\}\to \ZZ_{\ge 0}$  two functions.

Choose a collection of $s$ points $x_1< \ldots < x_s$
in $\RR$, and define  $D^\TT_0=(-\infty;x_1)$,
$D^\TT_i=(x_i,x_{i+1})$, and  $D^\TT_s=(x_s;-\infty)$.
Choose another collection of points
 $\Q^\TT$ in $\RR\setminus \{x_1,\ldots, x_s\}$, 
such that each $D^\TT_i$ contains exactly $n(i)$ points of $\Q^\TT$.
Let us  consider the set $\H^\TT(\delta,n)$ of all
equivalence class of minimal tropical morphisms $f:C\to\RR$ where

\begin{itemize}
\item $C$ is a  rational tropical curve with $s$ boundary
  components;

\item $f(\partial C)\subset\{x_1,\ldots,x_s\}$;

\item $f(\Ve^0(C))\subset \Q^\TT$;
 
\item  given $0\le i\le s$ and $p\in\D^\TT_i$, if we denote by $e_1,\ldots,e_r$
  the edges of $C$ which contain a point of $f^{-1}(p)$, then we have
  $\sum_{j=1}^r w_{f,e_j}=\delta(i)$;

\item  for each $p\in\Q^\TT$, $f^{-1}(p)$ contains exactly one element of
  $\Ve^0(C)$, which is a 3-valent vertex of $C$;

\item for each $1\le i\le s$,  the set $f^{-1}(x_i)$  contains exactly one
  boundary component of $C$, adjacent to an edge of weight
 $|\delta(i)-\delta(i+1)|$;

\item each end of $C$ is of weight 1.

\end{itemize}
As previously,  two tropical morphisms   
$f :C\to \RR$ and $f':C':\to \RR$ are considered equivalent 
 if  there exists
a tropical isomorphism $\Phi:C\to C'$ such that
 $f'\circ \Phi=\phi\circ f$.

Once again, we have
$$|\Q^\TT|=\delta(0)+\delta(s) +s -2.$$

Given $v$ a puncture of $C$  adjacent to the end $e$, we set
  $w_{f,v}=w_{f,e}$, and we define $w_{f,\infty}$ as the product of
the weights $w_{f,v}$ when $v$ ranges over all punctures of $C$.
The multiplicity of an element of $\H^\TT(\delta,n)$ is then defined as
$$\mu_H(f)=\frac{\prod_{e\in\Ed(C)}w_{f,e}}{w_{f,\infty} }. $$ 

\begin{defi}
The \emph{tropical open Hurwitz number} $H^\TT(\delta,n)$ is defined as
$$H^\TT(\delta,n)=\sum_{f\in  \H^\TT(\delta,n)}\frac{1}{|Aut(f)|}\mu_H(f).  $$
\end{defi}

As previously we  extend the definition of the numbers
$H^\TT(\delta,n)$  by
setting
$$H^\TT(\delta,n)=0\quad \text{if}\quad \Im\delta\nsubseteq \ZZ_{\ge 0}. $$

\begin{thm}[{\cite[Theorem 2.11]{Br13}}]\label{Hurwitz count}
For any two functions $\delta:\{0,\ldots, s\}\to \ZZ$ and 
$n:\{0,\ldots, s\}\to \ZZ_{\ge 0}$, we have 
$$H(\delta,n)=H^\TT(\delta,n).$$
\end{thm}

\begin{exa}\label{ex comp open H}
Using Theorem \ref{Hurwitz count} we compute easily the following open
Hurwitz numbers:

\noindent if $s=1$, $\delta(0)=2$, $\delta(1)=0$, $n(0)=1$, $n(1)=0$,
then $H(\delta,n)=\frac{1}{2}$.

\noindent if $s=2$, $\delta(0)=1$, $\delta(1)=2$,  $\delta(2)=0$,
$n(0)=n(2)=0$, $n(1)=1$,
then $H(\delta,n)=1$.

\noindent if $s=1$, $\delta(0)=3$, $\delta(1)=0$, 
$n(0)=2$, $n(1)=0$,
then $H(\delta,n)=1$.

\end{exa}

\small
\def\rightmark{\em Bibliography}

\bibliographystyle{alpha}
\bibliography{Biblio.bib}

\end{document}